\documentclass[11pt, a4paper]{article}

\usepackage{amssymb,amsmath, wasysym, graphicx, epsfig, setspace, wrapfig, comment}
\usepackage{latexsym,dsfont,amsfonts,amsmath,amssymb}
\usepackage{pstricks,epsfig,psfrag,color}
\usepackage[show]{ed}
\usepackage{mathrsfs}
\usepackage{soul}
\usepackage{mathrsfs}   
\usepackage{tikz}
\usepackage[breaklinks=true]{hyperref}
\usepackage{apptools}
\usepackage{amsthm}

\usepackage{mathtools}

\newcommand{\Ash}[1]{\calA^\mathrm{sh}_{#1}}
\newcommand{\pd}[3]{\frac{\partial ^{#1} #2}{\partial #3}}

\newcommand{\bs}[1]{\boldsymbol{#1}}
\newcommand{\nrm}[2]{\ensuremath{\|#1\|_{#2}}}

\newcommand{\pdy}[1]{\partial^{#1}_{\bsy}}
\newcommand{\dist}{\mathrm{dist}}

\newcommand{\diam}{\mathrm{diam}}

\newcommand{\Df}{\ensuremath{D_\mathrm{f}}}
\newcommand{\Dff}{\ensuremath{D_{\mathrm{f}, 1}}}
\newcommand{\N}[0]{\mathbb{N}}

\newcommand{\R}[0]{\mathbb{R}}

\newlength\figureheight
\newlength\figurewidth

\makeatletter
\newcommand{\BIG}{\bBigg@{3}}
\newcommand{\vast}{\bBigg@{4}}
\newcommand{\Vast}{\bBigg@{5}}
\makeatother

\usepackage[normalem]{ulem}

\newtheorem{theorem}{Theorem}
\newenvironment{namedproof}[1]{\begin{trivlist}
    \item[\hskip\labelsep{\textit{Proof of #1}}]}{$\hfill\Box$\end{trivlist}}

\newtheorem{assumption}{Assumption A$\!\!$}

\numberwithin{theorem}{section}
\newtheorem{Lemma}[theorem]{Lemma}
\newtheorem{Proposition}[theorem]{Proposition}
\newtheorem{Remark}[theorem]{Remark}

\newcommand{\sigdiff}[0]{\ensuremath{\sigma_\mathrm{diff}}}
\newcommand{\sigabs}[0]{\ensuremath{\sigma_\mathrm{abs}}}
\newcommand{\sigfiss}[0]{\ensuremath{\sigma_\mathrm{fiss}}}

\newcommand{\sumj}{\sum_{j=1}^\infty}

\newcommand{\satop}[2]{\stackrel{\scriptstyle{#1}}{\scriptstyle{#2}}}

\newcommand{\bsbeta}{{\boldsymbol{\beta}}}
\newcommand{\bsDelta}{{\boldsymbol{\Delta}}}

\newcommand{\bse}{{\boldsymbol{e}}}

\newcommand{\bsalpha}{{\boldsymbol{\alpha}}}
\newcommand{\bsgamma}{{\boldsymbol{\gamma}}}

\newcommand{\bsnu}{{\boldsymbol{\nu}}}
\newcommand{\bsm}{{\boldsymbol{m}}}

\newcommand{\bsx}{{\boldsymbol{x}}}

\newcommand{\bsy}{{\boldsymbol{y}}}
\newcommand{\wbsy}{\widetilde{\boldsymbol{y}}}
\newcommand{\wa}{\widetilde{a}}
\newcommand{\wb}{\widetilde{b}}
\newcommand{\wlam}{\widetilde{\lambda}}

\newcommand{\bsz}{{\boldsymbol{z}}}

\newcommand{\bsw}{{\boldsymbol{w}}}

\newcommand{\bszero}{{\boldsymbol{0}}}

\newcommand{\rd}{\,\mathrm{d}}

\newcommand{\bE}{\mathbb{E}}

\newcommand{\bbE}{\mathbb{E}}
\newcommand{\calA}{\mathcal{A}}

\newcommand{\calI}{\mathcal{I}}

\newcommand{\calM}{\mathcal{M}}

\newcommand{\calF}{\mathcal{F}}

\newcommand{\calG}{\mathcal{G}}

\newcommand{\calP}{\mathcal{P}}
\newcommand{\calW}{\mathcal{W}}

\renewcommand{\forall}{\text{ for all }}

\def\R{\mathbb{R}}

\newcommand{\setu}{{\mathrm{\mathfrak{u}}}}

\newcommand{\mask}[1]{{}}

\newcommand{\eps}{\varepsilon}

\definecolor{darkred}{RGB}{139,0,0}
\definecolor{darkgreen}{RGB}{0,100,0}
\definecolor{darkmagenta}{RGB}{170,0,120}
\definecolor{darkpurple}{RGB}{110,0,180}
\definecolor{darkblue}{RGB}{40,0,200}
\definecolor{darkbrown}{rgb}{0.75,0.40,0.15}

\newcommand{\be}{\begin{equation}}
\newcommand{\ee}{\end{equation}}
\newcommand{\bea}{\begin{eqnarray}}
\newcommand{\eea}{\end{eqnarray}}
\newcommand{\beas}{\begin{eqnarray*}}
\newcommand{\eeas}{\end{eqnarray*}}

\def\r2p{{\sqrt{2\pi}}}

\newcommand*{\CLip}{\ensuremath{C_\mathrm{Lip}}}
\newcommand*{\Cgap}{\ensuremath{C_\mathrm{gap}}}
\newcommand*{\wCLip}{\ensuremath{\widetilde{C}_\mathrm{Lip}}}
\newcommand*{\evalueLap}{\ensuremath{\chi_1}}
\newcommand*{\amin}{\ensuremath{a_{\min}}}
\newcommand*{\amax}{\ensuremath{a_{\max}}}

\newcommand*{\lambdabar}{\ensuremath{\overline{\lambda_1}}}
\newcommand*{\ubar}{\ensuremath{\overline{u_1}}}
\newcommand*{\lambdaunder}{\ensuremath{\underline{\lambda_1}}}

\newcommand*{\evalover}{\ensuremath{\overline{\lambda_k}}}

\newcommand{\indx}{{\mathfrak F}}

\newcommand{\bigO}{\mathcal{O}}

\graphicspath{{.}{./Figures/}}

\title{Analysis of quasi-Monte Carlo methods for 
elliptic eigenvalue problems with stochastic coefficients}

\date{\today}

\makeatletter
\let\@fnsymbol\@arabic
\makeatother

\author{A. D. Gilbert\footnotemark[1]\and
             I. G. Graham\footnotemark[2] \and
             F. Y. Kuo\footnotemark[3] \and
             R. Scheichl\footnotemark[1] $^\mathrm{,}$\footnotemark[2] \and
             I. H. Sloan\footnotemark[3]
             }

\begin{document}
\sloppy
\maketitle

\footnotetext[1]{Institute for Applied Mathematics \& Interdisciplinary Center for Scientific Computing,
									Universit\"at Heidelberg, 69120 Heidelberg, Germany\\
									\indent\indent\texttt{a.gilbert@uni-heidelberg.de}, \texttt{r.scheichl@uni-heidelberg.de}
									}
\footnotetext[2]{Department of Mathematical Sciences, University of Bath, Bath BA2 7AY UK\\
                          \indent\indent\texttt{i.g.graham@bath.ac.uk}
                          }
\footnotetext[3]{School of Mathematics and Statistics,
                          University of New South Wales, Sydney NSW 2052, Australia\\
                         \indent\indent\texttt{f.kuo@unsw.edu.au},
                         \texttt{i.sloan@unsw.edu.au}}
                                                   
\numberwithin{equation}{section}

\begin{abstract}
We consider the forward problem of uncertainty quantification for the  
generalised Dirichlet eigenvalue problem for a coercive second order partial differential operator with 
random coefficients, motivated by problems in structural mechanics,
photonic crystals  and neutron diffusion.  The PDE  coefficients are
assumed to be uniformly bounded random fields, represented as 
infinite series parametrised by  
uniformly distributed i.i.d. random variables.  The expectation  
of the fundamental eigenvalue of this problem is computed by (a) truncating the infinite series which  define  the coefficients; (b) approximating the resulting truncated problem using  
lowest order conforming finite elements and a sparse matrix eigenvalue solver; and (c) approximating the resulting finite (but high dimensional) integral 
by a randomly shifted quasi-Monte Carlo lattice rule,  with specially chosen generating vector. 
 We prove error estimates for the combined error, which depend on
 the truncation dimension $s$,
the finite element mesh  diameter $h$, 
and the number of quasi-Monte Carlo samples $N$.
Under suitable regularity assumptions, our bounds are of the particular  form
$\mathcal{O}(h^2 + N^{-1 + \delta})$, where 
$\delta > 0$ is arbitrary and the hidden constant is
independent of the truncation dimension, which needs to grow as
  $h\to 0$ and $N \to \infty$.   
As for the analogous PDE source problem, the conditions under which our
error bounds hold depend on a parameter $p \in (0, 1)$
representing the summability of the terms in the series expansions of the coefficients.
Although the eigenvalue problem is nonlinear, which means it is generally considered harder than the 
source problem, in almost all cases ($p \neq 1$)
we obtain error bounds that converge at the same rate
as the corresponding rate for the source problem.
The proof involves  a detailed  study of the regularity of the fundamental eigenvalue as a function of the random parameters. As a key intermediate result in
 the analysis, we prove that the spectral gap (between the fundamental and the second eigenvalues) is uniformly positive over all realisations of the random problem.  

\end{abstract} 
\section{Introduction}
In this paper, we will propose methods for solving  
random 2nd-order elliptic eigenvalue
problems (EVP) of the general  form  
\begin{align}
\label{eq:evp}
-\nabla\cdot\big(a(\bsx, \bsy)\,\nabla u(\bsx, \bsy)\big) 
+ b(\bsx, \bsy)\,u(\bsx, \bsy)
&= 
\lambda(\bsy) \,c(\bsx,\bsy ) \,u(\bsx, \bsy), \ \  \text{for } \bsx \in D,
\end{align}
where the derivative operator $\nabla$ is with respect to the 
{\em physical variable} $\bsx$ and where the {\em stochastic parameter}
\begin{equation*}
\bsy \,=\, (y_j)_{j \in \N} \in U \coloneqq [-\tfrac{1}{2}, \tfrac{1}{2}]^\N 
\end{equation*}
is an infinite-dimensional vector of independently and identically distributed 
uniform random variables on $[-\tfrac{1}{2}, \tfrac{1}{2}]$.  For
simplicity, the
physical domain $D \subset \R^d$, 
for $d = 1, 2, 3$, is assumed to be a bounded convex domain
with Lipschitz boundary. To
guarantee well-posedness of the eigenvalue problem \eqref{eq:evp}, 
we impose homogeneous Dirichlet boundary conditions:
\begin{align}
\label{eq:DBC} 
u(\bsx, \bsy) = 0 \quad  \text{for} \quad  \bsx \in \partial D . 
\end{align} 

Under the initial assumption that 
$a(\cdot, \bsy),\ b(\cdot, \bsy),\ c(\cdot, \bsy)  \in L^\infty(D)$, together with  
$a(\bsx, \bsy),\ c(\bsx, \bsy)\geq a_{\min}  >0$  for all $(\bsx, \bsy) \in D \times U$,
the eigenvalues in~\eqref{eq:evp}
are real and bounded from below (this is a simple extension of the results in 
\cite[Sec.~3.2]{W66}),  and the leftmost (or
{\em dominant}) eigenvalue $\lambda_1$ is simple. Since the
coefficients depend on the stochastic parameters, 
the eigenvalues $\lambda(\bsy)$ and corresponding eigenfunctions $u(\bsx, \bsy)$ 
will also be stochastic.

Problems of the form \eqref{eq:evp} appear in many areas of
engineering and physics. Two prominent examples are in nuclear reactor physics \cite{W66,DH76,S97,JC13} and in
photonics \cite{D99,K01,GG12,NS12}.
Problems of a similar type also appear in quantum physics,  
in acoustic, electromagnetic or elastic wave
propagation, and in structural mechanics, where there is a huge
engineering literature on the topic (see e.g. \cite{Pet04,GGR05}).

In nuclear reactor physics, the eigenproblem \eqref{eq:evp}
corresponds to the mono-energetic diffusion approximation of the
neutron transport equation \cite{W66,DH76,S97,JC13}. The dominant 
eigenvalue $\lambda_1$ of \eqref{eq:evp} describes the criticality of the
reactor, while the corresponding eigenfunction $u_1$ models 
the associated neutron flux. The coefficient functions $a$, 
$b$ and $c$ correspond, respectively, to the 
diffusion coefficient, the absorption cross section and the fission 
cross section of the various materials in the reactor. These coefficients
can vary strongly in $\bsx$ and are 
subject to uncertainty in the composition of the constituent
materials (e.g.~liquid and vapour in the coolant), due to wear
(e.g.~``burnt'' fuel) and due to geometric deviations from the original
reactor design \cite{AI10,W10,AEHW12,W13}. 

In photonic band gap calculations in translationally
invariant materials, e.g. photonic crystal fibres (PCFs), two
decoupled eigenproblems of the type
\eqref{eq:evp} have to be solved (with periodic boundary conditions): the transverse magnetic (TM) and the transverse
electric (TE) mode problem \cite{D99,K01,GG12,NS12}. Here, $b
\equiv 0$ and we have either $a \equiv 1$ and $c =
n^2$ (TM mode problem) or  $a =
1/n^{2}$ and $c \equiv 1$ (TE mode problem), where $n = n(\bsx,\bsy)$ is the refractive
index of the PCF. The  refractive
index can be subject to uncertainty,  due to heterogeneities or impurities in the material and due
to geometric variations \cite{E12,ZCC15}.

The  current paper  demonstrates  the power of  
quasi-Monte Carlo (QMC) methods for computing statistics of the
eigenvalues of \eqref{eq:evp}, \eqref{eq:DBC}.  Our analysis will be restricted to 
approximating the expected value of the dominant 
eigenvalue $\lambda_1$ and linear functionals of the
corresponding eigenfunction,
but the method is applicable much 
more generally, and in particular to the applications listed above. 
We assume that, for all $\bsx \in D$ and $\bsy \in U$, the
coefficients $a$ and $b$ admit series expansions of the following form:
\begin{align}
 a(\bsx, \bsy) \,=\, a_0(\bsx) + \sum_{j = 1}^\infty y_j a_j(\bsx) 
\quad \text{and} \quad
b(\bsx, \bsy) \,=\, b_0(\bsx) + \sum_{j = 1}^\infty y_j b_j(\bsx)\,,
\label{eq:a_general}
\end{align}
and for the analysis assume that $c(\bsx, \bsy) = c(\bsx)$.
Although the fields $a$ and $b$ are parametrised by the same
  infinite sequence of random
variables $(y_j)_{j\ge 0}$, this setting for the coefficients allows
complete flexibility 
with respect to the correlation between $a$ and $b$. To model
two fields that are not correlated with each other, it suffices to set
$a_{2j-1} \equiv 0$ and $b_{2j} \equiv 0$, for all $j \ge 1$. On the
other hand, if there exists a $j \ge 1$ such that $a_j \not\equiv 0$
and $b_j \not\equiv 0$ then the two random fields will be correlated.

For a function $ f : U \to \R$,  
its  expected value  with respect to the product uniform probability distribution is the  
infinite-dimensional integral:
\begin{equation*}
\int_{[-\frac{1}{2}, \frac{1}{2}]^\N} f(\bsy) \,\rd \bsy \,\coloneqq\, \lim_{s \rightarrow \infty} \int_{[-\frac{1}{2}, \frac{1}{2}]^s} f\left(y_1, \ldots, y_s, 0, \ldots\right) \rd y_1\cdots \rd y_s \, ,
\end{equation*}
provided that the limit exists.
Our quantity of interest is then
\begin{align}
\label{eq:qoi}
\bbE_\bsy\left[\lambda_1\right] \,=\, \int_{[-\frac{1}{2}, \frac{1}{2}]^\N} \lambda_1(\bsy) \, \rd \bsy \,.
\end{align}

Our strategy for approximating \eqref{eq:qoi} is to first truncate the
expansions in~\eqref{eq:a_general} to $s$ parameters by setting 
$y_j = 0$, for $j > s$, thus reducing 
\eqref{eq:qoi} to a finite-dimensional quadrature problem. Then, for each $\bsy \in [-1/2,1/2]^s$,
we approximate \eqref{eq:evp} using  a finite
element (FE) discretisation on a mesh $\mathscr{T}_h$ with mesh size
$h$, to obtain a parametrised generalised
matrix eigenproblem that can be solved iteratively (e.g. via an Arnoldi method or similar). 
The corresponding approximate dominant eigenvalue  
is denoted $\lambda_{1,
  s, h}(\bsy)$. Now to approximate the integral in
\eqref{eq:qoi}, we construct $N$ suitable QMC quadrature points in
the $s$-dimensional unit cube $[0, 1]^s$ via a rank-1 lattice rule
with generating vector $\bsz \in \N^s$
(cf. \cite{DKS13}). The entire pointset is shifted by a uniformly-distributed random 
shift $\bsDelta \in [0, 1)^s$ and then translated  into the cube $[-\frac{1}{2}, \frac{1}{2}]^s$. 
The final estimate of $\bE_\bsy\left[\lambda_1\right]$ is then the (equal-weight) 
average of the approximate eigenvalues  $\lambda_{1, s, h}$ 
at these $N$ shifted QMC quadrature points and is  
denoted $Q_{N,s}(\bsz,\bsDelta) \lambda_{1, s, h}$.

The error depends on  $h,s$
and $N$ and to estimate it  we make some further assumptions on the 
coefficients, which are all detailed in Assumption~A\ref{asm:coeff}. 
In particular, to bound the FE error (w.r.t.~$h$),
we require some spatial 
regularity of $(a_j)_{j\ge0}$, $(b_j)_{j\ge0}$ and $c$.
To bound the dimension-truncation error (w.r.t.~$s$) and the
quadrature error (w.r.t.~$N$), we
assume $p$-summability of the sequences 
$\left(\|a_j\|_{L^\infty(D)} \right)_{j\ge 0}$ and
$\left(\|b_j\|_{L^\infty(D)}\right)_{j\ge 0}$, for some $p \in
  (0,1)$.

The main result in this paper is that, under these assumptions,
there exists a constant independent of $h, s$ and $N$
such that
\begin{equation}
\label{eq:main_bound}
\sqrt{\bbE_\bsDelta\left[\left|\bbE_\bsy[\lambda_1] - Q_{N, s}(\bsz, \bsDelta)\lambda_{1,s, h}\right|^2\right]}
\,\leq\,C\,\left(h^2 + s^{-\frac{2}{p} + 1} + N^{-\alpha} \right)
\end{equation}
where $\alpha = \min(1 - \delta, 1/p - 1/2)$
for arbitrary $\delta \in (0, 1/2)$. 
This result, for which a full statement is given in 
Theorem \ref{thm:total_error} (along with a similar result for linear functionals
$\calG$ of the corresponding eigenfunction), summarises the
individual contributions to the overall error from the three
approximations, i.e. discretisation ($h$),  dimension-truncation ($s$)
and quadrature ($N$).
The errors in the three separate processes are established individually  in Theorems
\ref{thm:fe_err}, \ref{thm:trunc} and \ref{thm:qmc_err},
respectively. To give a simple example of the power of estimate
\eqref{eq:main_bound}, if $p$ is small enough, then the dimension-truncation
error is negligible and the total error in \eqref{eq:main_bound} is 
bounded by the optimal FE convergence rate $h^2$ plus a QMC convergence 
rate that is arbitrarily close to~$N^{-1}$. 
Importantly, the constant $C$ does not depend on $s$.

A key result in obtaining \eqref{eq:main_bound} is Lemma~\ref{lem:dlambda}, where we 
establish the  regularity of $\lambda_1$ and
$u_1$ with respect to $\bsy$. The bounds on the mixed partial derivatives 
$|\pdy{\bsnu}\lambda_1(\bsy)|$ of $\lambda_1$   
are of product and order-dependent (POD) form, as in the case of 
(linear) boundary value problems \cite{KSS12}. Here, $\bsnu$ is a
multi-index with finitely many non-zero entries $\nu_j \in \N$.
The order dependence of the bounds in Lemma \ref{lem:dlambda} is 
$(|\bsnu|!)^{1 + \epsilon}$, for $\epsilon$ arbitrarily close to zero, which is only slightly
larger than in the bounds in \cite{KSS12} and still allows us to achieve the (nearly) optimal 
dimension-independent QMC convergence rates. The constants in these bounds depend on the gap 
between $\lambda_1(\bsy)$ and the second smallest eigenvalue 
$\lambda_2(\bsy)$. Another important result is Proposition~\ref{prop:simple},
where we prove that this gap is bounded away from zero uniformly in $\bsy$ 
under the  assumptions which we shall make  on $a$,  $b$ and $c$.
Both Proposition~\ref{prop:simple} and Lemma~\ref{lem:dlambda}
are essential components of our error analysis, however, 
they are also significant results in their own right.

Although stochastic eigenproblems have been of 
interest in engineering for some time, 
the mathematical literature is less developed.
A common method of tackling these problems is the \emph{reduced basis method}
\cite{MMOPR00,Pau07,FMPV16,HWD17}, whereby the full parametric solution (eigenvalue) 
is approximated in a low-dimensional subspace that is constructed as the span of the solutions
at specifically chosen parameter values.
For the current work the most relevant paper is \cite{AS12}, where a sparse tensor approximation was
used to estimate the expected value of the eigenvalue. A key result there 
is that simple eigenpairs are analytic with respect to the
stochastic parameters, shown using the classical perturbation theory of Kato
\cite{Kato84}. However, no bounds on the derivatives
are given, which are required to theoretically justify the application
of QMC rules. Here, we extend the results from \cite{AS12} by proving explicit bounds on
the derivatives, which in turn allows us to derive our \textit{a priori} error bounds. 
Alternatively, this paper can also be viewed as extending the results on
QMC methods for stochastic elliptic source problems \cite{KSS12} to eigenvalue problems, 
and we remark that because of the nonlinearity of eigenvalue problems 
this is not merely a trivial extension. 
Despite the increased difficulty of this nonlinearity our error bound \eqref{eq:main_bound}  
achieves the same rates of convergence as the equivalent result for the PDE source problem
in \cite{KSS12},  for all $p \in (0, 1)$.
The only difference is that our result does not hold for $p = 1$, in which
case the result in \cite{KSS12} requires an additional assumption
on the summability anyway.

The structure of the paper is as follows. In Section~\ref{sec:background}, 
we provide some relevant background theory of
elliptic eigenproblems and of randomised lattice rules, and establish
the FE error bound. Section \ref{sec:parametric} contains the parametric regularity
analysis, which is then used in Section \ref{sec:error_analysis} to bound the quadrature
and the truncation error. The paper concludes with two brief numerical
experiments in Section \ref{sec:numerical}, which
demonstrate the sharpness of the bounds.
An appendix contains the (technical) proof of the FE element error
estimate.

\section{Preliminary theory}
\label{sec:background}

In this section we present some preliminary theory on variational 
eigenvalue problems, FE discretisation and QMC methods. First we outline all of our
assumptions on the coefficients, which, in particular, ensure that the problem \eqref{eq:evp}
is well-posed.

\begin{assumption}
\label{asm:coeff}
\hfill
\begin{enumerate}
\item\label{itm:coeff} $a$ and $b$ are of the form \eqref{eq:a_general}
with $a_j,\ b_j \in L^\infty(D)$, for all $j \ge 0$, and $c \in L^\infty(D)$.
\item\label{itm:amin} There exists $\amin > 0$ such that $a(\bsx, \bsy) \geq \amin$,
$b(\bsx, \bsy) \geq 0$ and $c(\bsx) \geq \amin$, for all $\bsx \in D$, $\bsy \in U$.
\item\label{itm:summable} For some $p \in (0, 1)$,
\begin{align*}
\sum_{j = 1}^\infty \nrm{a_j}{L^\infty(D)}^p \,<\, \infty
\quad \text{and} \quad
 \sum_{j = 1}^\infty \nrm{b_j}{L^\infty(D)}^p 
\,<\, \infty\,.
\end{align*}
\item\label{itm:reg_coeff} $a_j \in W^{1, \infty}(D)$ for $j \geq 0$ and
\begin{align*}
\sum_{j = 1}^\infty \nrm{a_j}{W^{1, \infty}(D)}
\,<\, \infty\,.
\end{align*}
\end{enumerate}
\end{assumption}

The assumption that $b(\bsx, \bsy) \geq 0$ is made without loss of generality
because any EVP with $b < 0$, but satisfying the rest of
  Assumption~A\ref{asm:coeff},  can be shifted to an
equivalent problem with ``new $b$'' non-negative by adding 
$\sigma \cdot c(\bsx)\cdot u(\bsx, \bsy)$ to both sides of \eqref{eq:evp},
where $\sigma$ is chosen such that
$-b(\bsx, \bsy) \leq \sigma\cdot c(\bsx)$ for all $\bsx$,~$\bsy$. 
Such a $\sigma$ exists due to Assumption~A\ref{asm:coeff}.
The eigenvalues of the original EVP are simply the eigenvalues of 
the shifted problem shifted by $-\sigma$,
and the corresponding eigenspaces are unchanged. 
This has been used previously in, e.g., \cite{GG12}.

Throughout the paper, when it is unambiguous we will drop the 
$\bsx$-dependence when referring to a function defined on $D$ 
at a parameter value $\bsy$.
For example, we will write the coefficients and eigenfunctions as
$a(\bsy) \coloneqq a(\cdot, \bsy)$, $b(\bsy) \coloneqq b(\cdot, \bsy)$ 
and $u(\bsy) \coloneqq u(\cdot, \bsy)$.

Assumptions~A\ref{asm:coeff}.\ref{itm:coeff}--A\ref{asm:coeff}.\ref{itm:summable}
imply that for all $\bsy\in U$ we have 
$a(\bsy)$, $b(\bsy) \in L^\infty(D)$.
Furthermore, by the triangle inequality,  
the $L^\infty$-norms of these two coefficients can be bounded from above 
independently of $\bsy$: for all $\bsx \in D, \bsy \in U$
\[
\nrm{a(\bsy)}{L^\infty(D)} \,\leq\, \nrm{a_0}{L^\infty(D)} + \frac{1}{2} \sum_{j = 1}^\infty \nrm{a_j}{L^\infty(D)}\,,
\]
and similarly for $\nrm{b(\bsy)}{L^\infty(D)}$.
For convenience, we define a single upper bound for all three coefficients
\begin{align}
\label{eq:coeff_bounded}
\amax \,\coloneqq\, 
\max \left\{\sup_{\bsy \in U}\nrm{a(\bsy)}{L^\infty(D)},\, 
\sup_{\bsy \in U}\nrm{b(\bsy)}{L^\infty(D)},\,
\nrm{c}{L^\infty(D)}\right\}\,.
\end{align}

In Assumption~A\ref{asm:coeff}.\ref{itm:reg_coeff}, $W^{1, \infty}(D)$ is
the usual Sobolev space of functions with essentially bounded gradient on $D$, to which we
attach the norm $\nrm{v}{W^{1, \infty}(D)} \coloneqq \max\{\nrm{v}{L^\infty(D)}, \nrm{\nabla v}{L^\infty(D)}\}$. 
This assumption 
is needed to obtain the regularity result in Proposition~\ref{prop:H^2}.

\subsection{Abstract theory for variational eigenproblems}
\label{sec:varevp}
To construct the variational formulation of the EVP \eqref{eq:evp},  we 
introduce the test space $V \coloneqq H^1_0(D)$, the usual first-order Sobolev 
space of real-valued functions with vanishing boundary trace, 
as well as its dual $V^* = H^{-1}(D)$, 
and equip $V$ with the norm
\begin{align*}
\nrm{v}{V} \,\coloneqq\, \nrm{\nabla v}{L^2(D)}\, .
\end{align*}

We identify $L^2(D)$ with its dual and note that we have the following compact 
embeddings: $V \subset L^2(D) \subset V^*$.
The $L^2(D)$ inner product is denoted by $\langle \cdot, \cdot \rangle$ and we use 
the same notation for the extension to the duality pairing on $V\times V^*$. 
Throughout the paper we shall repeatedly refer to the 
eigenvalues of the negative Laplacian on $D$, with boundary condition \eqref{eq:DBC}. 
These are strictly positive and, counting multiplicities, we denote them by
\begin{equation}
\label{eq:Lap_eval}
0 \,<\, \chi_1 \,<\,  \chi_2 \,\leq\, \cdots\,.
\end{equation}

Multiplying each side of \eqref{eq:evp} by $v \in V$ and integrating (by parts) 
over $D$, we obtain the variational formulation
\begin{align}
\label{eq:varform}
\nonumber
\int_D a(\bsx, \bsy) \nabla u(\bsx, \bsy)\cdot \nabla v(\bsx) \,\rd \bsx 
+ \int_D b(\bsx, \bsy) u(\bsx, \bsy)v(\bsx) \,\rd \bsx\\
\,=\,
\lambda(\bsy) \int_D c(\bsx)u(\bsx, \bsy)v(\bsx) \, \rd \bsx 
\quad \text{for all } v \in V\,.
\end{align}
Correspondingly, for each $\bsy$ we define the symmetric bilinear forms 
$\calA(\bsy; \cdot, \cdot) : V \times V \rightarrow \R$ by
\begin{align}
\label{eq:bilinear}
\calA(\bsy; w, v) \,\coloneqq\, \int_D a(\bsx, \bsy) \nabla w(\bsx) \cdot \nabla v(\bsx) \, \rd \bsx 
+ \int_D b(\bsx, \bsy) w(\bsx)v(\bsx) \,\rd \bsx\,,
\end{align}
and $\calM(\cdot, \cdot) : V\times V \rightarrow \R$ by
\begin{align}
\label{eq:M_inner}
\calM(w, v) \,\coloneqq\, \int_D c(\bsx) w(\bsx)v(\bsx) \, \rd \bsx\,,
\end{align}
which are both inner products on their respective domains. Again, we will use 
the same notation for the corresponding duality pairings on $V \times V^*$. 
The norm induced by $\calM$ is denoted 
\[
\nrm{v}{\calM} \,=\, \sqrt{\calM(v, v)}
\]
and (using Assumption~A\ref{asm:coeff} and \eqref{eq:coeff_bounded})
is equivalent to the $L^2(D)$-norm:
\begin{align}
\label{eq:nrms_equiv}
\sqrt{a_{\min}}\nrm{v}{L^2(D)} 
\,\leq\,
\nrm{v}{\calM} 
\,\leq\,
\sqrt{a_{\max}}\nrm{v}{L^2(D)}\,,
\quad \text{for } v \in L^2(D)\,.
\end{align}

For each $\bsy \in U$, the variational eigenproblem equivalent to \eqref{eq:evp} is then:
Find $0 \neq u(\bsy) \in V$ and $\lambda(\bsy) \in \R$ such that
\begin{align}\label{eq:varevp}
\nonumber
\calA(\bsy; u(\bsy), v) &\,=\, 
\lambda(\bsy) \calM(u(\bsy), v ),\quad \text{for all } v \in V\,,\\
\nrm{u(\bsy)}{\calM} &\,=\, 1\,.
\end{align}

In the following, let $\bsy \in U$ be fixed. The
bilinear form $\calA(\bsy; \cdot, \cdot)$ is coercive 
and bounded, uniformly in~$\bsy$, i.e,
\begin{align}
\label{eq:A_coerc}
\calA(\bsy; v, v) \,&\geq\, a_{\min}\nrm{v}{V}^2\,, 
\quad \text{for all } v \in V\,,
\quad \text{and}\\
\label{eq:A_bounded}
\calA(\bsy; w, v) \,&\leq\, a_{\max}\left(1 + \evalueLap^{-1} \right)
\nrm{w}{V}\nrm{v}{V}\,,
\quad \text{for all } w, v \in V\,,
\end{align}
with $\amin$ as in Assumption~A\ref{asm:coeff}.\ref{itm:amin} and
$\amax$ from \eqref{eq:coeff_bounded}, respectively. To establish \eqref{eq:A_bounded}
we have used the upper bound \eqref{eq:coeff_bounded} and the Poincar\'e inequality: 
\begin{align}\label{eq:poin}
\nrm{v}{L^2(D)} \,\leq\, \evalueLap^{-1/2}\nrm{v}{V} \, ,
\quad \text{for } v \in V\,,
\end{align}
where for notational convenience we write the constant
in terms of the Laplacian eigenvalue $\evalueLap$, as defined above in \eqref{eq:Lap_eval}.
It is easy to see that the bound in \eqref{eq:poin} holds true and
that equality is attained for $v = \phi_1$, the eigenfunction corresponding
to $\chi_1$.

In addition to the variational form \eqref{eq:varevp},
it will also be convenient for us to study the corresponding solution operator,
which we define now. 
Let  $f \in V^*$ be arbitrary, and 
for each $\bsy \in U$, consider $T(\bsy): V^*\rightarrow V$ given by
\begin{equation}
\label{eq:T_def}
\calA(\bsy; T(\bsy) f, v) \,=\, \calM(f, v)
\quad \text{for all } v \in V\,.
\end{equation}
Due to the symmetry of both $\calA(\bsy; \cdot, \cdot)$ and $\calM(\cdot, \cdot)$,
each operator $T(\bsy)$ is self-adjoint with respect to $\calM$.
Since $\calA(\bsy; \cdot, \cdot)$ is coercive \eqref{eq:A_coerc} and bounded \eqref{eq:A_bounded},
by the Lax--Milgram Theorem, see, e.g., \cite{Brezis11}, for every $f \in V^*$
there exists a unique solution $T(\bsy)f \in V$ to \eqref{eq:T_def},
which satisfies $\nrm{T(\bsy)f}{V} \leq \nrm{f}{V^*}/\amin$.
Hence, each $T(\bsy) : V^* \to V$ is bounded.

We can also consider the operators $T(\bsy)$ from $L^2(D)$ to $L^2(D)$, in which case
due to the compact embedding of $V$ into $L^2(D)$, each $T(\bsy) : L^2(D) \to L^2(D)$ is 
compact. In this case, for $f \in L^2(D)$ the Lax--Milgram Theorem again 
gives a unique solution $T(\bsy)f \in V$ with the bound
\begin{equation}
\label{eq:T_LaxMil}
\nrm{T(\bsy)f}{V} \,\leq\,
\frac{1}{\sqrt{\chi_1}}\frac{\amax}{\amin} \nrm{f}{L^2(D)}\,,
\end{equation}
where we have also used the equivalence of norms \eqref{eq:nrms_equiv},
along with the Cauchy--Schwarz and Poincar\'e inequalities to bound $\nrm{f}{V^*}$.

From the spectral theory for compact, selfadjoint operators  we know that
each $T(\bsy)$ has countably-many eigenvalues,
which are all finite, real, strictly positive and have finite multiplicity (see, e.g., \cite{Brezis11}).
Counting multiplicities, the eigenvalues of $T(\bsy)$ are denoted (in non-increasing order) by
\[
\mu_1(\bsy) \,\geq\, \mu_2(\bsy) \,\geq\, \cdots \,>\, 0\,,
\]
with $\mu_k(\bsy) \to 0$ as $k \to \infty$.
Comparing \eqref{eq:varevp} with \eqref{eq:T_def} we have that
$\lambda(\bsy)$ is an eigenvalue of \eqref{eq:varevp} if and 
only if $\mu(\bsy) = 1/ \lambda(\bsy)$ is an eigenvalue of $T(\bsy)$,
and their eigenspaces coincide.
 It follows that \eqref{eq:varevp}
 has countably-many eigenvalues
$(\lambda_k(\bsy))_{k \in \N}$,  which are all positive, have
finite multiplicity and accumulate only at infinity. Counting multiplicities, we write them 
(in non-decreasing order) as  
\begin{equation*}
0 \,<\, \lambda_1(\bsy) \,\leq\, \lambda_2(\bsy) \,\leq\, \lambda_3(\bsy) \,\leq\,
\cdots\, . 
\end{equation*}
with $\lambda_k(\bsy) \rightarrow \infty$ as $k \rightarrow
\infty$.
For an eigenvalue $\lambda(\bsy)$ of \eqref{eq:varevp}
we define its eigenspace to be 
\begin{equation*}
E(\bsy, \lambda(\bsy)) \,\coloneqq\, 
\left\{u : u \text{ is an   eigenfunction corresponding to } \lambda(\bsy)\right\}\,,
\end{equation*}
and from these eigenspaces, we can choose a sequence of
eigenfunctions $(u_k(\bsy))_{k \in \N}$ corresponding to
$(\lambda_k(\bsy))_{k \in \N}$ that form an orthonormal basis in $V$ 
with respect to $\calM(\cdot, \cdot)$. This again follows by the
Spectral Theorem for $T(\bsy)$.

The \emph{min-max principle} (e.g. \cite[(2.8)]{BO89}) gives a representation of the $k$th eigenvalue 
as a minimum over all subspaces $S_k \subset V$ of dimension $k$:
\begin{align}
\label{eq:minmax}
\lambda_k(\bsy) \,=\, \min_{\substack{S_k \subset V\\ \dim(S_k) = k}}
\max_{0 \neq u \in S_k}
\frac{\calA(\bsy; u, u)}{\calM(u, u)}\,.
\end{align}

We can use a combination of \eqref{eq:coeff_bounded},
\eqref{eq:nrms_equiv} and \eqref{eq:A_coerc}
to bound both the numerator and denominator in the min-max representation \eqref{eq:minmax}
from above and below independently of $\bsy$, which in turn lets us
bound the $k$th eigenvalue above and below independently of $\bsy$. Indeed, we have
\begin{align*}
\lambda_k(\bsy) \,&\geq\,
\frac{\amin}{\amax} \min_{\substack{S_k \subset V\\ \dim(S_k) = k}}\max_{0 \neq u \in S_k} 
\frac{\langle \nabla u, \nabla u\rangle}{\langle u, u\rangle}
\quad \text{and}\\
 \lambda_k(\bsy) \,&\leq\, 
\frac{\amax}{\amin} \min_{\substack{S_k \subset V\\ \dim(S_k) = k}}\max_{0 \neq u \in S_k} 
\frac{\langle \nabla u, \nabla u\rangle + \langle u, u\rangle}{\langle u, u\rangle}\,,
\end{align*}
but now the right hand side of both bounds contains the bilinear form
corresponding to the negative Laplacian on $D$.
Hence, using the min-max representation of the $k$th Laplacian eigenvalue $\chi_k$
the bounds on $\lambda_k(\bsy)$ can be equivalently written as
\begin{align}\label{eq:lambda_bnd}
\underline{\lambda_k} \,\coloneqq\, \frac{\amin}{\amax}\chi_k
\,&\leq\, \lambda_k(\bsy) \,\leq\,
\frac{\amax}{\amin} (\chi_k + 1)\,\eqqcolon\,\overline{\lambda_k}\,.
\end{align}
Taking $v = u_k(\bsy)$ as a test function in \eqref{eq:varevp}, 
we obtain 
\begin{align}
\label{eq:RQ}
\lambda_k(\bsy) 
\,=\,\calA\left(\bsy; u_k(\bsy), u_k(\bsy)\right)\,.
\end{align}
Then, using  coercivity 
\eqref{eq:A_coerc} as well as the upper bound \eqref{eq:lambda_bnd} on $\lambda_k(\bsy)$, we obtain 
\begin{align}
\label{eq:u_bnd}
\nrm{u_k(\bsy)}{V} \,\leq\, \sqrt{\frac{\lambda_k(\bsy)}{a_{\min}}}
\,\leq\, \frac{\sqrt{\amax(\chi_k + 1)}}{\amin}
\,\eqqcolon\, \overline{u_k}\,.
\end{align}

Of particular interest is the smallest (also referred to as the minimal or fundamental) 
eigenvalue $\lambda_1(\bsy)$.
It follows by the Krein--Rutman Theorem that for every $\bsy$
the fundamental eigenvalue $\lambda_1(\bsy)$ is simple, see
e.g. \cite[Theorems 1.2.5 and 1.2.6]{Hen06}.
The fact that $\lambda_1(\bsy)$ is simple for all $\bsy$ along with the uniform bound 
\eqref{eq:lambda_bnd} ensures that the integrand \eqref{eq:qoi} is well-defined.
That the integrals $\bbE_\bsy[\lambda_1]$ and $\bbE_\bsy[\calG(u_1)]$ (for $\calG \in V^*$)
exist follows by the truncation error bounds \eqref{eq:lam_s_strong} and \eqref{eq:u_s_strong}, respectively, 
which are shown later in Section~\ref{sec:dim} using Assumption~A\ref{asm:coeff}.

In order to prove our finite element convergence results in Theorem~\ref{thm:fe_err}, 
for $t \in [0,1]$, we introduce the spaces: 
\begin{equation}
\label{eq:Z^t}
Z^t \,\coloneqq\, \left\{v \in V : \Delta v \in H^{-1 + t}(D) \right\}\,,
\end{equation}
with norm
\begin{equation}
\label{eq:Z_nrm}
\nrm{v}{Z^t} \,\coloneqq\,
\left(\nrm{v}{L^2(D)}^2 + \nrm{\Delta v}{H^{-1 + t}(D)}^2\right)^\frac{1}{2}\,.
\end{equation}
where $H^r(D)$ is the usual fractional order Sobolev space (see \cite{KSS12}). 
When $t=1$ we abbreviate $Z^1 $ by $Z$.  
Since $D$ is convex, $Z = V \cap H^2(D)$. 

The following proposition shows that under Assumption~A\ref{asm:coeff}, 
in particular A\ref{asm:coeff}.\ref{itm:reg_coeff},
the eigenfunctions belong to $H^2(D)$ with norm bounded in terms of the corresponding
eigenvalue.

\begin{Proposition}\label{prop:H^2}
Let $\bsy \in U$, let Assumption~A\ref{asm:coeff} hold and suppose  $(\lambda(\bsy), u(\bsy))$
is an eigenpair of \eqref{eq:varevp}. 
Then $u(\bsy) \in Z = V \cap H^2(D)$ and there exists a constant $C>0$
independent of $\bsy$, such that 
\begin{align}
\label{eq:u_H^2}
\nrm{u(\bsy)}{Z} \,\leq\, C\lambda(\bsy)\, , \quad \text{for all} \quad \bsy \in U .  
\end{align}
\end{Proposition}
\begin{proof}
Since $u(\bsy) \in L^2(D)$ we can apply \cite[Theorem 4.1]{KSS12}
with $t = 1$ and $f = (\lambda(\bsy)c - b(\bsy))u(\bsy) \in L^2(D)$ 
to give $u(\bsy) \in V \cap H^2(D)$ with
 \begin{align*}
 \nrm{u(\bsy)}{Z} \,&\leq\, C \nrm{(\lambda(\bsy)c - b(\bsy))u(\bsy)}{L^2(D)}\\
& \leq\, C\left(\lambda(\bsy)\nrm{c}{L^\infty(D)}  +  
\nrm{b(\bsy)}{L^\infty(D)}\right)
\nrm{u(\bsy)}{L^2(D)}
 \end{align*} 
Then, using  \eqref{eq:coeff_bounded}, \eqref{eq:nrms_equiv} and the normalisation in \eqref{eq:varevp}, we obtain 
\begin{align*} \nrm{u(\bsy)}{Z} \, & \leq\, C \, \frac{\amax}{\sqrt{\amin}}  
\left(\lambda(\bsy)  +  1\right) 
\nrm{u(\bsy)}{\calM}\,   = \,   C \frac{\amax}{\sqrt{\amin}} \left( \lambda(\bsy) +1\right)  \\
&\leq  \,   C \frac{\amax}{\sqrt{\amin}} \left( 1 +  \frac{1}{\lambda_1(\bsy)} \right)\lambda(\bsy)  ,    
\end{align*}
 and the result follows by the lower bound in \eqref{eq:lambda_bnd}.
\end{proof}

\subsection{Bounding the spectral gap} 
\label{subsec:gap} 
The Krein--Rutman theorem guarantees that the {\em spectral gap} $\lambda_2(\bsy) - \lambda_1(\bsy) $  is positive for each parameter value $\bsy$. However, our estimates  for the derivatives of $\lambda_1(\bsy)$
proved in Section~\ref{sec:parametric} require {\em uniform positivity} of this  gap over all $\bsy \in U$. Here, we prove the required 
uniform positivity, using Assumption~A\ref{asm:coeff}.\ref{itm:summable}. 
We remark that this proof provides a justification for an assumption made 
previously without proof  in  \cite{AS12}.

The first step is the following elementary lemma, which shows that subsets of   
$\ell^\infty$ that are majorised by an $\ell^q$ sequence (for some  $1 < q< \infty$) are compact.

\begin{Lemma}
\label{lem:compact_subset}
Let $\bsalpha \in \ell^q$ for some  $1 < q < \infty$. The set $U(\bsalpha) \subset \ell^\infty$ 
given by
\begin{equation*}
U(\bsalpha) \,\coloneqq\, \left\{\bsw \in \ell^\infty 
: |w_j| \leq \frac{1}{2} |\alpha_j|\right\}
\end{equation*}
is a compact subset of $\ell^\infty$.
\end{Lemma}

\begin{proof}
Since $\ell^\infty$ is a normed (and hence a metric) space, $U(\bsalpha)$
is compact if and only if it is sequentially compact. To show sequential compactness of $U(\bsalpha)$, 
take any sequence 
$\{\bsy^{(n)}\}_{n \geq 1} \subset U(\bsalpha)$. 
Clearly, by definition of $U(\bsalpha)$,   each $\bsy^{(n)} \in \ell^q$ and moreover, 
\[
\nrm{\bsy^{(n)}}{\ell^q} \,\leq\, \frac{1}{2}\nrm{\bsalpha}{\ell^q} \,<\, \infty
\quad \text{for all } n \in \N\, .
\]
So $\bsy^{(n)}$ is a bounded sequence in $\ell^q$. 
Since $q < \infty$,  $\ell^q$ is a reflexive Banach space, and so  
by \cite[Theorem~3.18]{Brezis11} $\{\bsy^{(n)}\}_{n \geq 1}$ has a subsequence
that converges weakly to a limit in  $\ell^q$. We denote this limit by $\bsy^*$, and,   
with a slight abuse of notation, we denote the convergent subsequence again  by 
$\{\bsy^{(n)}\}_{n \geq 1}$.

We now prove that $\bsy^* \in U(\bsalpha)$ and that the weak convergence is in fact  strong,  
i.e.  we show  $\bsy^{(n)} \rightarrow \bsy^*$ in $\ell^\infty$, as 
$n \rightarrow \infty$.
For any  $j \in \N$, consider the linear functional $f_j:
\ell^q \rightarrow \R$ given by $f_j(\bsw) = w_j$,  where $w_j$
  denotes the $j$th element of the sequence $\bsw = (w_j)_{j \geq 1} \in \ell^q$. 
Clearly, $f_j \in (\ell^q)^*$ (the dual space) and using the weak convergence established above, it follows that  
\[
 y_j^{(n)} \,=\, f_j(\bsy^{(n)}) \,\rightarrow\, f_j(\bsy^*) \,=\, y^*_j 
\quad \text{as } n \rightarrow \infty\,, \quad \text{for each fixed} \ \ j .   
 \] 
That is, we have componentwise convergence. 
 Furthermore, since $|y_j^{(n)}| \leq \frac{1}{2}|\alpha_j|$ it follows that 
 $|y_j^*| \leq \frac{1}{2}|\alpha_j|$ for each $j$,  
 and hence $\bsy^* \in U(\bsalpha)$.
 
 Now, for any $J \in \N$ we can write
\begin{align}
\label{eq:q_norm_split}\nonumber 
\nrm{\bsy^{(n)} - \bsy^*}{\ell^q}^q
\,&=\, 
\sum_{j = 1}^J |y_j^{(n)} - y_j^*|^q
+ \sum_{j = J + 1}^\infty |y_j^{(n)} - y_j^*|^q\\
\,&\leq\,
J\max_{j = 1, 2, \ldots, J} |y_j^{(n)} - y_j^*|^q
+ \sum_{j = J + 1}^\infty |\alpha_j|^q\,.
 \end{align}
 
Let $\varepsilon > 0$. Since $\bsalpha \in \ell^q$, we can choose $J \in \N$ such that
\[
\sum_{j = J + 1}^\infty |{\alpha_j}|^q \,\leq\, \frac{\varepsilon^q}{2}\,,
\]
and since $\bsy^{(n)}$ converges componentwise we can choose $K \in \N$ such that
\[
|y_j^{(n)} - y_j^*| \,\leq\, (2J)^{-1/q} \varepsilon
\quad \text{for all } j = 1, 2, \ldots, J
\text{ and } n \geq K\,.
\]
Thus, by \eqref{eq:q_norm_split} we have that $\nrm{\bsy^{(n)} - \bsy^*}{\ell^q}^q \leq \varepsilon^q$
for all $n \geq K$, and hence that $\nrm{\bsy^{(n)} - \bsy^*}{\ell^q} \rightarrow 0$ 
as $n \rightarrow \infty$.
 
Because $\nrm{\bsw}{\ell^\infty} \leq \nrm{\bsw}{\ell^q}$ when 
$\bsw \in \ell^q$  and $1 < q < \infty$, 
this also implies that $\bsy^{(n)} \rightarrow \bsy^*$ in $\ell^\infty$, completing the proof.  
\end{proof}

A key property following from the perturbation theory of Kato \cite{Kato84} 
is that the eigenvalues $\lambda_k(\bsy)$ are continuous in $\bsy$,  which 
for completeness is shown 
below in Proposition~\ref{prop:eval_cts}. First, recall that $T(\bsy)$ is 
the solution operator as defined in \eqref{eq:T_def}, and
let $\Sigma(T(\bsy))$ denote the spectrum of $T(\bsy)$.

\begin{Proposition}
\label{prop:eval_cts}
Let Assumption~A\ref{asm:coeff} hold. Then the eigenvalues 
$\lambda_1, \lambda_2 , \ldots$ are Lipschitz continuous in~$\bsy$.
\end{Proposition}

\begin{proof}
We prove the result by establishing the continuity of
the eigenvalues $\mu_k(\bsy)$ of $T(\bsy)$.
Let $\bsy$, $\bsy' \in U$ and consider the operators $T(\bsy), T(\bsy') : L^2(D) \rightarrow L^2(D)$
as defined in \eqref{eq:T_def}.
Since $T(\bsy)$, $T(\bsy')$ are bounded and self-adjoint with respect to $\calM$,
it follows from \cite[V, \S 4.3 and Theorem~4.10]{Kato84} that we have the
following notion of continuity of $\mu(\cdot)$ in terms of $T(\cdot)$
\begin{equation}
\label{eq:spectrum_cts}
\sup_{\mu \in \Sigma(T(\bsy))} \dist (\mu, \Sigma(T(\bsy')))
\,\leq\, \nrm{T(\bsy) - T(\bsy')}{L^2(D) \rightarrow L^2(D)}\,.
\end{equation}
%
For an eigenvalue $\mu_k(\bsy) \in \Sigma(T(\bsy))$, \eqref{eq:spectrum_cts} 
implies that there exists a $\mu_{k'}(\bsy') \in \Sigma(T(\bsy'))$ such that
\begin{equation}
\label{eq:spec_cts2}
|\mu_k(\bsy) - \mu_{k'}(\bsy')|
\,\leq\, \nrm{T(\bsy) - T(\bsy')}{L^2(D) \rightarrow L^2(D)}\,.
\end{equation}
Note that this means there exists an eigenvalue of $T(\bsy')$ close to
$\mu_k(\bsy)$, but does not imply that the $k$th eigenvalue of $T(\bsy')$
is close to $\mu_k(\bsy)$, that is, in \eqref{eq:spec_cts2} $k$ is not necessarily 
equal to $k'$. 
However, consider any $\mu_k(\bsy)$ and let $m$ denote its multiplicity.
Since $m < \infty$, we can assume without loss of generality that the collection 
$\mu_k(\bsy) = \mu_{k + 1}(\bsy) = \cdots = \mu_{k + m - 1}(\bsy)$
is a \emph{finite system} of eigenvalues in the sense of Kato.
It then follows from the discussion in \cite[IV, \S 3.5]{Kato84} that the eigenvalues in this
system depend continuously on the operator with multiplicity preserved.
This preservation of multiplicity is key to our argument, 
since it states that for  $T(\bsy')$ sufficiently close to $T(\bsy)$ there are $m$ 
consecutive eigenvalues 
$\mu_{k'}(\bsy'), \mu_{k' + 1}(\bsy'), \ldots , \mu_{k' + m - 1}(\bsy') \in \Sigma(T(\bsy'))$,
no longer necessarily equal, that are close to $\mu_k(\bsy)$.

A simple argument then shows that
each $\mu_k$ is continuous in the following sense
\begin{equation}
\label{eq:mu_cts_op}
|\mu_k(\bsy) - \mu_k(\bsy')|
\,\leq\, \nrm{T(\bsy) - T(\bsy')}{L^2(D) \rightarrow L^2(D)}\,.
\end{equation}
To see this, consider, for $k=1,2,\ldots$, the graphs of $\mu_k$ on $U$. Note that the separate graphs can touch (and in principle can even coincide over some subset of $U$), but by definition cannot cross (since at every point in $U$ the successive eigenvalues are nonincreasing); and by the preservation of  multiplicity a graph cannot terminate and a finite set of graphs cannot change multiplicity at an interior point.  Thus by (2.23) the ordered eigenvalues $\mu_k$ must be continuous for each $k \ge 1$  and satisfy \eqref{eq:mu_cts_op}.

It then follows from the relationship $\mu_k(\bsy) = 1/\lambda_k(\bsy)$ along with the upper bound
in \eqref{eq:lambda_bnd} that  we have a similar result for the
eigenvalues $\lambda_k$ of \eqref{eq:varevp}:
\begin{equation}
\label{eq:eval_cts_op}
\left|\lambda_k(\bsy) - \lambda_k(\bsy')\right|
\,\leq\, 
\evalover^2\nrm{T(\bsy) - T(\bsy')}{L^2(D) \rightarrow L^2(D)}\,.
\end{equation}

All that remains is to bound the right hand of 
\eqref{eq:eval_cts_op} by $\CLip\nrm{\bsy - \bsy'}{\ell^\infty}$,
with $\CLip > 0$ independent of $\bsy$ and $\bsy'$.
To this end, note that since the right hand side
of \eqref{eq:T_def} is independent of $\bsy$ we have
\[
\calA(\bsy; T(\bsy)f, v) \,=\, \calA(\bsy'; T(\bsy')f, v)
\quad 
\text{for all } f \in L^2(D), v \in V\,.
\]
Rearranging and then expanding this gives
\begin{align*}
&\calA\left(\bsy; \left(T(\bsy) - T(\bsy')\right)f, v\right))\\
&\qquad=\, \calA(\bsy'; T(\bsy')f, v) - \calA(\bsy; T(\bsy')f, v)
\\
&\qquad=\, \int_D\Big(\left[a(\bsx, \bsy') - a(\bsx, \bsy)\right]
\nabla [T(\bsy')f](\bsx) \cdot \nabla v(\bsx)\\
&\qquad\qquad\quad\,+ \left[b(\bsx, \bsy') - b(\bsx, \bsy)\right] \,[T(\bsy') f](\bsx) \, v(\bsx)\Big) \,\rd\bsx\,.
\end{align*}
Letting $v = (T(\bsy) - T(\bsy'))f \in V$, the left hand side can be bounded from below
using the coercivity \eqref{eq:A_coerc} of $\calA(\bsy; \cdot, \cdot)$, 
and the right hand side can be bounded from above using the Cauchy--Schwarz
inequality to give
\begin{align*}
\amin &\nrm{(T(\bsy) - T(\bsy'))f}{V}^2
\,\leq\, \max\left(\nrm{a(\bsy) - a(\bsy')}{L^\infty(D)},\, \nrm{b(\bsy) - b(\bsy')}{L^\infty(D)}\right)\\
&\quad\quad\cdot\left(\nrm{T(\bsy')f}{V}\nrm{(T(\bsy) - T(\bsy'))f}{V} + 
\nrm{T(\bsy')f}{L^2(D)}\nrm{(T(\bsy) - T(\bsy'))f}{L^2(D)}\right)\,.
\end{align*}
Applying the Poincar\'e inequality \eqref{eq:poin} to both $L^2$-norm factors, 
dividing by $\amin\nrm{(T(\bsy) - T(\bsy'))f}{V}$ 
and using the upper bound in \eqref{eq:T_LaxMil} we have
\begin{align*}
&\nrm{(T(\bsy) - T(\bsy'))f}{V} \\
&\qquad\leq\, \frac{\amax(1 + 1/\chi_1)}{\amin^2\sqrt{\chi_1}}
\nrm{f}{L^2(D)} \max\left(\nrm{a(\bsy) - a(\bsy')}{L^\infty(D)},\, \nrm{b(\bsy) - b(\bsy')}{L^\infty(D)}\right)\,.
\end{align*}
Then, applying the Poincar\'e inequality \eqref{eq:poin} to the left hand side and
taking the supremum over $f \in L^2(D)$ with $\nrm{f}{L^2(D)} \leq 1$,
in the operator norm we have
\[
\nrm{T(\bsy) - T(\bsy')}{L^2(D) \rightarrow L^2(D)}
\,\leq\,
\frac{\amax(\chi_1 + 1)}{\amin^2\chi_1^2}
\max\left(\nrm{a(\bsy) - a(\bsy')}{L^\infty(D)},\, \nrm{b(\bsy) - b(\bsy')}{L^\infty(D)}\right)\,.
\]

Using this inequality as an upper bound for \eqref{eq:eval_cts_op} we have
that the eigenvalues inherit the continuity of the coefficients
\begin{equation}
\label{eq:eval_cts_coeff}
|\lambda_k(\bsy) - \lambda_k(\bsy')|
\,\leq\, C\max\left(\nrm{a(\bsy) - a(\bsy')}{L^\infty(D)},\, \nrm{b(\bsy) - b(\bsy')}{L^\infty(D)}\right)\,,
\end{equation}
where
\[
C \,=\, \evalover^2 \frac{\amax(\chi_1 + 1)}{\amin^2\chi_1^2}
\,<\, \infty\,,
\]
is clearly independent of $\bsy$ and $\bsy'$.

Finally, to establish Lipschitz continuity with respect to $\bsy$, we 
recall Assumptions~A\ref{asm:coeff}.\ref{itm:coeff} and A\ref{asm:coeff}.\ref{itm:summable},
expand the coefficients in \eqref{eq:eval_cts_coeff} above and use the triangle inequality to give
\begin{align*}
|\lambda_k(\bsy) - \lambda_k(\bsy')|
\,&\leq\, C \sumj |y_j - y_j'|\max\left(\nrm{a_j}{L^\infty(D)},\, \nrm{b_j}{L^\infty(D)}\right)\\
\,&\leq\, C\Bigg(\sumj \max\left(\nrm{a_j}{L^\infty(D)},\, \nrm{b_j}{L^\infty(D)}\right) \Bigg) 
\nrm{\bsy - \bsy'}{\ell^\infty}\,.
\end{align*}
By Assumption~A\ref{asm:coeff} the sum is finite, and hence the eigenvalue $\lambda_k(\bsy)$ is Lipschitz in $\bsy$.
\end{proof}

Now that we have shown Lipschitz continuity of the eigenvalues and identified suitable compact subsets,
we can show that the spectral gap is bounded uniformly away from 0.

\begin{Proposition}
\label{prop:simple}
Let Assumption~A\ref{asm:coeff} hold. Then there exists a $\delta > 0$,
independent of $\bsy$, such that
\begin{equation}
\label{eq:gap}
\lambda_2(\bsy) -\lambda_1(\bsy)
\,\geq\, \delta\,.
\end{equation}
\end{Proposition}

\begin{proof}
The idea of the proof is to rewrite 
$a(\bsx, \bsy)$ as
$$ a(\bsx, \bsy) \ = \ a_0(\bsx) \ +\  \sumj  \widetilde{y}_j \wa_j(\bsx), $$
with  $\widetilde{y}_j = \alpha_j y_j$ and $\wa_j(\bsx) = a_j(\bsx)/\alpha_j$. We then choose  
$\bsalpha \in \ell^q$ to decay slowly enough  
so that 
$\sumj \Vert \wa_j\Vert_{L^\infty(D)} < \infty$ and we apply  a similar  reparametrisation 
procedure to  $b(\bsx, \bsy)$.  Then using the intermediate result \eqref{eq:eval_cts_coeff}
from the proof of Proposition~\ref{prop:eval_cts}
we can show that the eigenvalues of the ``reparametrised''  problem  are continuous in the 
new parameter  $\wbsy$,  which now ranges over the compact set $U(\bsalpha)$. 
The required bound on the spectral gap is obtained by using the equivalence of the 
eigenvalues of the original and reparametrised problems. 

To give some details,
note that there is no loss of generality in  assuming  $p > 1/2$, 
because if Assumption~A\ref{asm:coeff}.\ref{itm:summable} holds 
with exponent $p' \leq 1/2$ then it also holds for all $p \in (p', 1)$. 
We consequently set $\eps = 1-p \in (0, 1/2)$ and consider the sequence 
$\bsalpha$ defined by 
\begin{align} \label{eq:defalpha} 
\alpha_j \,=\, \Vert a_j\Vert ^\eps_{L^\infty(D)} + \Vert b_j\Vert ^\eps_{L^\infty(D)} + 1/j, \quad 
\text{for each} \ \ j = 1,2,\ldots .   
\end{align}
Setting  $q = p/\eps = p/(1-p) \in (1,\infty)$, using Assumption A\ref{asm:coeff}.3 and  
the triangle inequality, it is easy to see that 
$\bsalpha \in \ell^q$. Moreover, the inclusion of $1/j$ in
\eqref{eq:defalpha} ensures that  $\alpha_j \not= 0$, for all $j
\geq 1$. Hence, from now on, for  $\bsw = (w_j)_{j = 1}^\infty \in \ell^\infty$, 
 we can define the sequences  $\bsalpha \bsw = (\alpha_j w_j)_{j=1}^\infty$ and  $\bsw/\bsalpha = (w_j/\alpha_j)_{j=1}^\infty$. 
Then, recalling the definition of $U(\bsalpha)$ in Lemma \ref{lem:compact_subset}, 
it is easy to see that  
\[
\wbsy \in U(\bsalpha) \ \text{if and only if}\ \wbsy/\bsalpha \in U  \quad \text{and moreover} \quad \bsy \in U \ \text{if and only if}\ \bsalpha \bsy \in U(\bsalpha)\,.
\]
Now for $\bsx \in D$ and $\wbsy \in U(\bsalpha)$, we define
\[
\wa(\bsx, \wbsy) = a_0(\bsx) + \sum_{j=1}^\infty \widetilde{y}_j \frac{a_j(\bsx)}{\alpha_j} \ \ \text{and} \ \ 
\wb(\bsx, \wbsy) = b_0(\bsx) + \sum_{j=1}^\infty  \widetilde{y}_j \frac{b_j(\bsx)}{\alpha_j}  \,,
\]
from which it is easily seen that 
\begin{equation}
\label{eq:coeff_1to1}
\wa(\bsx, \wbsy) = a (\bsx, \wbsy/\bsalpha) \quad \text{and} \quad \wb(\bsx, \wbsy) = b (\bsx, \wbsy/\bsalpha)\,.
\end{equation}
Then we set 
\[
\widetilde{\calA}(\widetilde{\bsy}; w, v)
\,\coloneqq\, \int_D \left(\widetilde{a}(\bsx, \widetilde{\bsy}) \nabla w(\bsx) . \nabla v(\bsx)
+ \widetilde{b}(\bsx, \widetilde{\bsy}) w(\bsx) v(\bsx) \right) \, \rd \bsx
\quad \text{for } w, v \in V\, ,
\]
and we consider the reparametrised eigenvalue problem 
find $\widetilde{\lambda}(\widetilde{\bsy}) \in \R$
and $0 \neq \widetilde{u}(\widetilde{\bsy}) \in V$ such that
\begin{align}
\label{eq:reparam_evp}\nonumber
\widetilde{\calA}(\widetilde{\bsy}; \widetilde{u}(\widetilde{\bsy}), v)
\,&=\, \widetilde{\lambda}_k(\widetilde{\bsy}) \calM(\widetilde{u}(\widetilde{\bsy}), v)
\quad \text{for all } v \in V\,,\\
\nrm{\widetilde{u}(\widetilde{\bsy})}{\calM} \,&=\, 1\,.
\end{align}
Note that because we have equality between the original and reparametrised coefficients 
\eqref{eq:coeff_1to1}, for each $\bsy \in U$, and corresponding $\wbsy = \bsalpha \bsy\in U(\bsalpha)$,
\eqref{eq:coeff_1to1} implies
that there is equality between eigenvalues
$\lambda_k(\bsy)$ of \eqref{eq:varevp} and  
$\wlam_k(\wbsy)$ of 
the reparametrised eigenvalue problem \eqref{eq:reparam_evp}
\begin{equation}
\label{eq:eval_1to1}
\lambda_k(\bsy) \,=\, \widetilde{\lambda}_k(\wbsy)
\quad \text{for } k \in \N\,,
\end{equation}
and their eigenspaces coincide.

Moreover, for an eigenvalue $\wlam_k(\wbsy)$ \eqref{eq:reparam_evp}, 
using \eqref{eq:eval_1to1} in the inequality \eqref{eq:eval_cts_coeff} we have
\[
|\widetilde{\lambda}_k(\widetilde{\bsy}) - \widetilde{\lambda}_k(\widetilde{\bsy}')|
\,\leq\, C \max\left(\nrm{a(\wbsy/\bsalpha) - a(\wbsy/\bsalpha)}{L^\infty(D)},\,
\nrm{b(\wbsy/\bsalpha) - b(\wbsy'/\bsalpha)}{L^\infty(D)}\right)\,,
\]
which after expanding the coefficients and using the triangle inequality becomes
\[
|\widetilde{\lambda}_k(\widetilde{\bsy}) - \widetilde{\lambda}_k(\widetilde{\bsy}')|
\,\leq\,
\underbrace{C\Bigg(\sum_{j = 1}^\infty 
\frac{1}{\alpha_j}\max\left(\nrm{a_j}{L^\infty(D)}, \ \nrm{b_j}{L^\infty(D)}\right)\Bigg)}_{\wCLip}
\nrm{\wbsy - \wbsy'}{\ell^\infty}\,,
\]
where $\wCLip$ is clearly independent of $\wbsy$ and $\wbsy'$.
Now by \eqref{eq:defalpha} together with Assumption~A\ref{asm:coeff}, 
we have 
\[
\sumj \frac{\Vert a_j\Vert_{L^\infty(D)}}{\alpha_j} \ \leq \ \sumj  \Vert a_j\Vert_{L^\infty(D)}^{1-\eps} \ = \  \sumj  \Vert a_j\Vert_{L^\infty(D)}^{p} <\  \infty \,,
\]
with the analogous estimate for $ \sumj \Vert b_j\Vert_{L^\infty(D)}/\alpha_j$. 
Thus, $\wCLip < \infty$ and hence the reparametrised eigenvalues 
are continuous on $U(\bsalpha)$.

It immediately follows that the spectral gap 
$\widetilde{\lambda}_2(\widetilde{\bsy}) - \widetilde{\lambda}_1(\widetilde{\bsy})$
is also continuous on $U(\bsalpha)$, which by Lemma~\ref{lem:compact_subset} is a 
compact subset of $\ell^\infty$. Therefore, the non-zero
minimum is attained giving that the spectral gap  
$\widetilde{\lambda}_2(\widetilde{\bsy}) - \widetilde{\lambda}_1(\widetilde{\bsy})$ 
is uniformly positive.
Finally, because there is equality between the original 
and reparametrised eigenvalues \eqref{eq:eval_1to1}
the result holds for the original problem over all $\bsy \in U$.
\end{proof}

\begin{Remark} 
\normalfont
An explicit bound on the spectral gap can be obtained by assuming
tighter restrictions on the coefficients.
For example, if $a \equiv 1 \equiv c $ and $b$ is weakly convex then
\cite{AC11} gives an explicit lower bound on the fundamental gap.
Alternatively, using the upper and lower bounds on the eigenvalues \eqref{eq:lambda_bnd},
we can determine restrictions on $\amin$ and $\amax$ such that the gap is bounded away from 0.
Explicitly, if
\[
\frac{\amin}{\amax} \,>\,
\sqrt{\frac{\chi_1 + 1}{\chi_2}}\,,
\]
then $\lambda_2(\bsy) - \lambda_1(\bsy) \geq \underline{\lambda_2} - \lambdabar > 0$.
\end{Remark} 

\subsection{Finite element discretisation}

To approximate eigenpairs $(\lambda(\bsy), u(\bsy))$ we introduce a
collection of finite element (FE) subspaces $V_h \subset V$ with dimension
$M_h$, each of which is associated with a conforming triangulation $\mathscr{T}_h$ of the domain $D$
and a  basis $(\phi_{h,i})_{i = 1}^{M_h}$. 
The parameter $h = \max\{ \diam(\tau) : \tau \in \mathscr{T}_h\}$ is called the \emph{meshwidth}.
The method works for very general spaces $V_h$, however to fully exploit
higher rates of convergence, stronger assumptions on the regularity of the coefficients 
and the domain would be required. 
As such, in the current paper we restrict our attention to piecewise linear finite elements, that is, 
each $V_h$ is the space of continuous functions that are linear on the elements of $\mathscr{T}_h$ 
and vanish on the boundary $\partial D$.
It is well-known that, for $t \in (0, 1]$, the best approximation error for the space $Z^t$
(as defined in \eqref{eq:Z^t} and \eqref{eq:Z_nrm}) by functions in $V_h$ satisfies
\begin{align}
\label{eq:fe_approx}
\inf_{w_h \in V_h} \nrm{v - w_h}{V} \,\leq\, Ch^t\nrm{v}{Z^t} 
\quad\text{for all } v \in Z^t\, .
\end{align}

For each $\bsy \in U$ the discrete eigenvalue problem is to find
$\lambda_h(\bsy) \in \R$ and $u_h(\bsy) \in V_h$ satisfying
\begin{align}\label{eq:discrete_evp}
\nonumber\calA(\bsy; u_h(\bsy), v) \,&=\, \lambda_h(\bsy)\calM( u_h(\bsy), v) \quad \text{for all } v \in V_h\,,\\
\nrm{u_h(\bsy)}{\calM} \,&=\, 1\,.
\end{align}
For each $\bsy \in U$,  the discrete eigenvalue problem \eqref{eq:discrete_evp}
admits $M_h$ eigenvalues
\begin{align*}
0 \,<\, \lambda_{1, h}(\bsy) \,\leq\, \lambda_{2, h}(\bsy) \,\leq\, \cdots \,\leq\,
\lambda_{M_h, h}(\bsy)\,,
\end{align*}
with corresponding eigenfunctions
\begin{align*}
u_{1, h}(\bsy),\, u_{2, h}(\bsy),\, \ldots, u_{M_h, h}(\bsy) \,\in\, V_h\,.
\end{align*}
For each fixed $k$, the $k$th finite element eigenvalue $\lambda_{k, h}(\bsy)$
converges from above to the $k$th eigenvalue of \eqref{eq:varevp}, i.e., for each $k$,   
%
\begin{align*}
\lambda_{k, h}(\bsy) \,\rightarrow\, \lambda_k(\bsy) 
\quad \text{as } h \to 0, 
\quad \text{with}\quad
\lambda_{k, h}(\bsy) \,\geq\, \lambda_k(\bsy)
\quad \text{for all } h  > 0, 
\end{align*}
and the corresponding FE eigenfunctions $(u_{k, h}(\bsy))_{k = 1}^{M_h}$
satisfy  
\begin{align*}
\dist_V(u_{k, h}(\bsy), E(\bsy, \lambda_k(\bsy)))
\,\rightarrow\, 0
\quad \text{as } h \to 0, 
\end{align*}
where $\dist_V(v, \calP)$ is the distance of $v \in V$ from
the subspace $\calP \subset V$
\begin{align*}
\dist_V(v, \calP)
\,\coloneqq\,
\inf_{w \in \calP} \nrm{v - w}{V}\,.
\end{align*}

The classical results on FE error estimates for eigenproblems are found in 
\cite{BO87,BO89,BO91};
however, we cannot simply use these results verbatim since their constants 
depend in complex and often hidden ways on the eigenvalues and eigenvalue gaps. 
For us, this means that they depend on $\bsy$, and so care must be taken to ensure
that the constants do not become unbounded for some $\bsy \in U$.
Theorem~\ref{thm:fe_err} below quantifies the FE convergence and in the proof
we show that all constants are independent of $\bsy$.
The proof is rather long and technical, and as such is left for the appendix.

\begin{theorem}
\label{thm:fe_err}
Let $\bsy \in U$ and suppose that Assumption~A\ref{asm:coeff} holds.
Then for $h > 0$ sufficiently small 
\begin{align}
\label{eq:lambda_fe_err}
|\lambda_1(\bsy) - \lambda_{1, h}(\bsy)| \,&\leq\, C_1h^{2} \, ,
\end{align} 
and $u_{1, h}(\bsy) \in E(\bsy, \lambda_{1, h}(\bsy))$ can be chosen such that
\begin{align}
\label{eq:u_fe_err}
\nrm{u_1(\bsy) - u_{1, h}(\bsy)}{V} \,&\leq\, C_2h \,.
\end{align}
Moreover, for $\calG \in H^{-1 + t}(D)$ with $t \in [0, 1]$
\begin{align}
\label{eq:Gu_fe_err}
\left|\calG(u_1(\bsy)) - \calG(u_{1, h}(\bsy))\right|
\,\leq\, C_3 h^{1 + t}\,,
\end{align}
and $C_1, C_2, C_3 > 0$ are independent of $\bsy$.  
\end{theorem}

\subsection{Quasi-Monte Carlo methods}\label{sec:qmc}
In this section, we give a brief overview of quasi-Monte Carlo
(QMC) rules and the analysis of the resulting integration error. 
For more details on QMC point sets and theory, see \cite{DKS13}.

QMC methods are equal-weight quadrature rules that  can be used for the approximation
of integrals over the (translated) $s$-dimensional unit cube
\begin{align*}
\calI_sf \,\coloneqq\, \int_{[-\frac{1}{2}, \frac{1}{2}]^s} f(\bsy) \,\rd \bsy \,,
\end{align*}
where the dimensionality $s$ is high.

In this paper we use a class of QMC rules called randomly shifted
rank-1 lattice rules, where the points
are constructed using a \emph{generating vector} $\bsz \in \N^s$ and a
\emph{shift} $\bsDelta$, which is uniformly distributed on $[0, 1]^s$.
Specifically, we have
\begin{equation*}
Q_{N,s}f \,\coloneqq\,
 \frac{1}{N} \sum_{k = 0}^{N - 1} f\left(\left\{\frac{k\bsz}{N} + \bsDelta\right\} - \boldsymbol{\tfrac{1}{2}}\right) \,,
\end{equation*}
where the braces denote taking the fractional part of each component and we have
subtracted the vector $\boldsymbol{\tfrac{1}{2}} \coloneqq (\tfrac{1}{2}, \ldots, \tfrac{1}{2})$
to map the points from $[0, 1]^s$ to $[-\tfrac{1}{2}, \tfrac{1}{2}]^s$.
In practice, the advantages of random shifting are threefold: the final approximation
is an unbiased estimate of the integral; using multiple shifts provides a practical
estimate of root-mean-square (RMS) error; and the construction of a good lattice rule
is simplified by the randomisation.

The error analysis of randomly shifted lattice rules requires
the integrand belong to a weighted Sobolev space such as one of those first
introduced in \cite{SW98}. The $s$-dimensional weighted Sobolev space,
denoted by $\calW_{s, \bsgamma}$, is the space of functions with
square-integrable mixed first derivatives and a norm which depends on a
family of positive real numbers called weights. For each $\setu \subseteq
\{1, \ldots, s\}$, the weight, denoted $\gamma_\setu$, measures
the ``importance'' of the subset of variables $y_j$ with $j \in \setu$.
We let the entire collection of weights be denoted by $\bsgamma$.

In this paper we equip $\calW_{s, \bsgamma}$ with the ``unanchored''
weighted norm.  To define it  we require the following notation: let
$\{1:s\} \coloneqq \{1, \ldots, s\}$, $\bsy_\setu \coloneqq (y_j)_{j \in
\setu}$, $\bsy_{-\setu} \coloneqq (y_j)_{j \in \{1:s\} \setminus \setu}$
and let $\partial^{|\setu|}/\partial \bsy_\setu \coloneqq \prod_{j\in\setu}
(\partial/\partial y_j)$ denote the first order mixed partial derivative
with respect to the variables $\bsy_\setu$. Now, let the norm (squared) of
$f \in \calW_{s, \bsgamma}$ be given by
\begin{align}
\label{eq:nrm_W}
\nrm{f}{s, \bsgamma}^2 \,=\, \sum_{\setu \subseteq \{1:s\}}
\frac{1}{\gamma_\setu} \int_{[-\frac{1}{2},\frac{1}{2}]^{|\setu|}}
\left(\int_{[-\frac{1}{2},\frac{1}{2}]^{s - |\setu|}} \pd{|\setu|}{}{\bsy_\setu} f(\bsy) \, \rd \bsy_{-\setu} \right)^2 \, \rd \bsy_\setu \, .
\end{align}

Good generating vectors $\bsz$ can be efficiently constructed using the
\emph{Fast CBC algorithm}, see \cite{SKJ02b} for CBC and 
\cite{NC06,NC06np} for ``its acceleration'' (or ``fast CBC''). 
It has been shown, see e.g. \cite[Theorem 5.10]{DKS13}, that
the RMS error of such a QMC approximation is bounded above by
\begin{align}
\label{eq:msebound}
&\sqrt{\bbE_\bsDelta\left(\left|\calI_s f - Q_{N, s}
f \right|^2\right)} \nonumber\\
&\,\leq\,
\left(\frac{1}{\varphi(N)}
\sum_{\emptyset\neq \setu \subseteq \{1:s\}} \gamma_\setu^\eta
\left(\frac{2\zeta(2\eta)}{(2\pi^2)^\eta}\right)^{|\setu|}
\right)^\frac{1}{2\eta}\,
\nrm{f}{s, \bsgamma}
\quad \text{for all } \eta \in (\tfrac{1}{2}, 1]\,,
\end{align}
where the subscript $\bsDelta$ in $\bbE_\bsDelta$ indicates that the
expectation is taken with respect to the (uniformly distributed) random shift, $\varphi(N) \coloneqq
|\{1\le \xi \le N : \gcd(\xi,N)=1\}|$ is the Euler totient function, and
$\zeta(x) \coloneqq \sum_{k = 1}^\infty k^{-x}$ for $x>1$ is the Riemann
zeta function. In particular, if $N$ is prime then $\varphi(N) = N - 1$.

\section{Parametric regularity}
\label{sec:parametric}

In this section we examine the regularity with respect to $\bsy$ of the minimal eigenpair
$(\lambda_{1}(\bsy), u_1(\bsy))$ of the variational eigenproblem
\eqref{eq:varevp}. The
results we obtain show that $\lambda_1(\bsy)$ belongs to the weighted
space $\calW_{s, \bsgamma}$ with norm defined in \eqref{eq:nrm_W}. 
This is required for the analysis of the QMC error in approximating 
$\bbE_\bsy[\lambda_1]$. Also, to obtain an \textit{a priori} bound
on the QMC  error we require a bound on the norm of $\lambda_1(\bsy)$ in
$\calW_{s, \bsgamma}$, hence we must bound its mixed first derivatives,
see Lemma~\ref{lem:dlambda}.
There we use the bounds on the spectral gap obtained in \S \ref{subsec:gap},
and present results not only for the first-order mixed derivatives but
also for higher-order derivatives.

We begin with the following coercive-type estimate, which is required in order to bound
the norm of the derivatives of the eigenfunction.

\begin{Lemma}\label{lem:coerc}
Let Assumption~A\ref{asm:coeff} hold. Then
for each $\bsy \in U$ and $\lambda \in \R$, define
$\Ash{\lambda}(\bsy; \cdot, \cdot) :V \times V \rightarrow \R$
to be the shifted bilinear form given by
\begin{align}
\Ash{\lambda}(\bsy; u, v) \,\coloneqq\, \calA(\bsy; u, v) - \lambda\calM( u, v) \, ,
\end{align}
with $\calA$ and $\calM$ defined by \eqref{eq:bilinear} and
\eqref{eq:M_inner}, respectively.
Restricted to the
$\calM$-orthogonal complement of the eigenspace corresponding to
$\lambda_1(\bsy)$, denoted by $E(\bsy, \lambda_1(\bsy))^\perp$, 
the $\lambda_1(\bsy)$-shifted bilinear form is uniformly coercive
in $\bsy$, i.e.,
%
\begin{align}\label{eq:Ashcoerc}
\Ash{\lambda_1(\bsy)}(\bsy; v, v) \,\geq\, 
\Cgap
\nrm{v}{V}^2 \quad \text{for all } v \in E(\bsy, \lambda_1(\bsy))^\perp \,,
\end{align}
where, for $\delta$ as in Proposition~\ref{prop:simple},
\begin{equation*}
\Cgap \,\coloneqq\, \frac{a_{\min}\delta}{\overline{\lambda_2}} \,.
\end{equation*}
\end{Lemma}

\begin{proof}
Since the eigenfunctions $(u_k(\bsy))_{k \in \N}$ form a basis
in $V$ that is orthonormal with respect to the inner product $\calM$, 
for $v \in E(\bsy, \lambda_1(\bsy))^\perp$, letting $v_k(\bsy) \coloneqq
\calM( v, u_k(\bsy)) u_k(\cdot, \bsy)$ for $k = 1, 2, \ldots$, we can
write
\begin{align*}
v \,=\, \sum_{k = 2}^\infty v_k(\bsy) \,,
\end{align*}
where we have used $v_1(\bsy) = 0$ since $\calM( v, u_1(\bsy)) = 0$.
Henceforth, we will suppress the dependence of the eigenvalues and $v_k$ on $\bsy$.
For $v \in E(\bsy, \lambda_1)^\perp$ we have
\begin{align*}
\Ash{\lambda_1}(\bsy; v, v) \,&=\,
\Ash{\lambda_1}\left(\bsy; \sum_{k = 2}^\infty v_k, \sum_{\ell = 2}^\infty v_\ell \right)
\,=\,\sum_{k, \ell = 2}^\infty \left(\calA(\bsy;v_k, v_\ell) - \lambda_1 \calM( v_k, v_\ell )\right)\,.
\end{align*}
Since all $v_k$ are just scaled versions of $u_k$, they also satisfy
the variational equation \eqref{eq:varform}, so that $\calA(\bsy; v_k, v_\ell) =
\lambda_k \calM( v_k,v_\ell)$ and they are orthogonal with respect to $\calM(\cdot, \cdot)$,
implying that $\calA(\bsy; v_k, v_\ell) = 0$ for $k \neq \ell$. Thus we can reduce
the above double sum to
\begin{align*}
\calA^\mathrm{sh}_{\lambda_1}(\bsy;v, v)
 \,&=\, \sum_{k = 2}^\infty \left(\calA(\bsy;v_k, v_k) - \frac{\lambda_1}{\lambda_k} \calA(\bsy;v_k, v_k)  \right) \\
 &\geq\, \left(1 - \frac{\lambda_1}{\lambda_2}\right)\sum_{k = 2}^\infty \calA(\bsy;v_k, v_k) 
\,=\,  \left(1 - \frac{\lambda_1}{\lambda_2}\right)\sum_{k, \ell = 2}^\infty \calA(\bsy;v_k, v_\ell)
 \\
&=\, \left(1 - \frac{\lambda_1}{\lambda_2}\right) \calA(\bsy;v, v)
\,\geq\, a_{\min} \left(1 - \frac{\lambda_1}{\lambda_2}\right) \nrm{v}{V}^2\,.
\end{align*}
The final bound, which is independent of $\bsy$, follows from
Proposition~\ref{prop:simple} and \eqref{eq:lambda_bnd}.
\end{proof}

\begin{Remark}\normalfont
A similar estimate holds for the shifted bilinear form on $V_h\times V_h$, provided
$h$ is sufficiently small such that the FE eigenvalue gap is uniformly bounded from below.
Indeed, we can write
\[
  \lambda_{2,h} - \lambda_{1,h} \,=\, (\lambda_{2, h} - \lambda_1) - (\lambda_{1, h} - \lambda_{1})\,,
\]
and since the FE eigenvalues converge from above we can bound this from below by
\[
  \lambda_{2,h} - \lambda_{1,h} \,\geq\, \lambda_2 - \lambda_1 - \left|\lambda_{1} - \lambda_{1, h}\right|
  \,\geq\, \delta - Ch^{2}\,.
\]
with $C > 0$.
The second inequality follows from Proposition~\ref{prop:simple} and
Theorem~\ref{thm:fe_err}.
Thus, choosing $h$ such that $Ch^{2} <\delta$, or equivalently, taking $h < h_0$ with
\begin{align}
\label{eq:h0}
h_0 \, \coloneqq\,
\left(\frac{\delta}{C}\right)^{\frac{1}{2}}\,,
\end{align}
is a sufficient condition for $\lambda_{2,h} - \lambda_{1,h} > 0$, and
then Lemma~\ref{lem:coerc} can be rewritten for the FE eigenproblem.
\end{Remark}

We also require the following technical lemma to handle some
combinatorial factors that arise when bounding the derivatives.

\begin{Lemma}\label{lem:S_n}
Let $\epsilon \in (0, 1)$. For all $n \in \N$, the following bound holds:
\begin{align*}
S_n(\epsilon) \,\coloneqq\, \sum_{k = 1}^{n-1} \binom{n}{k}^{-\epsilon} \,\leq\, C_\epsilon
\,\coloneqq\, \frac{2^{1 - \epsilon}}{1 - 2^{-\epsilon}}
\left(\frac{e^2}{\sqrt{2\pi}} \right)^\epsilon\,.
\end{align*}
\end{Lemma}

\begin{proof}
By the symmetry of the binomial coefficient the sum can be bounded by
\begin{align*}
S_n(\epsilon) \,&\leq\, 2 \sum_{k = 1}^{\lfloor\frac{n}{2}\rfloor} \binom{n}{k}^{-\epsilon}
\,\leq\, 2\left(\frac{e^2}{\sqrt{2\pi}} \right)^\epsilon
\sum_{k = 1}^{\lfloor\frac{n}{2}\rfloor}
 \left(\frac{k^{k + \frac{1}{2}} (n -k)^{n - k + \frac{1}{2}}}{n^{n + \frac{1}{2}}}\right)^\epsilon\,,
\end{align*}
where we used the following bounds given by Stirling's formula
\begin{align*}
\sqrt{2\pi} n^{n + \frac{1}{2}} e^{-n}
\,\leq\, n! \,\leq\,
e n^{n + \frac{1}{2}} e^{-n}\,.
\end{align*}
Since $k \leq \frac{n}{2}$ we have the bound
$k^{k + \frac{1}{2}} \leq (\frac{n}{2})^k\sqrt{k}$, which gives
\begin{align*}
S_n(\epsilon) \,&\leq\, 2\left(\frac{e^2}{\sqrt{2\pi}} \right)^\epsilon
\sum_{k = 1}^{\lfloor\frac{n}{2}\rfloor}
\left(\frac{\left(\frac{n}{2}\right)^k\sqrt{k} (n -k)^{n - k + \frac{1}{2}}}
{n^{n + \frac{1}{2}}}\right)^\epsilon\\
&=\, 2\left(\frac{e^2}{\sqrt{2\pi}} \right)^\epsilon
\sum_{k = 1}^{\lfloor\frac{n}{2}\rfloor} \frac{1}{2^{\epsilon k}}
\bigg(\sqrt{k} \left(1 - \frac{k}{n}\right)^{n - k + \frac{1}{2}}
\bigg)^\epsilon\,.
\end{align*}

The next step is to show that for each term in the sum the factor occurring
inside the brackets is always bounded above by 1.
To proceed, for $n = 1, 2, 3, \ldots$ we define the functions 
$R_n: [0, 1/2] \rightarrow \R$ by
\begin{align*}
R_n(x) \,\coloneqq\,
\sqrt{n x}\left(1 - x\right)^{n(1 - x) + \frac{1}{2}}\,.
\end{align*}
We prove by induction that for $n = 1, 2, 3, \ldots$
\begin{align}
R_n(x) \,\leq\, 1
\quad \text{for all } x\in [0, \tfrac{1}{2}]\,.
\label{eq:Rn<1}
\end{align}
For $n = 1$
\begin{align*}
R_1(x) \,=\, \sqrt{x}(1 - x)^{1 - x + \frac{1}{2}}
\,\leq\,
\sqrt{\frac{1}{2}}(1 - x)^{1 - x + \frac{1}{2}} \,\leq\, 1.
\end{align*}
For $n \geq 1$ suppose $R_n(x) \leq 1$ and consider $R_{n + 1}$.
For $x$ in the interval $[1/(n + 1), 1/2]$
\begin{align*}
R_{n + 1}(x) \,&=\, \sqrt{(n + 1)x}(1 - x)^{(n + 1)(1 - x) + \frac{1}{2}}
\,=\, \sqrt{\frac{n + 1}{n}} \sqrt{n x} (1 - x)^{1 - x} (1 - x)^{n(1 - x) + \frac{1}{2}}\\
&=\, (1 - x)^{1 - x} \sqrt{\frac{n + 1}{n}} R_n(x)
\,\leq\, (1 - x)^{1 - x} \sqrt{\frac{n + 1}{n}}\,.
\end{align*}
To bound this from above we bound one $x$ below by $1/(n + 1)$ 
to give
\begin{align*}
R_{n + 1}(x) \,&\leq\, \left(1 - \frac{1}{n + 1}\right)^{1 - x}\sqrt{\frac{n + 1}{n}}
\,=\, \left(\frac{n}{n + 1}\right)^{\frac{1}{2} - x}
\,\leq\, 1\,.
\end{align*}
And for $x \in [0, 1/(n + 1)]$
\begin{align*}
R_{n + 1}(x) \,\leq\, \sqrt{(n + 1)\frac{1}{n + 1}} (1 - x)^{(n + 1)(1 - x) + \frac{1}{2}}
\,=\, (1 - x)^{(n + 1)(1 - x) + \frac{1}{2}}\,\leq\, 1\,.
\end{align*}
Thus, for all $n = 1, 2, 3, \ldots$ and $x \in [0, 1/2]$ we have
$R_n(x) \leq 1$.

Returning to the sum $S_n$, since
\begin{align*}
\frac{k}{n} \in [0, \tfrac{1}{2}]
\quad \text{for all }
k \leq \frac{n}{2}\,,
\quad \text{and} \quad
R_n\left( \tfrac{k}{n}\right) \,=\, \sqrt{k} \left(1 - \frac{k}{n}\right)^{n - k + \frac{1}{2}}
\end{align*}
we have,  by \eqref{eq:Rn<1},
\begin{align*}
S_n(\epsilon) \,&\leq\,
2\left(\frac{e^2}{\sqrt{2\pi}} \right)^\epsilon
\sum_{k = 1}^{\lfloor\frac{n}{2}\rfloor} \frac{R_n\left(\frac{k}{n}\right)^\epsilon}{(2^\epsilon)^k}
\,\leq\,
2\left(\frac{e^2}{\sqrt{2\pi}} \right)^\epsilon
\sum_{k = 1}^{\lfloor\frac{n}{2}\rfloor} \frac{1}{(2^\epsilon)^k}\\
&\leq\,
2\left(\frac{e^2}{\sqrt{2\pi}} \right)^\epsilon
\sum_{k = 1}^\infty \frac{1}{(2^\epsilon)^k}
\,=\,
\frac{2^{1 - \epsilon}}{1 - 2^{-\epsilon}}
\left(\frac{e^2}{\sqrt{2\pi}} \right)^\epsilon
\,\eqqcolon\, C_\epsilon\,,
\end{align*}
where we used the formula for the sum of a geometric series.
\end{proof}

Lemma~\ref{lem:dlambda} below gives the bounds on the derivatives of
$\lambda_1$ and $u_1$ required for our QMC error analysis. We prove the
bounds for higher order mixed derivatives, which will be written in
multi-index notation. Let $\bsnu = (\nu_j)_{j \in \N}$, with $\nu_j \in \N
\cup \{0\}$, be a multi-index with only finitely many non-zero entries and
define $|\bsnu| \coloneqq \sum_{j \geq 1} \nu_j$. 
We call such a multi-index \emph{admissible}, and
let $\indx$ denote the set of all admissible multi-indexes. We will use
$\pdy{\bsnu}$ to denote the mixed partial derivative where the element
$\nu_j$ is the order of the derivative with respect to $y_j$. Operations
between multi-indices are handled component wise. Thus, for $\bsm =
(m_j)_{j \in \N}, \bsnu = (\nu_j)_{j \in \N}$ we use the following
notation: $\bsnu! = \prod_{j \geq 1} \nu_j!$; $\bsnu - \bsm \coloneqq (\nu_j - m_j)_{j \in \N}$;
$\bsm \leq \bsnu$ if $m_j \leq \nu_j$ for all $j \in \N$; $\bsm < \bsnu$ if 
$\bsm \leq \bsnu$ and $\bsm \neq \bsnu$; 
$\binom{\bsnu}{\bsm} \coloneqq \prod_{j \in \N} \binom{\nu_j}{m_j}$;
and for $\bsalpha = (\alpha_j)_{j \in \N} \in \ell^\infty$ let $\bsalpha^\bsnu \coloneqq \prod_{j \in \N} \alpha_j^{\nu_j}$.
For $j \in \N$, the $j$th unit multi-index is denoted by $\bse_j$, that is, $\bse_j$ is 1 in the $j$th
position and 0 everywhere else.

Since the coefficients $a(\bsx,\bsy)$ and $b(\bsx,\bsy)$ in
\eqref{eq:a_general} are linear in the parameter~$\bsy$, their
derivatives are (suppressing the $\bsx$, $\bsy$ dependence below)
\begin{align}
\label{eq:da}
\pdy{\bsnu} a \,&=\, \begin{cases}
a & \text{if } \bsnu = \bs0\, ,\\
a_j & \text{if } \bsnu = \bse_j\, ,\\
0 & \text{otherwise,}
\end{cases}
\quad\text{and}\quad
\pdy{\bsnu} b \,=\, \begin{cases}
b & \text{if } \bsnu = \bs0\, ,\\
b_j & \text{if } \bsnu = \bse_j\, ,\\
0 & \text{otherwise.}
\end{cases}
\end{align}

\begin{Lemma}\label{lem:dlambda}
Let $\epsilon \in (0, 1)$, $\bsnu \in \indx$ be a multi-index, and
suppose that Assumption~A\ref{asm:coeff} holds. 
Then for all $\bsy \in U$ the corresponding derivative of the smallest eigenvalue of 
\eqref{eq:varevp} is bounded by
\begin{align}
\label{eq:dlambda}
|\pdy{\bsnu}\lambda_1(\bsy)| \,\leq\, \lambdabar\,
(|\bsnu|!)^{1 + \epsilon} \, \bsbeta^\bsnu \,,
\end{align}
and the norm of the derivative of the corresponding eigenfunction is
similarly bounded by
\begin{align}
\label{eq:du}
\nrm{\pdy{\bsnu}u_1(\bsy)}{V} \,\leq\, \ubar \,
(|\bsnu|!)^{1 + \epsilon} \, \bsbeta^\bsnu \,,
\end{align}
where $\lambdabar$ and $\ubar$ are defined in \eqref{eq:lambda_bnd}
and \eqref{eq:u_bnd}, respectively.
The sequence $\bsbeta = (\beta_j)_{j \in \N}$ is defined by
\begin{align}
\label{eq:beta}
\beta_j \,\coloneqq\, 
C_\bsbeta\,\max\left(\nrm{a_j}{L^\infty(D)},\nrm{b_j}{L^\infty(D)}\right) \,,
\end{align}
with $C_\bsbeta > 0$ independent of $\bsy$, given by 
\begin{align} \label{eq:Cbeta}
 C_\bsbeta \,\coloneqq\,
 \frac{1}{\Cgap}\frac{\amin^2\,\lambdabar}{\amax^2 \,\lambdaunder}
\left(\frac{3\lambdabar}{2\lambdaunder}C_\epsilon + 1\right)\,,
\end{align}
where $\Cgap$ is as in Lemma~\ref{lem:coerc}, and $C_\epsilon$ is as in Lemma~\ref{lem:S_n}.
\end{Lemma}

\begin{proof}
For $\bsnu = \bs0$ the bounds
\eqref{eq:dlambda} and \eqref{eq:du} clearly hold by
\eqref{eq:lambda_bnd} and \eqref{eq:u_bnd}, respectively.

For $\bsnu \ne \bs0$, we will prove the two bounds by induction
on~$|\bsnu|$.
To this end, we first obtain recursive bounds by differentiating the
variational form \eqref{eq:varform} with respect to the stochastic
parameters $\bsy\in U$, see \eqref{eq:dl1_rec} and \eqref{eq:du1_rec},
which will then be used to prove \eqref{eq:dlambda} and \eqref{eq:du}
inductively.
From \cite{AS12}, we know that simple eigenpairs of
\eqref{eq:varevp} are analytic in $\bsy$, so the partial derivatives
$\pdy{\bsnu}\lambda_1$, $\pdy{\bsnu} u_1$ exist and we further have
$\pdy{\bsnu} u_1 \in V$. Hence, we can differentiate \eqref{eq:varform}
with $\lambda = \lambda_1$ and $u = u_1$ using the Leibniz general product
rule to obtain the following formula, which is true for all $v \in V$,
\begin{align}\label{eq:dvarform}
\nonumber\sum_{\bsm \leq \bsnu} \binom{\bsnu}{\bsm} \bigg(
&-\, \left(\pdy{\bsm}\lambda_1(\bsy)\right) \int_D  c(\bsx)\left(\pdy{\bsnu - \bsm} u_1(\bsx, \bsy)\right) v(\bsx) \,\rd \bsx\\
\nonumber&+\,\int_D \left(\pdy{\bsm}a(\bsx, \bsy)\right)
\nabla \left(\pdy{\bsnu - \bsm}u_1(\bsx, \bsy)\right) \cdot \nabla v(\bsx) \,\rd \bsx \\
&+\,\int_D \left(\pdy{\bsm}b(\bsx, \bsy)\right)\left(\pdy{\bsnu - \bsm}u_1\right)
v(\bsx)\,\rd \bsx\bigg)  \,=\, 0\, .
\end{align}

Henceforth, we consider $\bsy \in U$ to be fixed and will suppress the dependence
of $a(\bsx, \bsy)$, $b(\bsx, \bsy)$, $c(\bsx)$, $\lambda_1(\bsy)$, $u_1(\bsx, \bsy)$,
and their respective derivatives, on $\bsx$ and $\bsy$.

To obtain a bound on the derivatives of the eigenvalue, we take $v =
u_1$ in \eqref{eq:dvarform}. In this case, the $\bsm = \bs0$ term
vanishes since \eqref{eq:varform} is satisfied for $\lambda = \lambda_1$ and 
$u = u_1$ with $\pdy{\bsnu} u_1 \in V$ as a test function. Separating out the
$\pdy{\bsnu}\lambda_1$ term and using $\nrm{u_1}{\calM} = 1$ gives
\begin{align}
\label{eq:dlam_exact}
\pdy{\bsnu}\lambda_1 \,=\,
&- \sum_{\bs0 \neq \bsm < \bsnu} \binom{\bsnu}{\bsm} (\pdy{\bsm}\lambda_1)
\int_D c \left(\pdy{\bsnu - \bsm} u_1\right) u_1
\nonumber\\
&+ \sum_{j = 1}^\infty \nu_j \bigg(\int_D a_j \nabla \left( \pdy{\bsnu - \bse_j}u_1\right)\cdot \nabla u_1
\,+\,\int_D b_j \left(\pdy{\bsnu - \bse_j}u_1\right) u_1
 \bigg)\,,
\end{align}
since the terms involving higher-order derivatives of the
coefficients are 0 (see \eqref{eq:da}).
Taking the absolute value then applying the triangle and Cauchy--Schwarz  inequalities
gives the upper bound
\begin{align*}
&|\pdy{\bsnu}\lambda_1| \,\leq\,
\sum_{\bs0 \neq \bsm < \bsnu} \binom{\bsnu}{\bsm}
|\pdy{\bsm}\lambda_1|
\nrm{\pdy{\bsnu - \bsm}u_1}{\calM}\nrm{u_1}{\calM}  \\
&+
\sum_{j = 1}^\infty \nu_j \left(
\nrm{a_j}{L^\infty(D)}\nrm{\pdy{\bsnu - \bse_j}u_1}{V} \nrm{u_1}{V}
 + \nrm{b_j}{L^\infty(D)}
 \nrm{\pdy{\bsnu - \bse_j}u_1}{L^2(D)} \nrm{u_1}{L^2(D)}\right)
\,.
\end{align*}
Then, by the equivalence of the norms $\nrm{\cdot}{\calM}$ and
$\nrm{\cdot}{L^2(D)}$ in \eqref{eq:nrms_equiv}, the Poincar\'e
inequality \eqref{eq:poin}, the upper bound \eqref{eq:u_bnd} on
$\nrm{u_1}{V}$, and the normalisation of $u_1$, we have
\begin{align*}
|\pdy{\bsnu}\lambda_1| \,&\leq\,
\sqrt{\frac{\amax}{\evalueLap}} \sum_{\bs0 \neq \bsm < \bsnu} \binom{\bsnu}{\bsm}
|\pdy{\bsm}\lambda_1|
\nrm{\pdy{\bsnu - \bsm}u_1}{V} \\
&+ \ubar
\sum_{j = 1}^\infty \nu_j \left(
\nrm{a_j}{L^\infty(D)}
 + \tfrac{1}{\evalueLap} \nrm{b_j}{L^\infty(D)}\right)
 \nrm{\pdy{\bsnu - \bse_j}u_1}{V}
\,.
\end{align*}
Defining $\beta_j$ as in \eqref{eq:beta} but leaving $C_\bsbeta>0$ to
be specified later, we obtain
\begin{align}
\label{eq:dl1_rec}
|\pdy{\bsnu}\lambda_1| \,&\leq\,
\sqrt{\frac{\amax}{\evalueLap}} \sum_{\bs0 \neq \bsm < \bsnu} \binom{\bsnu}{\bsm}
|\pdy{\bsm}\lambda_1|
\nrm{\pdy{\bsnu - \bsm}u_1}{V} 
+ \ubar
\left(1 + \tfrac{1}{\evalueLap}\right)
\sum_{j = 1}^\infty \nu_j \frac{\beta_j}{C_\bsbeta}
 \nrm{\pdy{\bsnu - \bse_j}u_1}{V}\,,
\end{align}
which depends
only on the lower order derivatives of both $\lambda_1$ and $u_1$.

Substituting $v = \pdy{\bsnu}u_1$ into \eqref{eq:dvarform} to 
obtain a similar bound on the derivatives of the eigenfunction will not work, 
because in \eqref{eq:dvarform} the bilinear
form acting on $\pdy{\bsnu}u_1$ (exactly $\calA^{\mathrm{sh}}_{\lambda_1}$
from Lemma~\ref{lem:coerc}) is not coercive on the whole  domain $V \times
V$. The 
way around this is to expand $\pdy{\bsnu}u_1$ in the eigenbasis and then
utilise the estimate in Lemma~\ref{lem:coerc}. We write $\pdy{\bsnu}u_1$ as
\begin{align}
\label{eq:du1decomp}
\pdy{\bsnu}u_1 \,=\, \sum_{k \in \N} \calM( \pdy{\bsnu}u_1, u_k) u_k
\,=\, \calM( \pdy{\bsnu}u_1, u_1) u_1 + \tilde{v}\,,
\end{align}
so that $\tilde{v} \in E(\bsy, \lambda_1(\bsy))^\perp$ is the $\calM$-orthogonal
projection of $\pdy{\bsnu}u_1$ onto  $E(\bsy, \lambda_1(\bsy))^\perp$.
Applying the triangle inequality to this decomposition and then using
\eqref{eq:u_bnd} we can bound the norm by
\begin{align}
\label{eq:du1_normdecomp}
\nrm{\pdy{\bsnu}u_1}{V} \,&\leq\,
\ubar |\calM(\pdy{\bsnu}u_1, u_1) | + \nrm{\tilde{v}}{V}\, .
\end{align}

Hence, it remains to bound $\calM(\pdy{\bsnu} u_1, u_1)$ and
$\nrm{\tilde{v}}{V}$. For the former, since $\calM( u_1, u_1) = 1$ we have
\begin{align*}
0 \,=\, \pdy{\bsnu} \calM(u_1, u_1)
\,=\, \sum_{\bsm \leq \bsnu} \binom{\bsnu}{\bsm} \calM( \pdy{\bsm} u_1, \pdy{\bsnu - \bsm} u_1) \, .
\end{align*}
By separating out the $\bsm = \bs0$ and $\bsm = \bsnu$ terms, which are equal by symmetry,
we obtain
\begin{align}
\label{eq:<du,u>}
|\calM( \pdy{\bsnu} u_1, u_1)|
\,&=\,
\left|-\frac{1}{2} \sum_{\bs0 \neq \bsm < \bsnu} \binom{\bsnu}{\bsm}
\calM( \pdy{\bsm} u_1, \pdy{\bsnu - \bsm} u_1)\right| \\
&\leq\, \frac{\amax}{2\evalueLap}\sum_{\bs0 \neq \bsm < \bsnu} \binom{\bsnu}{\bsm} \nrm{\pdy{\bsm}u_1}{V} \nrm{\pdy{\bsnu - \bsm}u_1}{V},
\end{align}
where we used the Cauchy--Schwarz inequality, the norm equivalence \eqref{eq:nrms_equiv}
and then the Poincar\'e inequality \eqref{eq:poin} to obtain $V$-norms.

For the $V$-norm of $\tilde{v}$, we let $v = \tilde{v}$ in \eqref{eq:dvarform}
and separate out the $\bsm = \bs0$ term to give
\begin{align}
\label{eq:perpform}
\nonumber
\calA\left(\pdy{\bsnu}u_1, \tilde{v}\right)
- &\lambda_1\calM\left(\pdy{\bsnu}u_1, \tilde{v}\right)
\,=\, \sum_{\bs0 \neq \bsm < \bsnu} \binom{\bsnu}{\bsm} (\pdy{\bsm}\lambda_1) \int_D c\, (\pdy{\bsnu - \bsm} u_1) \tilde{v} \\
 &- \sum_{j = 1}^\infty \nu_j\bigg(\int_D a_j \nabla (\pdy{\bsnu - \bse_j}u_1) \cdot \nabla \tilde{v}
+ \int_D b_j (\pdy{\bsnu - \bse_j}u_1)\tilde{v} \bigg)
 \, ,
\end{align}
where the $\bsm = \bsnu$ term vanishes from the right-hand side since
$\tilde{v}$ is orthogonal to $u_1$. Decomposing $\pdy{\bsnu}u_1$ as in
\eqref{eq:du1decomp} and then expanding, the left-hand side of
\eqref{eq:perpform} becomes
\begin{align*}
\text{LHS of \eqref{eq:perpform}} \,&=\,
\calM( \pdy{\bsnu} u_1, u_1)
\big(\calA(u_1, \tilde{v}) - \lambda_1\calM( u_1, \tilde{v})\big)
+ \calA(\tilde{v}, \tilde{v}) - \lambda_1\calM(\tilde{v}, \tilde{v}) \\
&=\, \calA(\tilde{v}, \tilde{v}) - \lambda_1\calM(\tilde{v}, \tilde{v})
\,\geq\,C_{\rm gap} \nrm{\tilde{v}}{V}^2\,,
\end{align*}
where the first term on the first line is 0 by \eqref{eq:varform} with
$\tilde{v}$ as a test function. The lower bound follows by the coercivity
estimate \eqref{eq:Ashcoerc} in Lemma \ref{lem:coerc}, since $\tilde{v}
\in E(\bsy, \lambda_1(\bsy))^\perp$.

The right-hand side of \eqref{eq:perpform} can be bounded from above
as for \eqref{eq:dlam_exact} to obtain
\begin{align*}
C_{\rm gap} 
\nrm{\tilde{v}}{V}^2
&\leq \,   \frac{\amax}{\evalueLap}\sum_{\bs0 \neq \bsm < \bsnu} \binom{\bsnu}{\bsm}
|\pdy{\bsm} \lambda_1| \nrm{\pdy{\bsnu - \bsm} u_1}{V}\nrm{\tilde{v}}{V}\\
&+\, \sum_{j = 1}^\infty \nu_j
\left(\nrm{a_j}{L^\infty(D)} +  \tfrac{1}{\evalueLap}\nrm{b_j}{L^\infty(D)}\right)
\|\pdy{\bsnu - \bse_j}u_1\|_{V}\nrm{\tilde{v}}{V}
 \,.
\end{align*}
Dividing through by $C_{\rm gap}\nrm{\tilde{v}}{V}$ and
using the definition of $\beta_j$ in \eqref{eq:beta}, again leaving
$C_\bsbeta>0$ to be specified later, we obtain
\begin{align}
\label{eq:vtilde_Vnorm}
\nonumber \nrm{\tilde{v}}{V} \,\leq\, \frac{1}{C_{\rm gap}}
\Bigg(&\frac{\amax}{\evalueLap} \sum_{\bs0 \neq \bsm < \bsnu} \binom{\bsnu}{\bsm}
|\pdy{\bsm}\lambda_1| \nrm{\pdy{\bsnu - \bsm}u_1}{V} \\
&+\,
\left( 1 + \tfrac{1}{\evalueLap} \right)
\sum_{j = 1}^\infty \nu_j \frac{\beta_j}{C_\bsbeta} \nrm{\pdy{\bsnu - \bse_j}u_1}{V}\Bigg) \, .
\end{align}

Substituting the two bounds \eqref{eq:<du,u>} and
\eqref{eq:vtilde_Vnorm} into \eqref{eq:du1_normdecomp}, the norm of
the derivative of the eigenfunction is bounded above by
\begin{align}
\label{eq:du1_rec}
\nonumber\nrm{\pdy{\bsnu}u_1}{V} \,\leq\,
&
\ubar \frac{\amax}{2 \evalueLap}
\sum_{\bs0 \neq \bsm < \bsnu} \binom{\bsnu}{\bsm}
 \|\pdy{\bsm}u_1\|_V \|\pdy{\bsnu - \bsm}u_1\|_V  \\
\nonumber &+\,
\frac{\amax}{\evalueLap C_{\rm gap}}
\sum_{\bs0 \neq \bsm < \bsnu} \binom{\bsnu}{\bsm} |\pdy{\bsm}\lambda_1| \nrm{\pdy{\bsnu - \bsm}u_1}{V}\\
&+\, \frac{1}{C_{\rm gap}}\left(1 + \tfrac{1}{\evalueLap}\right)
\sum_{j = 1}^\infty \nu_j \frac{\beta_j}{C_\bsbeta} \nrm{\pdy{\bsnu - \bse_j}u_1}{V} \,.
\end{align}

We are now ready to prove the bounds \eqref{eq:dlambda} and
\eqref{eq:du} by induction. To avoid any blow-up in the inductive step we require
tighter constants than $\lambdabar$ and $\ubar$.
Thus we proceed to prove that, for $\bsnu \ne \bs0$,
\begin{align} \label{eq:dl1}
|\pdy{\bsnu}\lambda_1(\bsy)| \,&\leq\, C_1 \,
(|\bsnu|!)^{1 + \epsilon} \, \bsbeta^\bsnu \,,
\quad \text{and}\\
\label{eq:du1}
\nrm{\pdy{\bsnu}u_1(\bsy)}{V} \,&\leq\, C_2\,
(|\bsnu|!)^{1 + \epsilon} \, \bsbeta^\bsnu \,,
\end{align}
where
\begin{align*}
  C_1 &\,\coloneqq\,
   \frac{\lambdabar}{\amin}  
  \left(1 + \tfrac{1}{\evalueLap}\right)\frac{1}{C_\bsbeta}\,,
 \quad\mbox{and}\\
  C_2 &\,\coloneqq\,
   \frac{\ubar}{\Cgap}
 \left(1 + \tfrac{1}{\evalueLap}\right)\frac{1}{C_\bsbeta}\,,
\end{align*}
with $C_\bsbeta>0$ still to be specified later to ensure that $C_1\le
\lambdabar$ and $C_2\le \ubar$.

Since the bounds \eqref{eq:dl1_rec} and
\eqref{eq:du1_rec} are true for all $\bsnu\ne\bszero$, and these bounds do
not involve the $\bsnu=\bs0$ cases, we will use them to establish the base
step of the induction $|\bsnu| = 1$. Letting $\bsnu = \bse_i$ in
\eqref{eq:dl1_rec} and \eqref{eq:du1_rec},
and then using the bounds in \eqref{eq:lambda_bnd} and \eqref{eq:u_bnd} gives
\begin{align*}
 |\pdy{\bse_i}\lambda_1| \,&\leq\,
 \ubar^2
 \left(1 + \tfrac{1}{\evalueLap}\right)\frac{\beta_i}{C_\bsbeta}
  \,=\,\frac{\lambdabar}{\amin}
 \left(1 + \tfrac{1}{\evalueLap}\right)\frac{\beta_i}{C_\bsbeta}\,,
 \quad\mbox{and}\\
 \nrm{\pdy{\bse_i}u_1}{V} \,&\leq\,
\frac{\ubar}{\Cgap}
\left(1 + \tfrac{1}{\evalueLap}\right)\frac{\beta_i}{C_\bsbeta}\,,
\end{align*}
as required.

For the inductive step for the eigenvalue derivative bound, suppose
that $|\bsnu| \ge 2$ and that the bounds \eqref{eq:dl1} and
\eqref{eq:du1} hold for all multi-indices of order $< |\bsnu|$.
Substituting the induction assumptions \eqref{eq:dl1} and \eqref{eq:du1} into \eqref{eq:dl1_rec}
and then factoring out $C_1\bsbeta^\bsnu$ gives
\begin{align*}
|\pdy{\bsnu} \lambda_1| \,\leq\,&
\sqrt{\frac{\amax}{\evalueLap}} \sum_{\bs0 \neq \bsm < \bsnu} \binom{\bsnu}{\bsm}
C_1\, (|\bsm|!)^{1 + \epsilon}  \bsbeta^\bsm
\cdot C_2 \, (|\bsnu - \bsm|!)^{1 + \epsilon}
\, \bsbeta^{\bsnu - \bsm} \\
&\;+ \ubar
\left(1 + \tfrac{1}{\evalueLap}\right)
\sum_{j\in \N} \nu_j \frac{\beta_j}{C_\bsbeta}
\cdot C_2 \, [(|\bsnu| - 1)!]^{1 + \epsilon} \,
\bsbeta^{\bsnu - \bse_j} \\
\,\leq\,&
C_1 \bsbeta^\bsnu
\Bigg(
\sqrt{\frac{\amax}{\evalueLap}}C_2
\sum_{\bs0 \neq \bsm < \bsnu} \binom{\bsnu}{\bsm}
(|\bsm|!)^{1 + \epsilon} (|\bsnu - \bsm|!)^{1 + \epsilon}\\
&\;+ \ubar
\left( 1 + \tfrac{1}{\evalueLap}\right)\frac{C_2}{C_1C_\bsbeta}
 |\bsnu| [(|\bsnu| - 1)!]^{1 + \epsilon}
\Bigg)\, .
\end{align*}

Using the identity $\sum_{\bsm \leq \bsnu, |\bsm| = k}
\binom{\bsnu}{\bsm} = \binom{|\bsnu|}{k}$ along with Lemma~\ref{lem:S_n},
we can bound the sum as follows
\begin{align}
\label{eq:sumbnd}
\nonumber
&\sum_{\bs0 \neq \bsm < \bsnu} \binom{\bsnu}{\bsm}
(|\bsm|!)^{1 + \epsilon} (|\bsnu - \bsm|!)^{1 + \epsilon} \nonumber\\
\,=\,&
\sum_{k = 1}^{|\bsnu| - 1} \left[\sum_{\bsm \leq \bsnu, |\bsm| = k}
\binom{\bsnu}{\bsm} \right] (k!)^{1 + \epsilon} [(|\bsnu| - k)!]^{1 + \epsilon} \nonumber\\
\,=\,&
(|\bsnu|!)^{1 + \epsilon} \sum_{k = 1}^{|\bsnu| - 1}
\binom{|\bsnu|}{k}^{-\epsilon}
\,\leq\, C_\epsilon (|\bsnu|!)^{1 + \epsilon}\,.
\end{align}
Substituting this into the bound on $|\pdy{\bsnu} \lambda_1|$ yields
\begin{align*}
|\pdy{\bsnu}\lambda_1| \,\leq\,
&C_1 (|\bsnu|!)^{1 + \epsilon} \, \bsbeta^\bsnu
\bigg[
\sqrt{\frac{\amax}{\evalueLap}}C_2
C_\epsilon
+ \ubar
\left( 1 + \tfrac{1}{\evalueLap}\right)\frac{C_2}{C_1C_\bsbeta}
\bigg]
\,.
\end{align*}
Substituting in the values for $C_1$ and $C_2$, and then using \eqref{eq:lambda_bnd}
and \eqref{eq:u_bnd} the expression in
between the square brackets simplifies to
\begin{align} \label{eq:brackets1}
\frac{1}{C_\bsbeta \Cgap} \frac{\amin^2\,\lambdabar}{\amax^2\,\lambdaunder}
\Bigg( C_\epsilon \sqrt{\frac{\lambdabar}{\lambdaunder}} + 1\Bigg)
\end{align}
and we will later specify $C_\bsbeta$ to ensure that this expression is
bounded by $1$, thus giving the required result \eqref{eq:dl1}.

For the eigenfunction derivative bounds, substituting the
induction hypotheses \eqref{eq:dl1} and \eqref{eq:du1} into
\eqref{eq:du1_rec}
\begin{align*}
\nrm{\pdy{\bsnu}u_1}{V} \,\leq\,&
\frac{\amax}{2\evalueLap}\ubar
\sum_{\bs0 \neq \bsm < \bsnu} \binom{\bsnu}{\bsm}
C_2 (|\bsm|!)^{1 + \epsilon}  \bsbeta^\bsm
\cdot C_2(|\bsnu - \bsm|!)^{1 + \epsilon}  \bsbeta^{\bsnu - \bsm} \,
\\
&+ \, \frac{\amax}{\evalueLap C_{\rm gap}}\sum_{\bs0 \neq \bsm < \bsnu} \binom{\bsnu}{\bsm}
C_1 (|\bsm|!)^{1 + \epsilon}  \bsbeta^\bsm  \,
\cdot C_2\,  (|\bsnu - \bsm|!)^{1 + \epsilon}  \bsbeta^{\bsnu - \bsm}\\
 &+\, \frac{1}{ C_{\rm gap}}
 \left(1 + \tfrac{1}{\evalueLap}\right)\sum_{j = 1}^\infty \nu_j  \, \frac{\beta_j}{C_\bsbeta}  \,
C_2 [(|\bsnu| - 1)!]^{1 + \epsilon}  \bsbeta^{\bsnu - \bse_j}\,.
\end{align*}
Factoring out $C_2\bsbeta^\bsnu$ and using \eqref{eq:sumbnd}
this becomes
\begin{align*}
\nrm{\pdy{\bsnu}u_1}{V}
\,\leq\,&  C_2\,
(|\bsnu|!)^{1 + \epsilon} \bsbeta^\bsnu
\bigg[\frac{\amax}{2\evalueLap}\ubar C_2 C_\epsilon 
+\frac{\amax}{\evalueLap C_{\rm gap}} C_1 C_\epsilon
+ \frac{1}{ C_{\rm gap}}\left(1 + \tfrac{1}{\evalueLap}\right)
\frac{1}{C_\bsbeta}\bigg]\,,
\end{align*}

Substituting in $C_1$ and $C_2$, and then using again \eqref{eq:lambda_bnd} and \eqref{eq:u_bnd}
the expression in between the square brackets simplifies to
\begin{align} \label{eq:brackets2}
\frac{1}{C_\bsbeta \Cgap} \frac{\amin^2\,\lambdabar}{\amax^2 \,\lambdaunder}
\left(\frac{3\lambdabar}{2\lambdaunder}C_\epsilon + 1\right)\,.
\end{align}
We now define $C_\bsbeta$ as in \eqref{eq:Cbeta}, so that the expression in
\eqref{eq:brackets2} is exactly $1$, thus proving the required bound for
the eigenfunction \eqref{eq:du1}, and ensuring also that the expression in
\eqref{eq:brackets1} is bounded by $1$ as required. This completes the
induction proof for \eqref{eq:dl1} and \eqref{eq:du1} for all $\bsnu \ne
\bs0$.

With this definition of $C_\bsbeta$ it can be verified that $C_1 \leq
\lambdabar$ and $C_2 \leq \ubar$ as required. Hence we have also proved
\eqref{eq:dlambda} and \eqref{eq:du} for all $\bsnu \in \indx$.
\end{proof}

\begin{Remark}\label{rem:dlam_h}\normalfont
Since $V_h \subset V$, for $h$ sufficiently small similar results can be proved analogously for the FE approximations,
with the constants
replaced by their FE  counterparts $C_{\bsbeta, h},
\overline{\lambda_{1, h}}, \overline{u_{1, h}}$.
\end{Remark}

\section{Error analysis}
\label{sec:error_analysis}

Since we are only interested in the fundamental eigenpair, to aid in the
notation we drop the subscript~1 and define $(\lambda, u) \coloneqq
(\lambda_1, u_1)$. Henceforth, with a slight abuse of notation, we
will use combinations of the subscripts $s, h, N$ to denote, respectively,
truncating the stochastic dimension to $s$ variables, a FE approximation
with meshwidth $h$ and a lattice rule approximation with $N$ points. Also,
for the dimension truncation we will denote the truncated parameter vector
by $\bsy_s \coloneqq (y_1, y_2, \ldots, y_s)$.

Clearly, if the functions $(a_j)_{j \geq 1}, (b_j)_{j \geq 1}$ satisfy
Assumption~A\ref{asm:coeff}.\ref{itm:summable}
then the sequence $\bsbeta$ defined by \eqref{eq:beta} is summable with the same $p$. 
Also, we will henceforth assume that $\bsbeta$ is ordered such that 
$\beta_1 \geq \beta_2 \geq \cdots$.

\subsection{Dimension truncation error}
\label{sec:dim} 
Here we present bounds on the error of the truncated eigenvalue,
$\lambda_s(\bsy_s) \coloneqq \lambda(\bsy_s; \bs0)$,
and the truncated eigenfunction, $u_s(\bsy) \coloneqq u_s(\bsy_s; \bs0)$, 
for both a given $\bsy$ (the strong error) and the expected value (the weak error).
To prove these estimates we will make extensive use of Taylor series
expansions in the variables $(y_j)_{j > s}$ about $\bs0$ 
with integral remainders (as in \cite{Char12}). 
However, motivated by the splitting strategy in \cite{Gant18},
we use a higher order Taylor series expansion to obtain the same rate as
in \cite{Gant18}, which is an extra order of convergence for the weak error when compared with \cite{Char12,KSS12}.

\begin{theorem}
\label{thm:trunc}
Suppose that Assumption~A\ref{asm:coeff} holds with $p \in (0, 1)$. 
There exists constants $C_1, C_2, C_3, C_4 > 0$
such that if $s \in \N$ is sufficiently large,
then for all $\bsy \in U$ the strong truncation error of the minimal eigenpair is bounded by
\begin{align}
\label{eq:lam_s_strong}
|\lambda(\bsy) - \lambda_s(\bsy_s)|
\,&\leq\, C_1  s^{-1/p + 1}
\,,\\
\label{eq:u_s_strong}
\nrm{u(\bsy) - u_s(\bsy_s)}{V}
\,&\leq\, 
C_2 s^{-1/p + 1}
\,.
\end{align}
The weak truncation error is bounded by
\begin{align}
\label{eq:lam_s_weak}
\left|\bbE_\bsy\left[\lambda - \lambda_s\right]\right|
\,&\leq\, 
C_3 s^{-2/p + 1}
\,,
\end{align}
and for $\calG \in V^*$
\begin{align}
\label{eq:u_s_weak}
\left|\bbE_\bsy\left[\calG(u) - \calG(u_s)\right]\right|
\,&\leq\, 
C_4 s^{-2/p + 1}
\,.
\end{align}
Here, $C_1, C_2, C_3, C_4$ are independent of $\bsy$ and $s$.
\end{theorem}
\begin{proof}
Since $\lambda$ is analytic in $\bsy$, Taylor's
Theorem allows us to  expand $\lambda$ as a zeroth order Taylor series 
(see, e.g., \cite[pp. 12,13]{Hoer03}) 
in the variables $\bsy_{\{j > s\}} \coloneqq (y_j)_{j > s}$ about the point $\bs0$:
\begin{align*}
\lambda(\bsy) \,=\, \lambda(\bsy_s; \bs0) +
\sum_{j > s} \int_0^1 \bigg(\pd{}{\lambda}{y_j}\bigg)
(\bsy_s; t\bsy_{\{j > s\}}) y_j \,\rd t\,.
\end{align*}
Noting that by definition $\lambda(\bsy_s; \bs0) = \lambda_s(\bsy_s)$, using the triangle inequality, the fact that $|y_j| \leq \frac{1}{2}$ and the upper bound 
\eqref{eq:dlambda}, it follows that
\begin{equation}
\label{eq:lam_s_sum}
|\lambda(\bsy) - \lambda_s(\bsy_s)| \,\leq\,
 \frac{\lambdabar}{2} \sum_{j > s}\beta_j
\,.
\end{equation}
The eigenfunction is also analytic, and so similarly
\begin{align}
\label{eq:u_s_sum}
\nrm{u(\bsy) - u_s(\cdot, \bsy_s)}{V}
\,&=\, \left\|\sum_{j > s} \int_0^1 \bigg(\pd{}{u}{y_j}\bigg)(\bsy_s; t\bsy_{\{j > s\}})
y_j \,\rd t \right\|_{V}
\nonumber\\
&\leq\, \sum_{j > s} \frac{1}{2} \int_0^1
\left\|\bigg(\pd{}{u}{y_j}\bigg)(\bsy_s; t\bsy_{\{j > s\}})\right\|_V \,\rd t
\,\leq\, \frac{\ubar}{2}  \sum_{j > s}  \beta_j
\,,
\end{align}
where we have used the upper bound \eqref{eq:du}.

In \cite[Theorem 5.1]{KSS12} it was shown that under
Assumption~A\ref{asm:coeff}.\ref{itm:summable}
the tail of the sum of the $\beta_j$ is bounded above by
\begin{align}\label{eq:sumtail}
\sum_{j > s} \beta_j \,\leq\,
\underbrace{\min\left(\frac{p}{1 - p}, 1\right) \nrm{\bsbeta}{\ell^p}
}_{\coloneqq C_\mathrm{trunc} }
s^{-1/p + 1}\,,
\end{align}
which after substitution into \eqref{eq:lam_s_sum} and 
\eqref{eq:u_s_sum} yields the two results for the strong error, 
\eqref{eq:lam_s_strong} and \eqref{eq:u_s_strong}, respectively.

For the weak error \eqref{eq:lam_s_weak} we now use a $k$th order Taylor series expansion 
(see, e.g.,  \cite[pp. 12,13]{Hoer03}), 
then handle the Taylor sum and the remainder term separately.
First, we introduce some notation: let $k = \lceil 1/ (1 - p)\rceil$, and
note that $1 < k < \infty$.
Define $\indx_s \coloneqq \{ \bszero \neq \bsnu \in \indx : \nu_j = 0 $ for all $ j = 1, 2, \ldots, s\}$,  
and for $\ell \in \N$ let $\indx_{\ell, s} \coloneqq \{ \bsnu \in \indx_s : |\bsnu| = \ell\}$.

The $k$th order Taylor series expansion of $\lambda$ in the variables $\bsy_{\{j > s\}}$
is
\begin{align*}
\lambda(\bsy) \,=\, &\lambda(\bsy_s; \bs0)
+ \sum_{\ell = 1}^k\sum_{\bsnu \in \indx_{\ell, s}} \frac{\bsy^\bsnu}{\bsnu!} 
(\partial^\bsnu \lambda) (\bsy_s; \bs0)\\
&+ \sum_{\bsnu \in \indx_{k + 1, s}} \frac{k + 1}{\bsnu!}
\bsy^\bsnu\int_0^1 (1 - t)^k (\partial^\bsnu \lambda) (\bsy_s; t\bsy_{\{j > s\}})\, \rd t\,.
\end{align*}
Taking the expected value with respect to $\bsy$, by linearity we obtain
\begin{align}
\label{eq:weak_trunc_equal}
\nonumber
\bbE_\bsy[\lambda &- \lambda_s] \,=\,
\sum_{\ell = 1}^k\sum_{\bsnu \in \indx_{\ell, s}} 
\frac{1}{\bsnu!}  \bbE_\bsy\left[ \bsy^\bsnu
(\partial^\bsnu \lambda) (\bsy_s; \bs0)\right]
\\&
+ \sum_{\bsnu \in \indx_{k + 1, s}} \frac{k + 1}{\bsnu!}
\bbE_\bsy\left[\bsy^\bsnu\int_0^1 (1 - t)^k 
(\partial^\bsnu \lambda) (\bsy_s; t\bsy_{\{j > s\}})
\, \rd t\right].
\end{align}
Since each $y_j$ is independent, for $\bsnu \in \indx_{\ell, s}$,  we have
\begin{align*}
\bbE_\bsy\left[  \bsy^\bsnu (\partial^\bsnu\lambda) (\bsy_s; \bszero)\right]
\,=\, \bbE_\bsy \left[ (\partial^\bsnu\lambda) (\bsy_s; \bszero)\right] 
\prod_{j > s}\bbE_\bsy[y_j^{\nu_j}]\,,
\end{align*}
and because $y_j$ has mean 0, any term in the first sum in \eqref{eq:weak_trunc_equal} 
with at least one $\nu_j = 1$ is zero.
This means that for the first term in \eqref{eq:weak_trunc_equal} 
we only need to consider higher order derivatives ($\ell \geq 2$),
and we can restrict the inner sum to $\bsnu \in \calF_{\ell, s}$ such that $\nu_j \neq 1$, giving
\begin{align*}
\bbE_\bsy[\lambda &- \lambda_s] \,=\,
\sum_{\ell = 2}^k\sum_{\substack{\bsnu \in \indx_{\ell, s}\\ \nu_j \neq 1\, \forall j}} 
\frac{1}{\bsnu!}  \bbE_\bsy\left[ \bsy^\bsnu
(\partial^\bsnu \lambda) (\bsy_s; \bs0)\right]
\\&
+ \sum_{\bsnu \in \indx_{k + 1, s}} \frac{k + 1}{\bsnu!}
\bbE_\bsy\left[\bsy^\bsnu\int_0^1 (1 - t)^k 
(\partial^\bsnu \lambda)(\bsy_s; t\bsy_{\{j > s\}})
\, \rd t\right].
\end{align*}

Taking the absolute value then using the triangle inequality, 
monotonicity of the expectation and that $|y_j| \leq 1/2$,
we have the bound
\begin{align*}
|\bbE_\bsy[\lambda &- \lambda_s]| \,\leq\,
\sum_{\ell = 2}^k\sum_{\substack{\bsnu \in \indx_{\ell, s}\\ \nu_j \neq 1\, \forall j}}
\frac{1}{2^\ell\bsnu!}  \bbE_\bsy\left[ (\partial^\bsnu \lambda) (\bsy_s; \bs0)|\right]\\
&+ \sum_{\bsnu \in \indx_{k + 1, s}} \frac{k + 1}{2^{k + 1}\bsnu!}
\bbE_\bsy\left[\int_0^1 (1 - t)^k |(\partial^\bsnu \lambda)(\bsy_s; t\bsy_{\{j > s\}})|\, \rd t \right]\,.
\end{align*}

Bounding each derivative using \eqref{eq:dlambda}, which is independent of $\bsy$,
and then evaluating the remaining integral over $t$ exactly gives
\begin{align}
\label{eq:weak_err1}
|\bbE_\bsy[\lambda - \lambda_s]|
&\leq \sum_{\ell = 2}^k\sum_{\substack{\bsnu \in \indx_{\ell, s}\\ \nu_j \neq 1 \, \forall j}} 
\frac{\overline{\lambda}(\ell!)^{1 + \epsilon}}{2^\ell\bsnu!} \bsbeta^\bsnu
+ \sum_{\bsnu \in \indx_{k + 1, s}} \frac{\overline{\lambda} ((k + 1)!)^{1 + \epsilon}}{2^{k + 1}\bsnu!} \bsbeta^\bsnu
\nonumber\\
\,&\leq\, C_k \sum_{\ell = 2}^k \sum_{\substack{\bsnu \in \indx_{\ell, s}\\\nu_j \neq 1 \, \forall j}}
\bsbeta^\bsnu
+ C_{k + 1} \sum_{\bsnu \in \indx_{k + 1, s}} \bsbeta^\bsnu\,,
\end{align}
where
\begin{align*}
C_k \,=\, \overline{\lambda} \max_{\substack{\bsnu \in \indx\\ |\bsnu| \leq k}}\frac{(|\bsnu|!)^{1 + \epsilon}}{2^{|\bsnu|} \bsnu!}
\,<\, \infty
\end{align*}
is independent of $s$, but depends on $p$ through $k$, and we have again dropped 
the subscript 1 for the upper bound (see \eqref{eq:lambda_bnd}) on the minimal eigenvalue,
$\overline{\lambda} \coloneqq \lambdabar$.

For the first sum in \eqref{eq:weak_err1}, since the order $|\bsnu|$ satisfies 
$|\bsnu| \geq \nrm{\bsnu}{\ell^\infty}$ and every $\beta_j$ is positive, 
we can add extra terms to the sum to obtain the
bound
\[
\sum_{\ell = 2}^k \sum_{\substack{\bsnu \in \indx_{\ell, s}\\\nu_j \neq 1\, \forall j}}
\bsbeta^\bsnu
\,=\, \sum_{\substack{\bsnu \in \indx_s\\ |\bsnu| \leq k\\ \nu_j \neq 1 \, \forall j}}
\bsbeta^\bsnu
\,\leq\,
\sum_{\substack{\bsnu \in \indx_s\\ \nrm{\bsnu}{\ell^\infty} \leq k\\ \nu_j \neq 1\, \forall j}}
\bsbeta^\bsnu.
\]
Then, as in \cite{Gant18} we can write the sum on the right as the following product
\[
\sum_{\substack{\bsnu \in \indx_s\\ \nrm{\bsnu}{\ell^\infty} \leq k\\ \nu_j \neq 1 \, \forall j}}
\bsbeta^\bsnu
\,=\, - 1 +  \prod_{j > s} \bigg(1 + \sum_{\ell = 2}^k \beta_j^\ell\bigg)
\,=\, -1 + \prod_{j > s} \bigg(1 +\bigg( \frac{1 - \beta_j^{k - 1}}{1 - \beta_j}\bigg) \beta_j^2\bigg)\,,
\]
where in the last step we have used the formula for the sum of a geometric series.

In order to simplify the product above, we define the sequence $\widetilde{\bsbeta}$ by
\[
\widetilde{\beta}_j \,\coloneqq\, 
\begin{cases}
\beta_j^2\,, & \text{for } j = 1, 2, \ldots, s,\\[2mm]
\displaystyle\bigg(\frac{1 - \beta_j^{k - 1}}{1 - \beta_j}\bigg) \beta_j^2 \,, &
\text{for } j > s\,,
\end{cases}
\]
which, because the sequence $\bsbeta$ is assumed to be decreasing, 
is well defined for $s$ sufficiently large such that $\beta_s \leq 1/2$.
Further, since $\widetilde{\beta}_j \leq \beta_j^2/ (1 - \beta_s) \leq 2 \beta_j^2$ for all $j \in \N$,
it follows that $\widetilde{\bsbeta} \in \ell^{p/2}$ and $\nrm{\widetilde{\bsbeta}}{\ell^{p/2}}$
can be bounded from above independently of $s$.

Then, using the inequalities $\ln(1 + x) \leq x$ and $-1 + e^x \leq xe^x$, 
we can bound the first sum in \eqref{eq:weak_err1} by
\begin{align}
\label{eq:weak_sum1}
\sum_{\ell = 2}^k \sum_{\substack{\bsnu \in \indx_{\ell, s}\\\nu_j \neq 1 \, \forall j}}
\bsbeta^\bsnu
\,&\leq\,
 -1 + \prod_{j > s} (1 + \widetilde{\beta}_j)
\,=\, -1 + \exp \Bigg( \sum_{j > s} \ln (1 + \widetilde{\beta}_j)\Bigg)
\,\leq\, \exp\Bigg(\sum_{j > s} \widetilde{\beta}_j\Bigg) \sum_{j > s} \widetilde{\beta}_j
\nonumber\\
&\leq\, \exp\big(\nrm{\widetilde{\bsbeta}}{\ell^1}\big) 
\min\left(\frac{p}{2 - p}, 1 \right) \nrm{\widetilde{\bsbeta}}{\ell^{p/2}} \, s^{-2/p + 1}\,,
\end{align}
where we have also used \eqref{eq:sumtail} for the sequence $\widetilde{\bsbeta} \in \ell^{p/2}$.

For the second sum in \eqref{eq:weak_err1}, since 
$\binom{k + 1}{\bsnu} = \frac{(k+1)!}{\bsnu!} \geq 1$, 
then using \eqref{eq:sumtail} and the definition of $k$ we have 
\begin{align}
\label{eq:weak_sum2}
\sum_{\bsnu \in \indx_{k + 1, s}} \bsbeta^\bsnu
\,&\leq\, \sum_{ \bsnu \in \indx_{k + 1, s}} \binom{k + 1}{\bsnu} \bsbeta^{\bsnu}
\,=\, \Bigg(\sum_{j > s} \beta_j\Bigg)^{k + 1}
\,\leq\, C_\mathrm{trunc}^{k + 1} s^{(k + 1)(-1/p + 1)}
\nonumber\\
\,&\leq\, C_\mathrm{trunc}^{k + 1} s^{(1/(1 - p) + 1) \cdot (-1/p + 1)}
\,=\, C_\mathrm{trunc}^{k + 1} s^{-2/p + 1}\,.
\end{align}

The result \eqref{eq:lam_s_weak} for the weak error of the eigenvalue
is then obtained by substituting \eqref{eq:weak_sum1} and \eqref{eq:weak_sum2}
into \eqref{eq:weak_err1}.

As for the proof of the strong error \eqref{eq:u_s_strong}, we can also expand $u$
as a $k$th order Taylor series, and then use the same argument to prove the
bound \eqref{eq:u_s_weak} for the weak error of the eigenfunction.
\end{proof}

\subsection{QMC error}
Given the bounds in Lemma~\ref{lem:dlambda} on the mixed derivatives of the minimal
eigenpair we now obtain an upper bound of the root-mean-square error of the QMC approximation of
the truncated problem.

\begin{theorem}
\label{thm:qmc_err}
Let $N \in \N$ be prime, $\calG \in V^*$ and suppose that 
Assumption~A\ref{asm:coeff} holds.
Then the root-mean-square errors of the CBC-generated
randomly shifted lattice rule approximations
of $\bbE_\bsy\left[\lambda_s\right]$ and $\bbE_\bsy\left[\calG(u_s)\right]$
are bounded by
\begin{align}
\label{eq:qmcerr_lam}
\sqrt{\bbE_\bsDelta\left[ \left| \bbE_\bsy\left[\lambda_s\right] - Q_{N, s}\lambda_s\right|^2\right]}
\,&\leq\, C_{1, \alpha} N^{-\alpha}\,,
\quad\text{and}\\
\label{eq:qmcerr_u}
\sqrt{\bbE_\bsDelta\left[\left|\bbE_\bsy\left[\calG(u_s)\right] -
 Q_{N, s} \calG(u_s)\right|^{2}\right]}
\,&\leq\, C_{2, \alpha} N^{-\alpha}\,,
\end{align}
where
\begin{align}
 \alpha \,=\,
 \begin{cases}
 1 - \delta, \mbox{ for arbitrary $\delta\in (0,\frac{1}{2})$}, & \text{if } p \in (0, \frac{2}{3}]\,,\\
\frac{1}{p} - \frac{1}{2} & \text{if } p \in (\frac{2}{3}, 1)\,,
\end{cases}
\end{align}
and the constants $C_{1, \alpha}$ and $C_{2, \alpha}$ are independent
of $s$.
\end{theorem}

\begin{proof}
Since the estimates from Lemma \ref{lem:dlambda} are independent of $\bsy$
they can be used to bound the norm (squared) of  $\lambda_s$ in $\calW_{s,
\bsgamma}$. By \eqref{eq:dlambda} we obtain
\begin{align*}
 \nrm{\lambda_s}{s, \bsgamma}^2
&\,\leq\, \overline{\lambda}^2
\sum_{\setu \subseteq\{1:s\}} \frac{\Lambda_\setu^2}{\gamma_\setu}\,,\quad
 \Lambda_\setu \,\coloneqq\, (|\setu|!)^{1 + \epsilon} \prod_{j \in \setu} \beta_j \,,
\end{align*}
with weights $\bsgamma$ and $\epsilon\in(0,1)$ as yet unspecified.
Then, using \eqref{eq:msebound} the mean-square error of the lattice rule
approximation is bounded above by
\begin{align}
\label{eq:mse2}
\bbE_\bsDelta\left[\left|\bbE_\bsy\left[\lambda_s\right] - Q_{N, s}\lambda_s\right|^2\right] \,\leq\,
C_{s, \bsgamma, \eta}\, \varphi(N)^{-\frac{1}{\eta}}\,,
\end{align}
where
\begin{align*}
C_{s, \bsgamma, \eta} \,\coloneqq\,
\overline{\lambda}^2
\left(\sum_{\emptyset \neq \setu \subseteq \{1:s\}} \gamma_\setu^\eta\,
\rho(\eta)^{|\setu|}\right)^\frac{1}{\eta}
\left(\sum_{\setu \subseteq \{1:s\}} \frac{\Lambda_\setu^2}{\gamma_\setu} \right)
\end{align*}

We now choose the weight parameters such that $C_{s, \bsgamma, \eta}$ can
be bounded independently of $s$. From \cite[Lemma 6.2]{KSS12} the choice of
weights that minimise $C_{s, \bsgamma, \eta}$ are
\begin{align}
\label{eq:gamma_opt}
\gamma_\setu(\eta) \,=\, \left(\frac{\Lambda_\setu^2}{\rho(\eta)^{|\setu|}}\right)^\frac{1}{1 + \eta}\,,
\end{align}
which are of POD (product and order-dependent) form. With these
weights  it follows that $C_{s, \bsgamma,\eta} \le \overline{\lambda}^2\,S_{s,
\eta}^{(1 + \eta)/\eta}$, where
\begin{align*}
 S_{s, \eta} \,\coloneqq\, \sum_{\setu \subseteq \{1:s\}} \left(\Lambda_\setu^{2\eta} \rho(\eta)^
{|\setu|}\right)^\frac{1}{1 + \eta}\,,
\end{align*}
so we must show that the sum $S_{s, \eta}$ can be bounded independently
of $s$. Let
\begin{align*}
q \,\coloneqq\, \frac{2\eta(1 + \epsilon)}{1 + \eta}
\quad \text{and} \quad
\alpha_j \,\coloneqq\, \left(\rho(\eta) (\beta_j)^{2\eta}\right)^\frac{1}{1 + \eta} \quad \text{for all } j \in \N\,,
\end{align*}
so that
\begin{align*}
S_{s, \eta} \,=\, \sum_{\ell = 0}^s (\ell!)^q
\sum_{\satop{\setu \subseteq \{1:s\}}{|\setu| = \ell}} \prod_{j \in \setu} \alpha_j
\,\le\, \sum_{\ell = 0}^s (\ell !)^{q - 1}
 \bigg(\sum_{j = 1}^s \alpha_j\bigg)^\ell
\,<\, \infty\,,
\end{align*}
which holds by the ratio test provided that $q<1$ and
$\sum_{j=1}^\infty \alpha_j < \infty$. Under
Assumption~A\ref{asm:coeff}.\ref{itm:summable}, we therefore require that
\begin{align*}
 \frac{2\eta(1 + \epsilon)}{1 + \eta} < 1\,\iff \, \epsilon < \frac{1-\eta}{2\eta}
  \quad\mbox{and}\quad
\frac{2\eta}{1 + \eta} \le p \,\iff \, \eta \geq \frac{p}{2 - p}\,.
\end{align*}
To balance these conditions with the requirement that $\eta \in
(\frac{1}{2}, 1]$, we choose a different $\eta$ depending on the decay
rate $p$ and then choose $\epsilon \coloneqq (1-\eta)/(4\eta)$. Note that
$\eta=1$, equivalently $p = 1$, has to be excluded to ensure that $\epsilon>0$.

For $p \in (0, \frac{2}{3}]$, we have $\frac{p}{2 - p} \leq
\frac{1}{2}$ so there is no further restriction on $\eta$ and we take
$\eta \coloneqq \frac{1}{2(1 - \delta)}$ for arbitrary $\delta\in
(0,\frac{1}{2})$. However, for $p \in (\frac{2}{3},1)$ the value of
$\eta$ is restricted and we take it as small as possible, namely,
$\eta \coloneqq \frac{p}{2 - p}$. Substituting these choices of $\eta$
into \eqref{eq:mse2} and taking $N$ to be prime (for simplicity) yields
the result \eqref{eq:qmcerr_lam}.

The error bound \eqref{eq:qmcerr_u} follows in the same way, after
observing that the norm of $\calG(u_s)$ can be bounded using
\eqref{eq:du}
\begin{align*}
\nrm{\calG(u_s)}{s,\bsgamma}^2
\,&\leq\,
\sum_{\setu \subseteq \{1:s\}} \frac{1}{\gamma_\setu}
\int_{[0, 1]^{|\setu|}} \bigg(\int_{[0, 1]^{s - |\setu|}}
\nrm{\calG}{V^*} \bigg\|\pd{|\setu|}{u_s}{\bsy_\setu}(\cdot, \bsy_s)\bigg\|_V
\rd\bsy_{-\setu}\bigg)^2 \rd\bsy_{\setu}\\
&\leq\,
\nrm{\calG}{V^*}^2\overline{u}^2
\sum_{\setu \subseteq \{1:s\}} \frac{\Lambda_\setu^2}{\gamma_\setu}\,.
\end{align*}
This completes the proof.
\end{proof}

\begin{Remark}
\label{rem:qmcerr_lam_h}
\normalfont
The main ingredients in this proof are the bounds on the derivatives
of the eigenvalue, which are needed to show that 
$\lambda_s \in \calW_{s,\bsgamma}$. 
As was stated in Remark~\ref{rem:dlam_h}, these bounds also
hold for the derivatives of $\lambda_{s,h}$.
Thus, for $h$ sufficiently small (see \eqref{eq:h0}) the same error bound holds for
the QMC error of the FE error approximation $\lambda_{s, h}$, but
with the constants possibly depending on $h$.
\end{Remark}

\subsection{Total error}
Using the triangle inequality to give, 
the mean-square error of the combined truncation-FE-QMC approximation
of the expected value of $\lambda_1$ can be
bounded above by
\begin{align}
\label{eq:errdecomp}\nonumber
&\bbE_\bsDelta\left[\left|\bbE_\bsy[\lambda] - Q_{N, s}\lambda_{s, h}\right|^2\right]
\,\leq\,
C\Big(\bbE_\bsDelta\left[\left|\bbE_\bsy[\lambda - \lambda_s]\right|^2\right]\\
&+ \bbE_\bsDelta\left[\left|\bbE_\bsy[\lambda_s] - Q_{N, s}\lambda_{s}\right|^2\right]
+ \bbE_\bsDelta\left[\left|Q_{N, s}(\lambda_s -
\lambda_{s, h})\right|^2\right]\Big)\,,
\end{align}
for $C > 0$.
Here we have conveniently split the total error into
three separate errors: one each for the truncation, QMC and FE errors, respectively.
Note that there are different ways of splitting the total error, but we have chosen
the above technique because now the second term is the QMC error 
for the actual eigenvalue and not the FE approximation $\lambda_{s, h}$.
This is important because it means that we do not need specific bounds 
on the parametric regularity of the FE eigenvalue.

The terms in the upper bound on the mean-square error in \eqref{eq:errdecomp} can
be bounded using \eqref{eq:lam_s_weak}, \eqref{eq:qmcerr_lam} and 
\eqref{eq:lambda_fe_err}, respectively, leading to the following theorem. 
A similar splitting argument, using \eqref{eq:u_s_strong}, \eqref{eq:u_fe_err} and 
\eqref{eq:qmcerr_u} instead, gives a bound on the error of the
approximation for functionals of the corresponding eigenfunction.

\begin{theorem}
\label{thm:total_error} Let Assumption~A\ref{asm:coeff} hold, $h > 0$ be sufficiently small, $ s \in \N$, 
$N\in \N$ be prime and let $\bsz \in \N^s$ be a generating vector constructed using the CBC
algorithm with weights given by \eqref{eq:gamma_opt}.
Then 
the root-mean-square error, with respect to the random shift $\bsDelta
\in [0, 1]^s$, of our truncation-FE-QMC approximation of the mean
  of the
minimal eigenvalue $\lambda$ is bounded by
\begin{align}
\label{eq:err_lam}
\sqrt{\bbE_\bsDelta\left[\left|\bbE_\bsy[\lambda] - Q_{N, s}\lambda_{s, h}\right|^2\right]}
\,\leq\,C_1\Big(h^2 + s^{-2/p + 1} + N^{-\alpha}\Big)\,.
\end{align}
For any functional $\calG \in H^{-1 + t}(D)$ applied to the corresponding
  eigenfunction $u$, with $t \in [0, 1]$, the 
truncation-FE-QMC approximation of its mean 
is bounded by
\begin{align}
\label{eq:err_u}
\sqrt{\bbE_\bsDelta\left[\left|\bbE_\bsy[\calG(u)] - Q_{N, s}\calG(u_{s, h})\right|^2\right]}
\,\leq\,C_2\Big(h^{1 + t} + s^{-2/p + 1} + N^{-\alpha}\Big)\,.
\end{align}
Where
\begin{align*}
\alpha \,=\, \begin{cases}
1 - \delta, \mbox{ for arbitrary $\delta\in (0,\frac{1}{2})$,} &\text{if } p \in (0, \tfrac{2}{3}]\,,\\
\tfrac{1}{p} - \tfrac{1}{2} &\text{if } p \in (\tfrac{2}{3}, 1)\,,
\end{cases}
\end{align*}
and the constants $C_1, C_2 > 0$ are independent of $s$, $h$ and $N$.
\end{theorem}

Comparing this with the corresponding result for the elliptic source problem
\cite[Theorem~8.1]{KSS12}, observe that we obtain the exact same
rate of convergence in $N$ for all $p \in (0, 1)$. The only exception is
that our results do not hold when $p = 1$, whereas
\cite[Theorem~8.1]{KSS12} presents a result when $p = 1$. However, for that case
they do require an additional assumption (see \cite[equation (6.5)]{KSS12}).
For the truncation error, we obtain the same convergence rate as in \cite{Gant18},
which improved the rate in \cite[Theorem~8.1]{KSS12} by one order.
To compare the two FE convergence rates recall that the number of degrees of
freedom in the FE grid is $M_h = \bigO(h^{-d})$.
Letting $\calG \in H^{-1 + t}(D)$, note that the correct comparison is the case when the 
source term belongs to $L^2(D)$, so that in \cite[Theorem~8.1]{KSS12} $\tau = 1 + t$.
In this case the FE error convergence rate from \cite[Theorem~8.1]{KSS12} for 
the linear functional $\calG$ of the solution to the source problem is $M_h^{-\tau/d} = \bigO(h^{1 + t})$,
which is exactly the rate in \eqref{eq:err_u} for the eigenfunction.

\section{Numerical results}
\label{sec:numerical}
Now we present numerical results on the performance of our truncation-FE-QMC
algorithm in approximating the expected value of the smallest eigenvalue of 
two different eigenvalue problems of the form \eqref{eq:evp}.
In both cases, the stochastic coefficients $a$ and $b$ are composed of scaled 
trigonometric functions, which can be seen as artificial
Karhunen--Lo\`{e}ve (KL) expansions. We will focus on the question of whether
the error matches the theoretical estimate from Theorem~\ref{thm:total_error},
and so we will study different scalings of the basis functions $a_j, b_j$ in the coefficients,
which will correspond to different values of the decay parameter $p$.
All computations were performed on the Katana cluster at UNSW Sydney.

In order to estimate the quadrature error,
 we conduct a small number $R$ of 
different approximations based on independently and identically 
distributed random shifts.
We label the approximation generated by the $r$th random shift $\bsDelta^{(r)}$
as $Q^{(r)}_{N, s}\lambda_{s, h}$, and the final approximation is taken to be the average 
over the $R$ independent approximations:
\begin{equation*}
\widehat{Q}_{R, N, s}\lambda_{s, h} \,\coloneqq\,
\frac{1}{R} \sum_{r = 1}^R Q^{(r)}_{N, s} \lambda_{s, h}\,.
\end{equation*}
In this way, we obtain an unbiased estimate of the integral (of $\lambda_{s, h}$) 
and the sample standard deviation over the different shifts gives an estimate of the 
quadrature component of the RMS error
\begin{align}
\label{eq:rmse_est}
\sqrt{\bbE_{\bsDelta} \left[\left|
\bbE_\bsy[\lambda_{s, h}] - \widehat{Q}_{N, s}\lambda_{s, h}\right|^2\right]}
\,\approx\, \sqrt{\frac{1}{R(R - 1)} \sum_{r = 1}^R \left(
Q^{(r)}_{N, s}\lambda_{s, h} - \widehat{Q}_{R, N, s}\lambda_{s, h}\right)^2}\,.
\end{align}
Note that averaging over $R$ shifts increases the total number of function 
evaluations by a factor of $R$, and hence, one can expect the error to
correspondingly decrease by a factor of $1/\sqrt{R}$ (in line with the Monte Carlo rate).

The smallest eigenvalue of each FE system is approximated using 
the \texttt{eigs} function in Matlab, which in turn runs a Krylov--Schur
algorithm using the ARPACK library.
We set the tolerance for the accuracy of this eigensolver to be $10^{-14}$,
so as to ensure that the numerical errors incurred in computing the
FE eigenvalues are negligible compared to the overall approximation error.

In practice, we cannot compute the optimal function space weights $\gamma_\setu$
according to the formula \eqref{eq:gamma_opt} from the
proof of Theorem~\ref{thm:qmc_err}, because \eqref{eq:gamma_opt} 
depends on the sequence $\bsbeta$, and $C_\bsbeta$ \eqref{eq:Cbeta} 
contains factors that cannot be computed explicitly. Instead we choose the
weights so that the product components decay at the same rate as in
the optimal formula \eqref{eq:gamma_opt}.
As such, in our numerical experiments for $\setu \subset \N$ we set the 
function space weight $\gamma_\setu$ to be
\begin{equation}
\label{eq:gamma_num}
\gamma_\setu \,=\, |\setu|! \prod_{j \in \setu} \left(
\max\left\{\nrm{a_j}{L^\infty(D)}, \nrm{b_j}{L^\infty(D)}\right\}\right)^
\eta\,,
\end{equation}
where $\eta = 4/3$ if $p \in (0, 2/3]$ and $\eta = 2 - p$ if $p \in (2/3, 1)$.
Note that $\eta$ in \eqref{eq:gamma_num} is not the same as in
\eqref{eq:gamma_opt}.

\subsection{Problem 1}
In our first simple example, we consider an eigenvalue problem \eqref{eq:evp}
on the domain $D = (0, 1)^2$ where the only non-trivial coefficient
is $a(\bsy)$ in the second-order term. 
Explicitly, the coefficient $a(\bsy)$ is given as in
\eqref{eq:a_general} (with $a_j$ defined below) but $b(\bsy) \equiv 0$ and $c \equiv 1$.
For some decay $q \geq 4/3$, the basis functions for the coefficient $a(\bsy)$ 
are defined to be
\begin{equation}
\label{eq:a_eg1}
a_0 \equiv 1\,,
\quad
a_j(\bsx) \,=\, \frac{1}{1 + (j\pi )^q} \sin(j\pi x_1)\sin((j + 1)\pi x_2)\,,
\quad \text{for } \bsx = (x_1, x_2) \in (0, 1)^2\,.
\end{equation}
Clearly, for all $j \in \N$ we have that 
\[
\nrm{a_j}{L^\infty(D)} \,=\, \frac{1}{1 + (j\pi )^q} \,<\, \frac{1}{(j\pi)^q}\,,
\]
and hence $\sum_{j = 1}^\infty \nrm{a_j}{L^\infty(D)} < \zeta(q)/\pi^q$,
where again $\zeta$ is the Riemann zeta function.
It follows that the coefficient is bounded above and below as required with
\begin{equation*}
\amin \,=\, 1 - \frac{\zeta(q)}{\pi^q}\,
\quad \text{and} \quad
\amax \,=\, 1 + \frac{\zeta(q)}{\pi^q}\,.
\end{equation*}
Similarly, the parameter $q$ determines the rate of decay of the norms of the basis functions,
and in turn we can take $p$ in Assumption~A\ref{asm:coeff}.\ref{itm:summable}
to satisfy $p \in (1/q, 1)$. 
We also ran the case when the coefficient $b$ is nonzero
and of a similar form to \eqref{eq:a_eg1}.
The results are almost exactly the same as this example, and so have not been included.

In our numerical experiments for this problem we consider $q = 4/3, 2, 3$,
and vary the approximation parameters as follows.
The truncation dimensions tested are given by $s = 2, 4, 8, \ldots, 256$; we 
use uniform triangular FE meshes with $h = 1/4, 1/8, 1/16, \ldots, 1/1024$; and
the number of quadrature points is given by $N =  $ 31, 61, 127, 251, 503, 997, 1999, 4001, 8009, 16001.

Regarding the different QMC convergence rates to expect,
for $q = 4/3$ we have that the sequence $\bsbeta$ is $p$-summable for
$p > 3/4$, whereas for the faster decays
of $q = 2, 3$ we have that $\bsbeta$ is summable with exponent 
$p > 1/2$ and $p > 1/3$, respectively. 
Based on these restrictions on $p$,
for each $N$ we construct a generating vector by the CBC algorithm 
using weights given by \eqref{eq:gamma_num}
with $\eta = 2 - 1/q = 5/4$ for $q = 4/3$, and $\eta = 4/3$ for $q = 2, 3$.

First, we study the truncation error by varying $s$, while keeping
$h = 1/512$ and $N = 8009$ fixed (using a
single fixed shift for all values of $s$). The results are
given in Figure~\ref{fig:trunc}. The errors are estimated by 
comparing each result with a fine solution with truncation dimension
512.
Theorem~\ref{thm:total_error} predicts that the truncation errors converge like
$s^{-5/3}$, $s^{-3}$ and $s^{-5}$ for $q = 4/3, 2, 3$, respectively.
As a guide, these expected rates are given by the dashed lines in Figure~\ref{fig:trunc}.
Observe that for all cases the estimated truncation errors closely follow
the expected convergence.

In Figure~\ref{fig:FE_err} we present results for the FE error convergence
for $q = 4/3$ ($p \approx 3/4$).
Again, to isolate the FE component we vary $h$, but fix $s = 128$ and $N = 8009$ 
(using a single fixed shift for all values of $h$). 
Then we estimate the errors by comparing with a fine
solution that uses a meshwidth of $1/2048$.
Theorem~\ref{thm:total_error} predicts that  the FE error converges
like $h^2$, which is clearly observed. The other cases, $q = 2, 3$, both exhibit
very similar errors and so have not been included.
 \begin{figure}[!h]
\begin{minipage}{0.49\textwidth}
\centering
 \includegraphics[width=\textwidth]{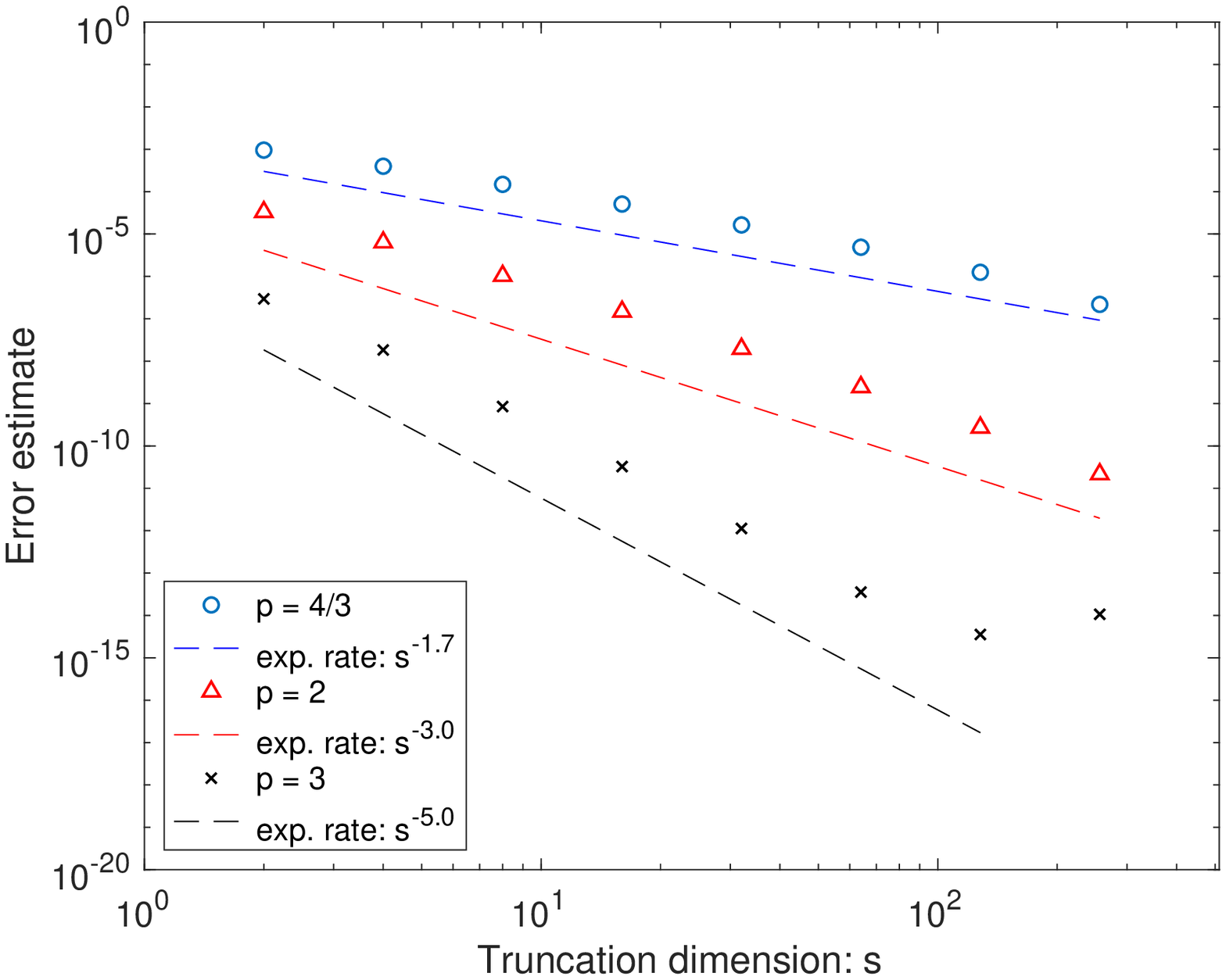}
 \caption{Dimension truncation error estimate for $q = 4/3, 2, 3$ ($p\approx 3/4, 1/2, 1/3$).}
 \label{fig:trunc}
  \end{minipage}
 \hfill
 \begin{minipage}{0.49\textwidth}
 \centering
 \includegraphics[width=\textwidth]{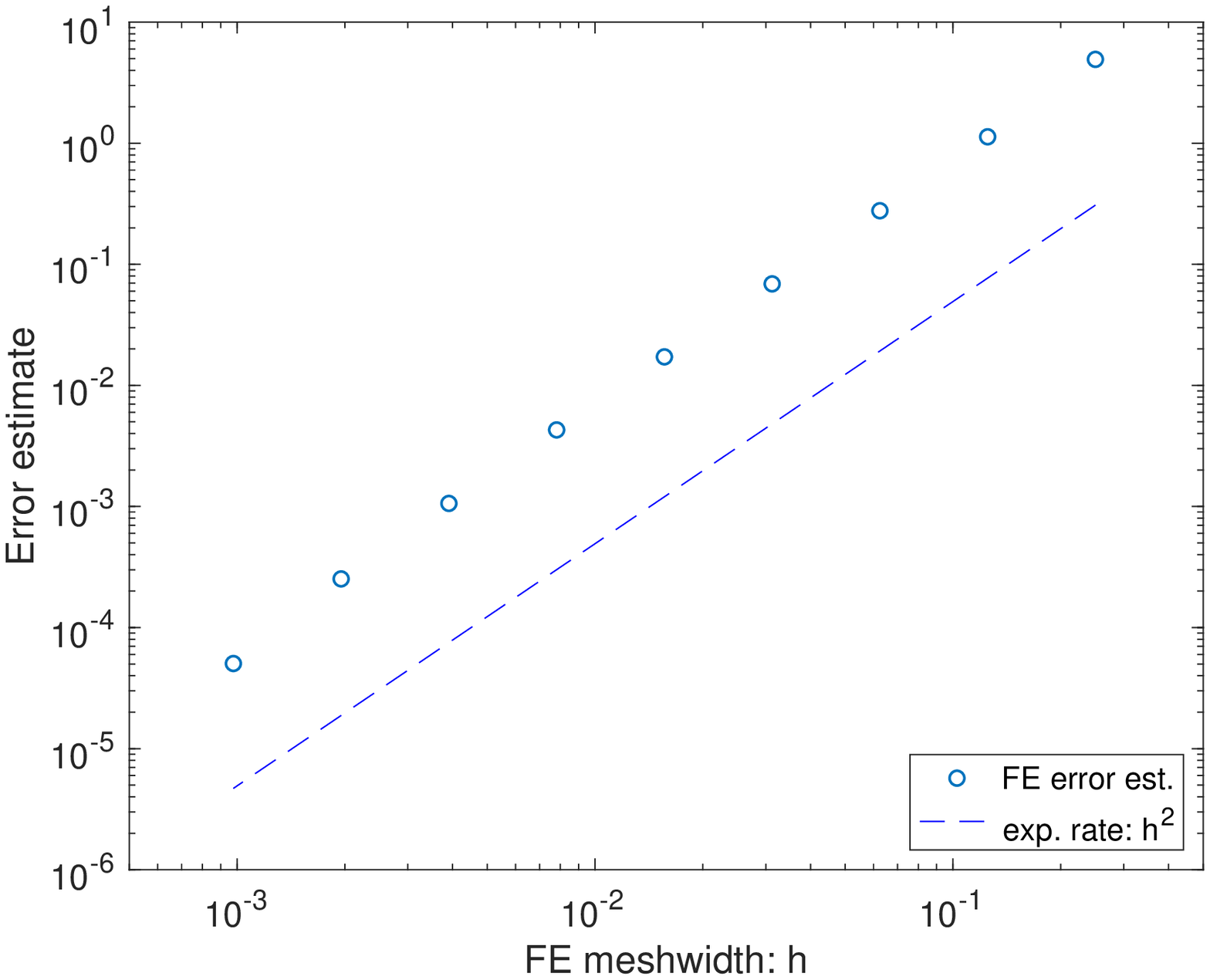}
 \caption{FE error estimate (right) for $q = 4/3, 2, 3$ ($p\approx 3/4, 1/2, 1/3$).}
 \label{fig:FE_err}
 \end{minipage}
 \end{figure}

To study the quadrature component of the error,
we fix $s = 256$, $h = 1/256$ and use $R = 8$ random shifts to estimate the RMS error
\eqref{eq:rmse_est}.
Figure~\ref{fig:QMC-MC_conv} plots the estimated RMS quadrature error
of our QMC approximation along with a Monte Carlo (MC) approximation
for comparison, for  $q = 4/3$ (left) and  $q = 2$ (right).
To fairly compare with the MC results, for each $N$
we present the estimated RMS error of a single randomly shifted lattice rule.
That is, we scale the RMS error estimate in \eqref{eq:rmse_est} 
by $\sqrt{R}$ so as to remove the extra $1/\sqrt{R}$ factor gained by random shifting.
In Figure~\ref{fig:QMC-MC_conv} the circle data points (crosses for MC) represent
the estimated RMS errors, the dashed lines portray the expected convergence rates
and the solid lines are a least-squares fit.
As a guide, we have plotted the FE error one can expect for 
$h = 2^{-8}, 2^{-9}, 2^{-10}$ as three red dotted lines (we use the results used to generate 
Figure~\ref{fig:FE_err}), note that the vertical heights of the lines
decrease as $h$ decreases.

From Theorem~\ref{thm:total_error}, for $q = 4/3$ (because $p> 2/3$)
we are in the regime where the convergence is limited and so expect a
rate of around $N^{-5/6}$.
However, for $q = 2$ the rate is not restricted and we expect QMC convergence
close to $N^{-1}$.
Observe that the least-squares fits match very closely to the expected rates.
Also, as is expected for a problem of this smoothness QMC
significantly outperforms the MC approximations, which decay 
at the anticipated rate of $1/\sqrt{N}$.
Comparing with the expected FE errors (the red dotted lines), we can see 
that for this example the dominant component of the error is the FE error,
this is to be expected and has been observed in source problems also.
Despite this, we have continued with QMC approximations that have errors below the 
FE error to demonstrate that the asymptotic convergence rates 
predicted by Theorem~\ref{thm:total_error} are achieved in practice.
For the truncation error, taking $s = 256$ results in errors of approximately $10^{-7}$ and
$10^{-10}$ for $q = 4/3$ and 2, respectively, which are at least an order of magnitude less than
the smallest quadrature errors in each case.
 
 \begin{figure}[!h]
 \centering
 \includegraphics[width=.48\textwidth]{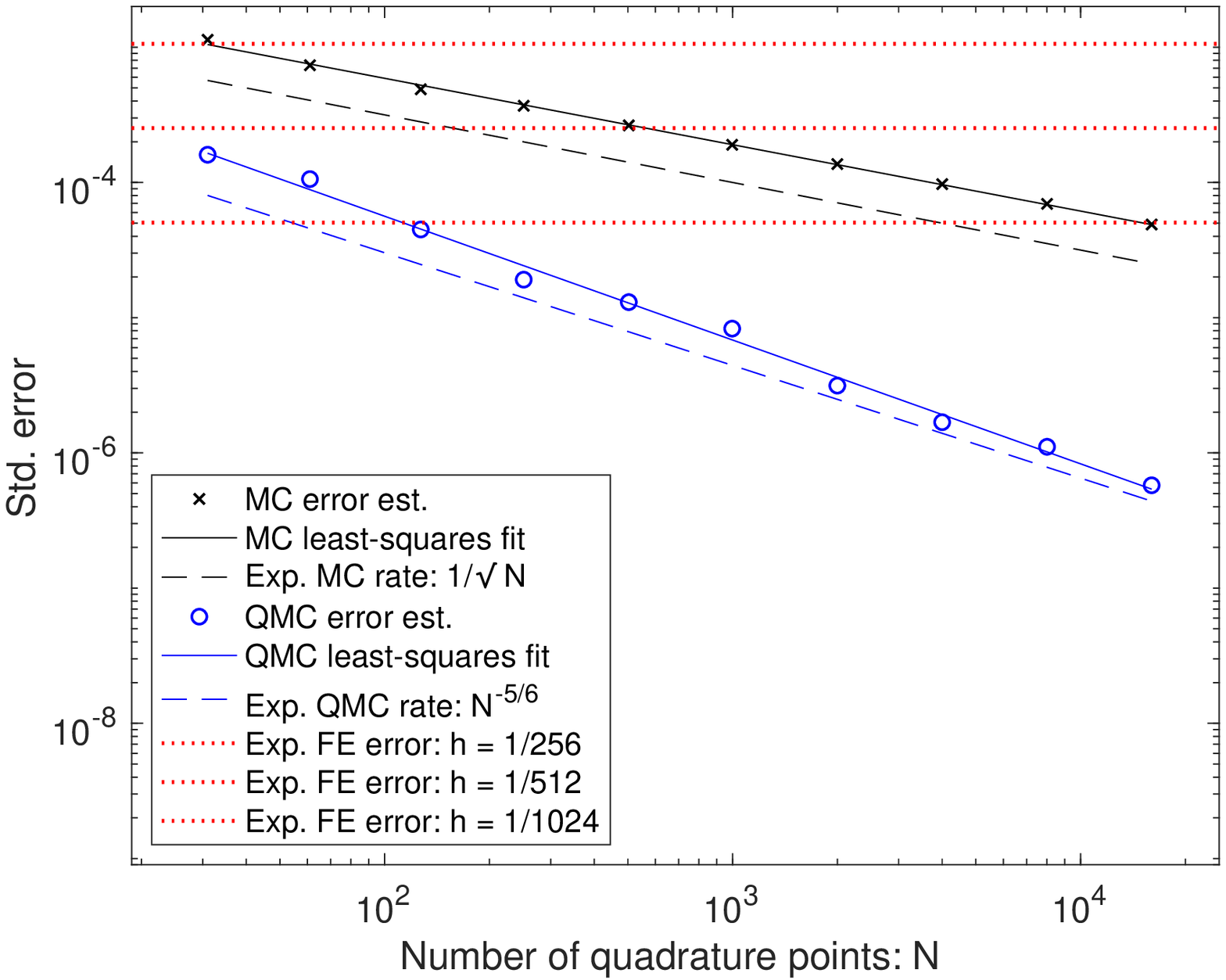}
 \hfill
 \includegraphics[width=.48\textwidth]{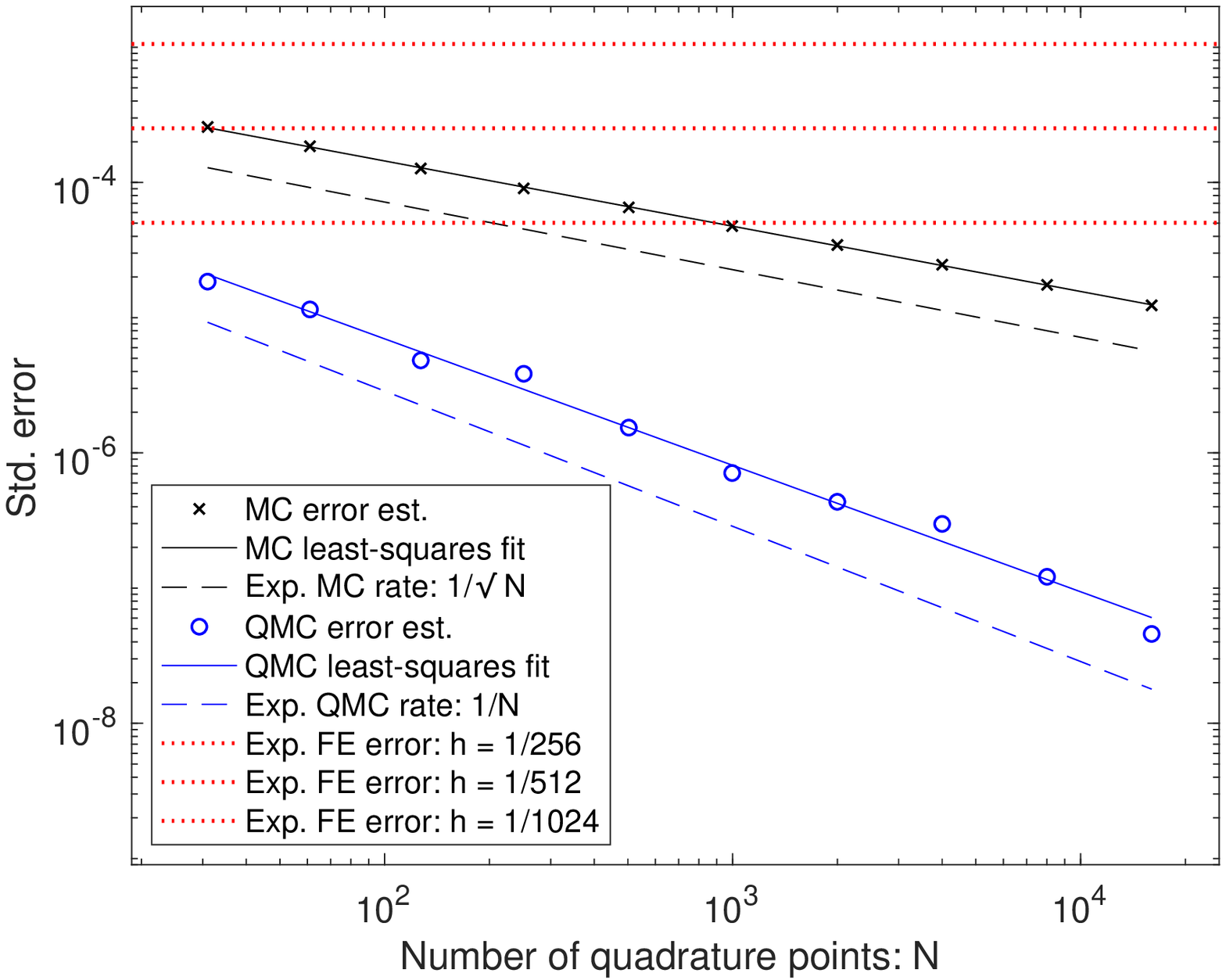}
 \caption{QMC and MC convergence for $q = 4/3$, $p\approx 3/4$ (left) and $q = 2$, $p \approx 1/2$ (right).}
 \label{fig:QMC-MC_conv}
 \end{figure}

For completeness, the computed values of 
the RMS error estimates for the shift-averaged estimate $\widehat{Q}_{R, N, s}\lambda_{s, h}$
and the convergence rates for $q = 4/3, 2$ and 3 
are given below in Table~\ref{tab:evp1_qmc}. We use the notation
``e'' to denote the base 10 exponent.
The least-squares computed rates are $-0.878$, $-0.992$ 
and $-1.01$ for $q = 4/3, 2, 3$, respectively, which are very close to the expected rates
of $-5/6\approx -0.83$, $-1$, and $-1$ from the theory. \\

\begin{table}[!h]
\centering
\caption{QMC RMS error estimates for $q = 4/3, 2, 3$.}
\label{tab:evp1_qmc}
\begin{tabular}{rr||c|c|c}
& & \multicolumn{3}{c}{RMS error estimate of $\widehat{Q}_{R, N, s}\lambda_{s, h}$}\\
$N$ & $R\times N$ & $q = 4/3$ ($p \approx 3/4$) & $q = 2$ ($p \approx 1/2$) & $q = 3$ ($p \approx 1/3$)
\\ 
\hline
251 & 2008 & 6.8\,e$-6$ & 1.4\,e$-6$ & 4.8\,e$-8$
\\
503 & 4024 & 4.6\,e$-6$ & 5.4\,e$-7$ & 1.8\,e$-8$
\\
997 & 7976 & 2.9\,e$-6$ & 2.5\,e$-7$ & 7.3\,e$-9$
\\
1999 & 15992 & 1.1\,e$-6$ & 1.5\,e$-7$ & 5.8\,e$-9$
\\
4001 & 32008 & 6.0\,e$-7$ & 1.1\,e$-7$ & 2.5\,e$-9$
\\
8009 & 64072 & 3.9\,e$-7$ & 4.3\,e$-8$ & 1.1\,e$-9$
\\
16001 & 128008 & 2.0\,e$-7$ & 1.6\,e$-8$ & 6.4\,e$-10$
\\
\hline 
\multicolumn{2}{r||}{Estimated rate} & 
$-0.878$ & $-0.992$ & $-1.01$
\end{tabular}
\end{table}

\subsection{Problem 2}
For our second example we consider more complicated coefficients, chosen
to represent a problem where the physical domain is composed of different materials,
e.g., in a simple model of a nuclear reactor the domain is composed of fuel rods surrounded by 
coolant (a gas/water mixture).
In particular, we allow the right hand side coefficient $c$ to be stochastic, and zero on
parts of the domain.
Although this problem does not satisfy Assumption~A\ref{asm:coeff}, 
we have included it to illustrate that our method also works well for a larger class of eigenvalue problems than what we were able to analyse theoretically.

In order to focus on the behaviour of the QMC error, in this numerical 
experiment we fix the truncation dimension at $s = 100$ and
FE meshwidth at $h = 1/256$, and we vary
the number of QMC points $N$.
Also, we set the number of random shifts to be $R = 8$.

Again the domain is $D = (0, 1)^2$ and we define $\bsx \coloneqq (x_1, x_2)$, but now the coefficients 
have different basis expansions for two different components of the domain 
depicted in Figure~\ref{fig:islands},  where $D$ is separated
into the union of four islands (in grey and labelled $\Df$) and the area around the islands 
$D \setminus \Df$ (in white). 
We define $\Df$  by
\[
\Df \,\coloneqq\, [\tfrac{1}{8}, \tfrac{3}{8}]^2 
\cup [\tfrac{5}{8}, \tfrac{7}{8}]^2
\cup [\tfrac{1}{8}, \tfrac{3}{8}] \times [\tfrac{5}{8}, \tfrac{7}{8}]
\cup [\tfrac{5}{8}, \tfrac{7}{8}] \times [\tfrac{1}{8}, \tfrac{3}{8}]\,.
\]
Since we use a uniform triangular  FE mesh with $h = 1/256$
the FE triangulation aligns with the boundaries of the components $\Df$.
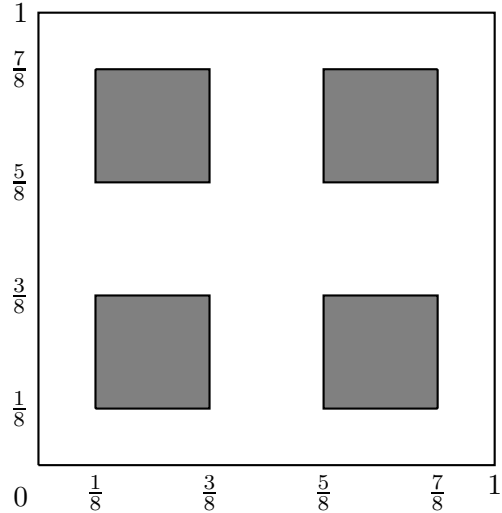
\begin{figure}[!h]
\centering
\begin{tikzpicture}[scale=6]
        \draw [thick] (0, 0) -- (1, 0) -- (1, 1) -- (0, 1) -- (0, 0);
        \fill[gray] (0.125, 0.125) rectangle (0.375, 0.375);
        \fill[gray] (0.625, 0.125) rectangle (0.875, 0.375);
        \fill[gray] (0.125, 0.625) rectangle (0.375, 0.875);
        \fill[gray] (0.625, 0.625) rectangle (0.875, 0.875);
        \draw [thick] (0.125, 0.125) -- (0.375, 0.125) -- (0.375, 0.375) -- (0.125, 0.375) -- (0.125, 0.125);
        \draw [thick] (1-0.125, 0.125) -- (1-0.375, 0.125) -- (1-0.375, 0.375) -- (1-0.125, 0.375) -- (1-0.125, 0.125);
        \draw [thick] (0.125, 1 - 0.125) -- (0.375, 1 - 0.125) -- (0.375, 1 - 0.375) -- (0.125, 1 - 0.375) -- (0.125, 1 - 0.125);
              \draw [thick] (1-0.125, 1 - 0.125) -- (1-0.375, 1 - 0.125) -- (1-0.375, 1 - 0.375) -- (1-0.125, 1 - 0.375) -- (1-0.125, 1 - 0.125);
        \draw (0, 1) node [left] {1};
        \draw(1, 0) node [below] {1};
        \draw(0, -0.07)  node [left] {0};
        \draw(0.125, 0) node [below] {$\tfrac{1}{8}$};
        \draw(0, 0.125) node [left] {$\tfrac{1}{8}$};
        \draw(0.375, 0) node [below] {$\tfrac{3}{8}$};
        \draw(0, 0.375) node [left] {$\tfrac{3}{8}$};
        \draw(1 - 0.125, 0) node [below] {$\tfrac{7}{8}$};
        \draw(0, 1 - 0.125) node [left] {$\tfrac{7}{8}$};
        \draw(1 - 0.375, 0) node [below] {$\tfrac{5}{8}$};
        \draw(0, 1 - 0.375) node [left] {$\tfrac{5}{8}$};        
        
\end{tikzpicture}
\caption{Domain $D$ with four islands forming $\Df$ (in grey).}
\label{fig:islands}
\end{figure}

Again, we would like the coefficients to have the form of an artificial KL expansion,
but we would also like the flexibility to be able to specify different decays
and scalings for each coefficient on the different components.
To achieve this, for $k \in \N$, we define the following sequences of trigonometric  functions
\begin{align*}
w_{k}(q; \bsx) \,&=\, 
\begin{cases}
\displaystyle\frac{1}{1 + (k\pi )^{q}} \sin \big(8k \pi x_1\big) \sin\big(8(k + 1)\pi x_2\big)
& \text{for } \bsx \in \Df\,,\\
0 & \text{for } \bsx \in D \setminus \Df\,,\text{ and}
\end{cases}\\
w_{k}'(q; \bsx) \,&=\, \begin{cases}
0
& \text{for } \bsx \in \Df\,,\\
\displaystyle \frac{1}{1 + (k\pi )^{q}} \sin \big(8k\pi x_1\big) \sin\big(8(k + 1)\pi x_2\big) 
& \text{for } \bsx \in D \setminus \Df\,,
\end{cases}
\end{align*}
which have wavelengths chosen such that their zeros align with the boundaries of $\Df$.
Now we use these functions to define the basis functions of the coefficients by
\begin{alignat*}{2}
a_0(\bsx) \,&=\, 
\begin{cases}
\sigdiff \,\coloneqq\,  0.01 & \text{for } \bsx \in \Df\,,\\[1mm]
\sigdiff' \,\coloneqq\, 0.011 &  \text{for } \bsx \in D \setminus\Df\,,
\end{cases}
\quad
&&a_j(\bsx) \,=\, 
\begin{cases}
\sigdiff \,w_{(j + 1)/2}(q_a; \bsx) & \text{for $j$ odd}\,,\\[1mm]
\sigdiff' \,w'_{j/2}(q_a'; \bsx) & \text{for $j$ even},\\
\end{cases}
\\[1mm]
b_0(\bsx) \,&=\, 
\begin{cases}
\sigabs \,\coloneqq\,  2 & \text{for } \bsx \in \Df\,,\\[1mm]
\sigabs' \,\coloneqq\, 0.3 &  \text{for } \bsx \in D \setminus\Df\,,
\end{cases}
\quad
&&b_j(\bsx) \,=\, 
\begin{cases}
\sigabs \,w_{(j + 1)/2}(q_b; \bsx) & \text{for $j$ odd}\,,\\[1mm]
\sigabs' \,w'_{j/2}(q_b'; \bsx) & \text{for $j$ even},\\
\end{cases}
\end{alignat*}
where the parameters $q_a, q_a', q_b, q_b' \geq 4/3$ give the different
decays of the coefficients on the different areas of the domain.

The coefficients $a,\ b,$ and $c$ represent, respectively, diffusion, absorption and fission in the reactor. 
Motivated by reactors where fission only occurs in fuel rods, 
we assume that $c$ vanishes outside $\Df$.
Furthermore, we also let $c$ be stochastic and take the same form as $a$ and $b$.
We define
\begin{align*}
c_0(\bsy) \,=\, 
\begin{cases}
\sigfiss \,\coloneqq\, 2.5 & \text{for } \bsx \in \Df\,,\\[1mm]
0 & \text{for } \bsx \in D \setminus \Df\,,
\end{cases}
\qquad
c_j(\bsy) \,=\, 
\begin{cases}
\sigfiss\, w_{(j + 1)/2)}(2; \bsx) & \text{for } j \text{ odd},\\[1mm]
0 & \text{for } j \text{ even,}
\end{cases}
\end{align*}
and then
\begin{align*}
c(\bsx, \bsy) \,=\, 
\begin{cases}
\displaystyle \sigfiss + \sum_{j = 1}^\infty y_{(j + 1)/2}\, c_{j}(\bsx) 
& \text{for } \bsx \in \Df,\\
0 & \text{for } \bsx \in D \setminus \Df\,.
\end{cases}
\end{align*}

The factors $\sigdiff,\ \sigdiff',\ \sigabs,\ \sigabs',$ and $\sigfiss$ are chosen
so that the mean of $a$, $b$ and $c$ on the different components of the domain correspond
to physically relevant values for the respective cross-sections for a nuclear reactor (see \cite{VdB15}).

Also, notice that the coefficients $a$, $b$ and $c$ will be correlated.
The practical motivation for this is that
the randomness at a point is generated by uncertainty
in the material properties at that point, and thus will affect each coefficient in the same manner.
However, for $\bsx \in \Df$ and $\bsx' \in D \setminus \Df$ the values of the coefficients
$a(\bsx, \bsy)$, $b(\bsx, \bsy)$, $c(\bsx, \bsy)$ are not correlated with any of $a(\bsx', \bsy)$, $b(\bsx', \bsy)$,
$c(\bsx', \bsy)$.

We have that
\begin{align*}
\nrm{a_j}{L^\infty(D)} \,&=\,
\begin{cases}
\displaystyle\frac{\sigdiff}{1 + (\frac{j + 1}{2}\pi)^{q_a}} & \text{if $j$ is odd,}\\[4mm]
\displaystyle\frac{\sigdiff'}{1 + (\frac{j}{2}\pi)^{q_a'}} & \text{if $j$ is even,}
\end{cases}
\quad \text{and}\\
\nrm{b_j}{L^\infty(D)} \,&=\,
\begin{cases}
\displaystyle\frac{\sigabs}{1 + (\frac{j + 1}{2}\pi)^{q_b}} & \text{if $j$ is odd,}\\[4mm]
\displaystyle\frac{\sigabs'}{1 + (\frac{j}{2}\pi)^{q_b'}} & \text{if $j$ is even.}
\end{cases}
\end{align*}
Letting $q \coloneqq \min(q_a, q_a', q_b, q_b')$, a quick calculation
gives that
\[
\max\left(\nrm{a_j}{L^\infty(D)},\ \nrm{b_j}{L^\infty(D)}\right)
\,<\, \frac{\max(\sigdiff,\ \sigdiff',\ \sigabs,\ \sigabs')}{(\pi j/2)^q}
\,=\, \frac{2}{(\pi j/2)^q}\,.
\]
Hence, since $q \geq 4/3$ the coefficients are bounded from above and below independently of $\bsy$,
and we can take 
\[
\amin =  0.01\left(1 - \zeta(q) \left(\frac{2}{\pi}\right)^q\right)
\quad \text{and} \quad
\amax = 2.5\left(1 + \zeta(q) \left(\frac{2}{\pi}\right)^q\right)\,.
\]
Again, Assumption~A\ref{asm:coeff}.\ref{itm:summable} holds for $p \in (1/q, 1)$, and
so for this example the convergence rate will be determined by the minimum
of the four rates $q_a, q_a', q_b, q_b'$.

For this example, we also provide results for a linear functional 
applied to $u_1(\bsy)$. We define $\calG \in V^*$ to be the linear functional that computes
the average neutron flux over one of the islands $\Dff \coloneqq [\tfrac{1}{8}, \tfrac{3}{8}]^2$, 
which is given by
\[
\calG(u_1(\bsy)) \,\coloneqq\, \frac{1}{|\Dff|} \int_{\Dff} u_1(\bsx, \bsy) \,\rd \bsx\,.
\]

The results for this example for different combinations of values 
$q_a,\ q_a',\ q_b,\ q_b' \in \{4/3, 2\}$ are presented in Table~\ref{tab:eg2_1}.
Each column corresponds to a different choice of decays, and presents results
for both the eigenvalue and $\calG$ applied to the eigenfunction.
We present RMS error estimates for increasing $N$, followed by 
the estimated convergence rate.
Recall that $R= 8$ is the number of random shifts used.

\begin{table}[!h]
\centering
\caption{Quadrature results for the approximation of $\bbE_\bsy[\lambda_1]$ 
and $\bbE_\bsy[\calG(u_1)]$ for different decays.}
\label{tab:eg2_1}
\begin{tabular}{rr||c|c||c|c||c|c}
 & & \multicolumn{6}{|c}{RMS error estimate}\\
\multicolumn{2}{r||}{$q_a,\ q_a',\ q_b,\ q_b'$} &
\multicolumn{2}{|c||}{4/3, 4/3, 4/3, 4/3} & \multicolumn{2}{|c||}{2, 4/3, 2, 4/3} & \multicolumn{2}{|c}{2, 2, 2, 2}
\\
\hline
\hline
$N$ & $R \times N$ & 
$\lambda_1$ & $\calG(u_1)$ & $\lambda_1$ & $\calG(u_1)$ & $\lambda_1$ & $\calG(u_1)$
\\
\hline
251 & 2008 & 
2.4\,e$-7$ & 5.1\,e$-6$ &
6.7\,e$-7$ & 5.3\,e$-6$ &
6.4\,e$-7$ & 5.7\,e$-6$
\\
503 & 4024 & 
9.7\,e$-8$ & 2.8\,e$-6$ &
3.6\,e$-7$ & 3.0\,e$-6$ &
4.5\,e$-7$ & 3.0\,e$-6$
\\
997 & 7976  &
9.1\,e$-8$ & 1.0\,e$-6$ &
2.9\,e$-7$ & 8.4\,e$-7$& 
2.7\,e$-7$ & 8.4\,e$-7$
\\
1999 & 15992  &
3.9\,e$-8$ & 6.9\,e$-7$ &
1.4\,e$-7$ & 6.9\,e$-7$ &
1.3\,e$-7$ & 7.1\,e$-7$
\\
4001 & 32008 &
2.0\,e$-8$ & 3.8\,e$-7$ &
4.2\,e$-8$ & 3.6\,e$-7$ & 
3.8\,e$-8$ & 3.6\,e$-7$
\\
8009 & 64072 &
1.4\,e$-8$ & 1.9\,e$-7$ &
5.1\,e$-8$ & 1.8\,e$-7$ & 
4.8\,e$-8$ & 1.8\,e$-7$
\\
16001 & 128008 &
1.1\,e$-8$ & 7.4\,e$-8$ &
1.7\,e$-8$ & 7.1\,e$-8$ & 
9.0\,e$-9$ & 7.5\,e$-8$
\\
\hline
\hline
\multicolumn{2}{r||}{Estimated rate} &
$-0.748$ & $-0.983$ &
$-0.871$ & $-0.997$ & 
$-0.993$ & $-1.003$
\end{tabular}
\end{table}

Even though this example does not satisfy the conditions of
our theory, we still obtain good
convergence rates for the QMC error. 
For the first two columns of Table~\ref{tab:eg2_1}, at least one decay has the value $4/3$, hence,
we expect summability with $p \approx 3/4$ and a convergence rate of $-5/6 \approx -0.83$.
When $q_a = q_a' = q_b = q_b' = 4/3$ (first column in Table~\ref{tab:eg2_1}), the eigenvalue error 
decays slower than $N^{-5/6}$ but the eigenfunction error still decays like $N^{-1}$.
Whereas in column 2, for the eigenvalue approximation,
we observe a QMC convergence rate that is slower than $N^{-1}$, but which is
still faster than the expected rate of $-5/6 \approx -0.83$.
Surprisingly, for the eigenfunction results we observe a QMC convergence rate 
that is almost $N^{-1}$ for all of our test cases, regardless of the decays.
When $q_a  = q_a' = q_b = q_b' = 2$ (last column in Table~\ref{tab:eg2_1}) 
we would expect summability with $p\approx 1/2$
and convergence arbitrarily close to $N^{-1}$, which is observed in the errors for both the 
eigenvalue and eigenfunction approximations.
For all of the other cases of $q_a,\ q_a',\ q_b,\ q_b' \in \{4/3, 2\}$ not presented, 
some but not all of the decays equal $4/3$ and 
we observe very similar results to column 2 of Table~\ref{tab:eg2_1}. In particular, the
convergence rates are almost identical. 

\section{Conclusion}
\label{sec:con}
We have presented a truncation-FE-QMC algorithm for approximating the expectation
of the smallest eigenvalue, and linear functionals of the corresponding eigenfunction,
of a stochastic eigenvalue problem. Along with the method, we have presented a full
analysis of the three components of the error and the final error bound has
the same decay rates as the corresponding elliptic source problem for all 
$p \in (0, 1)$. Throughout the analysis we also proved two key results. First,
we proved that the eigenvalue gap is bounded away from 0 uniformly in $\bsy$.
Second, we derived bounds on the derivatives of the smallest eigenvalue and the corresponding
eigenfunction.

Our first numerical example presents results that match almost exactly with what is predicted
by our theoretical analysis.
The second example goes beyond our theoretical setting and illustrates that the method is more
robust than our theory suggests.

Regarding future work, one possible avenue is to use higher order QMC rules
(which converge at a rate faster than $N^{-1}$) to approximate the expectation of the
minimal eigenvalue.
The use of higher order QMC rules requires higher order smoothness of the integrand,
but we have already proven bounds on the higher order derivatives of $\lambda_1$ (and $u_1$)
in Lemma~\ref{lem:dlambda}.
As such, we expect that the theory for using higher order QMC rules for
this eigenvalue problem should work quite easily.
Of course, to balance the faster quadrature convergence with the
discretisation error one should use a more accurate FE method,
which would be more challenging if the domain $D$ was not either convex or had
a smooth boundary.
Additionally, one could consider embedding our truncation-FE-QMC algorithm in a
multilevel framework, which we expect would significantly reduce the cost but would
also require further theoretical analysis.

\medskip
\noindent {\bf Acknowledgement.} \  We thank John Toland (University of Bath) for a very useful discussion which helped us formulate the proof in \S \ref{subsec:gap}.  
We gratefully acknowledge the financial support from the Australian Research Council (FT130100655, DP150101770, DP180101356), the Taiwanese National Center for Theoretical Sciences (NCTS) -- Mathematics Division, the Statistical and Applied Mathematical Sciences Institute (SAMSI) 2017 Year-long Program on Quasi-Monte Carlo and High-Dimensional Sampling Methods for Applied Mathematics, and the Mathematical Research Institute (MATRIX) 2018 program on the Frontiers of High Dimensional Computation.

\newpage
\begin{appendix}
\normalsize

\section{Proof of Theorem \ref{thm:fe_err}}

This proof follows the same structure as \cite{SF73} to bound the FE error, 
but in addition we have to track
the dependence of all constants on $\bsy$. We have opted for the
classical min-max argument as opposed to the Babu\v{s}ka--Osborn theory 
\cite{BO87,BO89,BO91}, because it is more elementary and allows us to determine 
the explicit influence of the constants.
We begin with some preliminary definitions.

For $\bsy \in U$ and $h > 0$, define the orthogonal projection 
$P_h : V \rightarrow V_h$ of $u \in V$ with respect to the inner
product $\calA(\bsy; \cdot, \cdot)$ on $V$ by
\[
\calA(\bsy; u - P_h u, v_h) \,=\, 0\,,
\quad \text{for all } v_h \in V_h\,.
\]
Although $P_h$ depends on $\bsy$ through the
bilinear form $\calA(\bsy; \cdot, \cdot)$, we suppress this $\bsy$
dependence in the notation. 
Let us denote the energy norm by $\nrm{\cdot}{\calA(\bsy)} = \sqrt{\calA(\bsy; \cdot,
\cdot)}$. Then it is easy to verify that due to the 
$\calA$-orthogonality of $P_h$ we have
\[
\nrm{u - P_h u}{\calA(\bsy)} \,=\, \inf_{v_h \in V_h} \nrm{u - v_h}{\calA(\bsy)}\,.
\]
Due to \eqref{eq:A_coerc} and \eqref{eq:A_bounded} the energy norm
 is equivalent to the $V$-norm:
\begin{align}
\label{eq:A-V-norm_equiv}
\sqrt{\amin}\nrm{v}{V} \,\leq\, \nrm{v}{\calA(\bsy)} \,\leq\,
\sqrt{\frac{\amin\lambdabar}{\chi_1}}
\nrm{v}{V}\,,
\quad \text{for all } v \in V.
\end{align}

Analogously to the min-max principle \eqref{eq:minmax},
when the $k$-dimensional subspaces $S_k$ are restricted to $V_h$ we have
the min-max representation for the FE eigenvalues
\begin{align}
\label{eq:minmax_fe}
\lambda_{k, h}(\bsy) \,=\, \min_{\substack{S_k \subset V_h\\ \dim(S_k) = k}}
\max_{0 \neq u_h \in S_k} \frac{\calA(\bsy; u_h, u_h)}{\calM(u_h, u_h)}\,.
\end{align}

The strategy of the proof of Theorem~\ref{thm:fe_err} is to first bound the difference between $u(\bsy)$ and its 
projection $P_h u(\bsy)$, which is fairly straightforward and follows from the
FE results for elliptic source problems. The difficulty lies in the fact that
the projections $P_h u(\bsy)$ are not the same as
the FE eigenfunctions $u_h(\bsy)$.
However, they are close. The next stage of the proof is to bound the
eigenvalue and eigenfunction errors \eqref{eq:lambda_fe_err}, \eqref{eq:u_fe_err}
in terms of the projection error.
For the eigenvalue error a key ingredient is the classical min-max principle \eqref{eq:minmax_fe}.
Combining the FE  error bounds with the projection error bounds
yields the required results.

\begin{Lemma}
\label{lem:proj_err}
Let $\bsy \in U$. The projection of $u_1(\bsy) \in E(\bsy,
\lambda_1(\bsy)) \subset V$ into $V_h$ satisfies
\begin{align}
\label{eq:proj_err_V}
\nrm{u_1(\bsy) - P_h u_1(\bsy)}{V} \,&\leq\, C h\,,
\end{align}
where $C > 0$ is independent of $\bsy$.
\end{Lemma}

\begin{proof}
The projection $P_h u_1(\bsy)$ can be equivalently
viewed as the FE approximation of an elliptic source problem.
Indeed, the variational eigenproblem \eqref{eq:varevp} 
for the eigenpair $(\lambda_1(\bsy), u_1(\bsy))$ can be written as
\[
\calA(\bsy; u_1(\bsy), v) \,=\,
\langle f(\bsy), v\rangle
\quad \text{for all } v \in V\,,
\]
where $f(\bsy) =\lambda_1(\bsy)c\cdot u_1(\bsy)$
is now assumed fixed. The FE approximation problem 
of this seeks
$\tilde{u}_h(\bsy) \in V_h$ such that
\[
\calA(\bsy; \tilde{u}_h(\bsy), v_h) \,=\,
\left\langle f(\bsy), v_h
\right\rangle
\quad \text{for all } v_h \in V_h\,,
\]
for which due to $\calA$-orthogonality the solution is exactly
the projection of the eigenfunction: $\tilde{u}_{h}(\bsy) = P_h u_1(\bsy)$.
This allows us to bound the projection error using the results from
elliptic source problems.
In particular, our differential operator fits the setting of 
affine parametric operator equations from \cite{DKLeGNS14}.
Since $u_1(\bsy) \in Z$, it follows that $f(\bsy) \in L^2(D)$ for all $\bsy$.
The spaces $V_h$ satisfy the approximation property \eqref{eq:fe_approx}, thus
by Theorem~2.4 in \cite{DKLeGNS14} we have
\begin{equation}
\label{eq:proj_err_V_1}
\nrm{ u_1(\bsy) - P_h u_1(\bsy)}{V}
\,\leq\,
C'\nrm{f(\bsy)}{L^2(D)}h\,,
\end{equation}
with constant $C'$ independent of $\bsy$ and $h$.

To bound $\nrm{f(\bsy)}{L^2(D)}$, 
we use the upper bound in \eqref{eq:lambda_bnd},
the bound \eqref{eq:coeff_bounded} on $c$
and then the fact that $\nrm{u_1(\bsy)}{\calM} = 1$ to give
\[
\nrm{f(\bsy)}{L^2(D)}\,\leq\, 
\lambda_1(\bsy) \nrm{c}{L^\infty(D)}^{1/2}\nrm{u_1(\bsy)}{\calM}
\,\leq\,
\amax^{1/2} \lambdabar
\]
Substituting this into \eqref{eq:proj_err_V_1} we have our desired result
with a constant $C$ independent of $\bsy$ and $h$.
\end{proof}

We now estimate the eigenvalue error (Lemma~\ref{lem:eval_proj_err})
and the eigenfunction error (Lemma~\ref{lem:evec_proj_err}) in terms of the
projection error that was estimated in Lemma~\ref{lem:proj_err}.
Lemma~\ref{lem:eval_sep} relates to the gap between
the FE eigenvalues and $\lambda_1(\bsy)$, and is used in the 
proof of Lemma~\ref{lem:evec_proj_err}.
\begin{Lemma}
\label{lem:eval_proj_err}
Let $\bsy \in U$ and let $h > 0$ be sufficiently small independently of $\bsy$.
Then
\begin{align}
\label{eq:eval_proj_err}
\left|\lambda_1(\bsy) - \lambda_{1, h}(\bsy)\right|
\,\leq\, C\nrm{u_1(\bsy) - P_h u_1(\bsy)}{V}^2\,,
\end{align}
where $C> 0$ is independent of $\bsy$.
\end{Lemma}

\begin{proof}
To prove the result we apply the min-max principle to $\lambda_{1, h}(\bsy)$
and choose the particular subspace $S_{1, h}(\bsy) \coloneqq \mathrm{span} (P_h u_1(\bsy))$,
which is a one-dimensional subspace of $V$ provided that $P_h u_1(\bsy) \neq 0$.
To prove that $\text{dim}\big( S_{1, h}(\bsy) \big) = 1$ suppose for
contradiction that $P_h u_1(\bsy) = 0$, then by \eqref{eq:RQ}
\[
1 
\,=\, \frac{1}{\lambda_1(\bsy)} \nrm{u_1(\bsy)}{\calA(\bsy)}^2 \,=\, \frac{1}{\lambda_1(\bsy)}
\nrm{u_1(\bsy) - P_h u_1(\bsy)}{\calA(\bsy)}^2 \,.
\]
Then using the equivalence of norms \eqref{eq:A-V-norm_equiv}, together
with the lower bound in \eqref{eq:lambda_bnd} and Lemma
\ref{lem:proj_err} we get
\begin{equation}\label{eq:appendix_contra}
1 \,\le
\underbrace{
\frac{\amin\lambdabar}{\chi_1\lambdaunder}
}_{\textstyle
  =:C'}
\nrm{u_1(\bsy) - P_h u_1(\bsy)}{V}^2 \,\le\, C' C^2 h^2\,,
\end{equation}
where $C>0$ is the constant from Lemma \ref{lem:proj_err} and both $C$
and $C'$ are independent of $\bsy$ and $h$. So for $h < (\sqrt{C'}\,C)^{-1}$, this leads to a contradiction and $\text{dim}\big( S_{1, h}(\bsy) \big) = 1$.
Therefore, we can choose $S_{1, h}(\bsy)$ in \eqref{eq:minmax_fe} to give the inequality
\begin{align}
\label{eq:minmax_bound}
\lambda_{1, h}(\bsy) 
\,\leq\, 
\max_{0 \neq v_h\in S_{1, h}(\bsy)} \frac{\calA(\bsy; v_h, v_h)}{\calM(v_h, v_h)}
\,=\, 
\frac{\calA(\bsy; P_h u_1(\bsy), P_h u_1(\bsy))}{\calM(P_h u_1(\bsy), P_h u_1(\bsy))}\,.
\end{align}


Using the fact that the norm of the projection is bounded by 1, the numerator is bounded by
\begin{equation}
\label{eq:numer_bnd}
\calA(\bsy; P_h u_1(\bsy), P_h u_1(\bsy))
\,\leq\,
\calA(\bsy; u_1(\bsy), u_1(\bsy)) \,=\, 
\lambda_1(\bsy)\,,
\end{equation}
where for the equality in the last step we have used \eqref{eq:RQ}.

Expanding the denominator in \eqref{eq:minmax_bound} gives
\begin{align*}
\calM(P_h u_1(\bsy), P_h u_1(\bsy))
\,=\, \nrm{u_1(\bsy)}{\calM}^2 - 2\calM\big(u_1(\bsy), u_1(\bsy) - P_h u_1(\bsy)\big) 
\, + \,
\nrm{u_1(\bsy) - P_h u_1(\bsy)}{\calM}^2\,.
\end{align*}
The first term on the right is 1 by the normalisation of $u_1(\bsy)$ and
the last term is positive, so we can bound $\calM(P_h u_1(\bsy), P_h u_1(\bsy))$
from below by
\begin{equation*}
\calM(P_h u_1(\bsy), P_h u_1(\bsy)) 
\,\geq\, 1 - 2\calM(u_1(\bsy), u_1(\bsy) - P_h u_1(\bsy))\,.
\end{equation*}
Then, using the fact that $u_1(\bsy)$ is an eigenfunction
satisfying \eqref{eq:varevp} and using also the $\calA$-orthogonality of the projection $P_h$
we have
\begin{align}
\label{eq:denom_bnd}
\calM(P_h u_1(\bsy), P_h u_1(\bsy)) 
\,&\geq\, 1 - \frac{2}{\lambda_1(\bsy)} 
\nrm{u_1(\bsy) - P_h u_1(\bsy)}{\calA(\bsy)}^2
\nonumber\\
&\geq\, 1 - 2C' \nrm{u_1(\bsy) - P_h u_1(\bsy)}{V}^{2}\,,
\end{align}
with $C'>0$ as in \eqref{eq:appendix_contra}, using again the lower
  bound in \eqref{eq:lambda_bnd} and the equivalence of norms 
\eqref{eq:A-V-norm_equiv}.
For $h$ sufficiently small independently of $\bsy$, the right hand side of
\eqref{eq:denom_bnd} is positive and we can substitute it together with 
\eqref{eq:numer_bnd} into \eqref{eq:minmax_bound}.
Rearranging the resulting inequality gives
\[
\left|\lambda_1(\bsy) - \lambda_{1, h}(\bsy)\right|
\,\leq\,
2C' \lambda_{1, h}(\bsy)  \nrm{u_1(\bsy) - P_h u_1(\bsy)}{V}^{2}\,.
\]

Now all that remains is to show that $\lambda_{1, h}(\bsy)$
can be bounded from above independently of $\bsy$ and $h$.
Analogously to \eqref{eq:lambda_bnd}, using the FE min-max
representation \eqref{eq:minmax_fe} we have 
\[
\lambda_{1, h}(\bsy) \,\leq\, \frac{\amax}{\amin}(\chi_{1, h} + 1)\,,
\]
where $\chi_{1, h}$ corresponds to the smallest eigenvalue of the
negative Laplacian on $D$, discretised in the FE space $V_h$ with
boundary condition \eqref{eq:DBC}. It is well-known, see
e.g. \cite[Theorem 10.4]{Boffi10}, 
that in the current setting we have $|\chi_{1, h} - \chi_1| \le C'' h^2$ with $C'' > 0$ independent of $h$. 
Thus, for $h$ sufficiently small and independent of $\bsy$, there exists a 
constant such that $\lambda_{1, h}(\bsy)$ can be bounded independent of 
$\bsy$ and $h$ as required.
\end{proof}

\begin{Lemma}
\label{lem:eval_sep}
Let $\bsy \in U$ and $h > 0$. 
Then for all $k = 2, 3, \ldots, M_h = \dim (V_h)$
\begin{align}
\label{eq:eval_sep}
\frac{\lambda_1(\bsy)}{\lambda_{k, h}(\bsy) - \lambda_1(\bsy)}
\,\leq\,  \rho
\,\coloneqq\, \frac{\lambdabar}{\delta} \,,
\end{align}
with $\delta$ as in Proposition~\ref{prop:simple}
and $\lambdabar$ as in \eqref{eq:lambda_bnd}.
\end{Lemma}

\begin{proof}

Since the FE eigenvalues converge to the true eigenvalues
from above, it follows from Proposition~\ref{prop:simple} and the upper bound 
on $\lambda_1(\bsy)$ in \eqref{eq:lambda_bnd} that
\[
\lambda_{2, h}(\bsy) - \lambda_1(\bsy)  \,\geq\,
\lambda_{2}(\bsy) - \lambda_1(\bsy)  \,\geq\,
\delta
\,\geq\, \delta\frac{\lambda_1(\bsy)}{\lambdabar}\,,
\]
which completes the proof because
the left hand side of \eqref{eq:eval_sep}
attains its maximum when $k = 2$.
\end{proof}

\begin{Lemma}
\label{lem:evec_proj_err}
Let $\bsy \in U$ and $h > 0$.
Then
\begin{align}
\label{eq:evec_proj_err}
\nrm{u_1(\bsy) - u_{1, h}(\bsy)}{\calM}
\leq C\nrm{u_1(\bsy) - P_h u_1(\bsy)}{\calM}\,,
\end{align}
where $C$ is independent of $\bsy$ and $h$.
\end{Lemma}

\begin{proof}
The FE eigenfunctions form an orthonormal basis for $V_h$
with respect to $\calM(\cdot, \cdot)$, and so the projection of 
$u_1(\bsy)$ can be written as
\[
P_h u_1(\bsy) \,=\, \sum_{k = 1}^{M_h} 
\alpha_{k, h}(\bsy)u_{k, h}(\bsy)\,,
\]
where $\alpha_{k, h}(\bsy) \coloneqq \calM\left(P_h u_1(\bsy), u_{k, h}(\bsy)\right)$.

The key coefficient in this expansion is $\alpha_{1, h}(\bsy)$. 
If we assume that $\alpha_{1, h}(\bsy) \geq 0$ (which we can always
ensure by controlling the sign of $u_{1, h}(\bsy)$), then the size of $\alpha_{1, h}(\bsy)$
gives a measure of how close $P_h u_1(\bsy)$ is to $u_{1, h}(\bsy)$.
As a first step towards \eqref{eq:evec_proj_err}, consider the difference
\begin{align}
\label{eq:evec-alpha}
\nrm{u_1(\bsy) - \alpha_{1, h}(\bsy)u_{1, h}(\bsy)}{\calM}
\, \leq\, 
\nrm{u_1(\bsy) - P_h u_1(\bsy)}{\calM}
\, + \, \nrm{P_h u_1(\bsy) - \alpha_{1, h}(\bsy)u_{1, h}(\bsy)}{\calM}\,.
\end{align}

The first term is exactly our target upper bound in \eqref{eq:evec_proj_err}. 
The square of the second term can be written as
\[
\nrm{P_h  u_1(\bsy) - \alpha_{1, h}(\bsy)u_{1, h}(\bsy)}{\calM}^2
\,=\, \sum_{k = 2}^{M_h} \alpha_{k, h}(\bsy)^2\,.
\]
By \cite[Lemma 6.4]{SF73} (or as is easily verified),
we can replace $\alpha_{k, h}(\bsy)$ by
\[
\alpha_{k, h}(\bsy) \,=\, \frac{\lambda_1(\bsy)}{\lambda_{k, h}(\bsy) - \lambda_1(\bsy)}
\calM\left(u_1(\bsy) - P_h u_1(\bsy), u_{k, h}(\bsy)\right)\,.
\]
Hence, using Lemma~\ref{lem:eval_sep} and letting $Q_h$ denote the
  $\calM$-orthogonal projection onto $V_h$, we can bound
\begin{align*}
\nrm{P_h  u_1(\bsy) - \alpha_{1, h}(\bsy)u_{1, h}(\bsy)}{\calM}^2
\,&\leq\, 
\rho^2\sum_{k = 2}^{M_h}
\calM\left(u_1(\bsy) - P_h u_1(\bsy), u_{k, h}(\bsy)\right)^2\\
&= \, 
\rho^2\sum_{k = 2}^{M_h}
\calM\Big(Q_h\big(u_1(\bsy) - P_h u_1(\bsy)\big), u_{k, h}(\bsy)\Big)^2\\[1ex]
&\leq\,
\rho^2\nrm{Q_h\big(u_1(\bsy) - P_h u_1(\bsy)\big)}{\calM}^2 \ \ 
\leq \ \  \rho^2 \nrm{u_1(\bsy) - P_h u_1(\bsy)}{\calM}^2\,.
\end{align*}
Thus, our intermediate bound \eqref{eq:evec-alpha}
can be written as
\begin{equation}
\label{eq:evec-alpha-2}
\nrm{u_1(\bsy) - \alpha_{1, h}(\bsy)u_{1, h}(\bsy)}{\calM}
\,\leq\, 
(1 + \rho)\nrm{u_1(\bsy) - P_h u_1(\bsy)}{\calM}\,.
\end{equation}

The final step to prove \eqref{eq:evec_proj_err} is to show that 
$\alpha_{1, h}(\bsy)$ is close to 1.
To that end, using the reverse triangle inequality and the
  fact that both $u_1(\bsy)$ and $u_{1, h}(\bsy)$ are normalised we
have the following measure of how close $\alpha_{1, h}(\bsy)$ is to 1:
\begin{equation}
\label{eq:alpha_1_bound}
\nrm{u_1(\bsy) - \alpha_{1, h}(\bsy) u_{1, h}(\bsy)}{\calM}
\,\geq\,
\left|\nrm{u_1(\bsy)}{\calM} - \alpha_{1, h}(\bsy)\nrm{u_{1, h}(\bsy)}{\calM}\right|
\,=\,
|1 - \alpha_{1, h}(\bsy)|\,.
\end{equation}
Finally, combining \eqref{eq:evec-alpha-2} and \eqref{eq:alpha_1_bound},
it follows by the triangle inequality that
\begin{align*}
\nrm{u_1(\bsy) - u_{1, h}(\bsy)}{\calM}
&\,\leq\,
\nrm{u_1(\bsy) - \alpha_{1, h}(\bsy) u_{1,h}(\bsy)}{\calM}
+ |1 - \alpha_{1, h}(\bsy)|\nrm{u_{1, h}(\bsy)}{\calM}
\\
&\leq\,
2(1 + \rho)\nrm{u_1(\bsy) - P_h u_1(\bsy)}{\calM}\,.
\end{align*}
\end{proof}

We now have all of the ingredients needed to prove our FE error bounds.

\begin{namedproof}{Theorem \ref{thm:fe_err}.}
The eigenvalue error bound \eqref{eq:lambda_fe_err} follows
directly from Lemmas~\ref{lem:proj_err} and \ref{lem:eval_proj_err}.
All of the constants
involved are independent of $\bsy$ and $h$, so the final constant is also.

For the bound on the eigenfunction error, we use \cite[Lemma 3.1]{BO89}
to write
\[
\nrm{u_1(\bsy) - u_{1, h}(\bsy)}{\calA(\bsy)}^2
\,=\, \lambda_{1, h}(\bsy) - \lambda_1(\bsy)
+ \lambda_1(\bsy)\nrm{u_1(\bsy) - u_{1, h}(\bsy)}{\calM}^2\,.
\]
Then, using the lower bound in the norm equivalence
\eqref{eq:A-V-norm_equiv}, as well as the upper bound on
$\lambda_1(\bsy)$ in \eqref{eq:lambda_bnd} and the fact that
$\lambda_{1,h} \ge \lambda_1$ this leads to
\begin{align}
\label{eq:evec_error_Vnorm}
\nrm{u_1(\bsy) - u_{1, h}(\bsy)}{V}^2
\, \leq\, \frac{1}{\amin} |\lambda_{1, h}(\bsy) - \lambda_1(\bsy)|
\, +\, 
\frac{\lambdabar}{\amin}
\nrm{u_1(\bsy) - u_{1, h}(\bsy)}{\calM}^2\,.
\end{align}
Now, combining Lemma~\ref{lem:evec_proj_err} with the upper bound
  in the norm equivalence \eqref{eq:nrms_equiv} and Poincar\'e's
  inequality \eqref{eq:poin} we can estimate
\begin{align*}
\nrm{u_1(\bsy) - u_{1, h}(\bsy)}{\calM}
\,\leq\, C \nrm{u_1(\bsy) - P_h u_1(\bsy)}{\calM}
\, \leq\, C \left(\frac{\amax}{\chi_1}\right)^{1/2} \nrm{u_1(\bsy) - P_h
  u_1(\bsy)}{V} \,,
\end{align*}
where $C>0$ is the constant from
  Lemma~\ref{lem:evec_proj_err}. Using this bound in
  \eqref{eq:evec_error_Vnorm}  together with Lemmas \ref{lem:proj_err} and
  \ref{lem:eval_proj_err} we get
\[
\nrm{u_1(\bsy) - u_{1, h}(\bsy)}{V}
\,\leq\, C' h \,,
\]
for some constant $C'>0$ depending only on $\amax$, $\amin$ and
$\chi_1$, as well as the constants in Lemmas \ref{lem:proj_err}, 
\ref{lem:eval_proj_err} and \ref{lem:evec_proj_err}, which are all
independent of $\bsy$ and $h$. This completes the proof of \eqref{eq:u_fe_err}.

Having established the error in the $V$-norm, we use the classical
Aubin-Nitsche duality argument to prove the final error bound \eqref{eq:Gu_fe_err}. 
Let $\calG \in H^{-1 + t}(D)$ and consider the dual problem:
Find $v_\calG(\bsy) \in V$ such that 
\begin{equation}
\label{eq:dual}
\calA(\bsy; w, v_\calG(\bsy)) \,=\, \calG(w)
\quad \text{for all } w \in V\,.
\end{equation}
Due to the symmetry of $\calA(\bsy; \cdot, \cdot)$, the standard theory for 
elliptic boundary value problems guarantees the existence of a unique
solution $v_\calG(\bsy)  \in V$ such that $\nrm{v_\calG(\bsy)}{V} \leq
C_1 \nrm{\calG}{V^*}$. In fact, it can also be shown that 
$v_\calG(\bsy) \in Z^t := V \cap H^{1 + t}(D)$ with 
$\nrm{v_\calG(\bsy)}{Z^t} \leq C_2$. The fact that $C_1,C_2 >0$ are
independent of $h$ is classical; the independence of $\bsy$ has been shown in
\cite{KSS12}. Thus, using
the norm equivalences in
\eqref{eq:A-V-norm_equiv} and the best-approximation property of $P_h
v_\calG(\bsy)$ in the energy norm we get
\begin{equation}
\label{eq:fe_dual_err}
\nrm{v_\calG(\bsy) - P_h v_{\calG}(\bsy)}{V}
\,\leq\, 
\sqrt{\frac{\lambdabar}{\chi_1}}
\inf_{w_h \in V_h}\nrm{v_\calG(\bsy) - w_h}{V} \,\leq\,
C_3 h^t\,,
\end{equation}
with $C_3>0$ again independent of $\bsy$ and $h$, where in the last
inequality we have used the approximation property
\eqref{eq:fe_approx} and the bound on the $Z^t$-norm.

Letting $w = u_1(\bsy) - u_{1, h}(\bsy)$ in \eqref{eq:dual},
by the definition of $P_h$ and the boundedness of the
 bilinear form \eqref{eq:A_bounded}
 there finally exists a constant $C_4 > 0$ independent of $\bsy$
   and $h$ such that
\begin{align*}
\left|\calG(u_1(\bsy)) - \calG(u_{1, h}(\bsy))\right|
\,&=\, \left|\calG(u_1(\bsy) - u_{1, h}(\bsy))\right|
\,=\, \left|\calA(\bsy; u_1(\bsy)) - u_{1, h}(\bsy), v_\calG(\bsy))\right|\\
&=\, \left|\calA(\bsy; u_1(\bsy) - u_{1, h}(\bsy), v_\calG(\bsy) - P_h v_{\calG}(\bsy) )\right|\\
&\leq\, \amax(1 + \chi_1) \nrm{u_1(\bsy) - u_{1, h}(\bsy)}{V}
\nrm{v_\calG(\bsy) - P_h v_{\calG}(\bsy) }{V}
\,\leq\, C_4 h^{1 + t}\,.
\end{align*}
In the last step, we have used the upper bound \eqref{eq:u_fe_err} on the 
FE error for the eigenfunction together with
  the dual error bound in \eqref{eq:fe_dual_err}.
\end{namedproof}
\end{appendix}

\bibliographystyle{plain}

\end{document}